\documentclass[12pt,twoside]{amsart}
\usepackage[margin=3cm]{geometry}
\usepackage[colorlinks=false]{hyperref}
\usepackage[english]{babel}
\usepackage{graphicx,titling}
\usepackage{float}
\usepackage{amsmath,amsfonts,amssymb,amsthm}
\usepackage{lipsum}
\usepackage[T1]{fontenc}
\usepackage{fourier}
\usepackage{color}
\usepackage[latin1]{inputenc}
\usepackage{esint}
\usepackage{caption}
\usepackage{ccicons}

\makeatletter
\def\blfootnote{\xdef\@thefnmark{}\@footnotetext}
\makeatother

\newcommand\ccnote{
    \blfootnote{\ccLogo\, \ccAttribution\,\, Licensed under a Creative Commons Attribution License (CC-BY).}
}

\usepackage[export]{adjustbox}
\numberwithin{equation}{section}
\usepackage{setspace}
\renewcommand{\le}{\leqslant}

\renewcommand{\ge}{\geqslant}

\renewcommand{\mathbb}{\varmathbb}
\newcommand{\ud}{\,\mathrm{d}}
\newcommand{\p}{\ensuremath{\partial}}
\newcommand{\n}{\ensuremath{\nonumber}}
\newcommand{\eps}{\ensuremath{\varepsilon}}
\usepackage{fancyhdr}
\newtheorem{theorem}{Theorem}[section]
\newtheorem{lemma}[theorem]{Lemma}
\newtheorem{corollary}[theorem]{Corollary}
\newtheorem{proposition}[theorem]{Proposition}
\newtheorem{definition}[theorem]{Definition}
\newtheorem{remark}[theorem]{Remark}
\address{Sameer Iyer, University of California, Davis, Department of Mathematics, 1 Shields Avenue, Davis, CA 95616.}
\email{sameer@math.ucdavis.edu}
\medskip
\address{Nader Masmoudi,  NYUAD Research Institute, New York University Abu Dhabi, PO Box 129188, Abu Dhabi, United Arab Emirates; Courant Institute of Mathematical Sciences, New York University, 251 Mercer Street, New York, NY 10012, USA.} 
\email{masmoudi@cims.nyu.edu}

\begin{document}

\thispagestyle{empty}

\begin{minipage}{0.28\textwidth}
\begin{figure}[H]
\includegraphics[width=2.5cm,height=2.5cm,left]{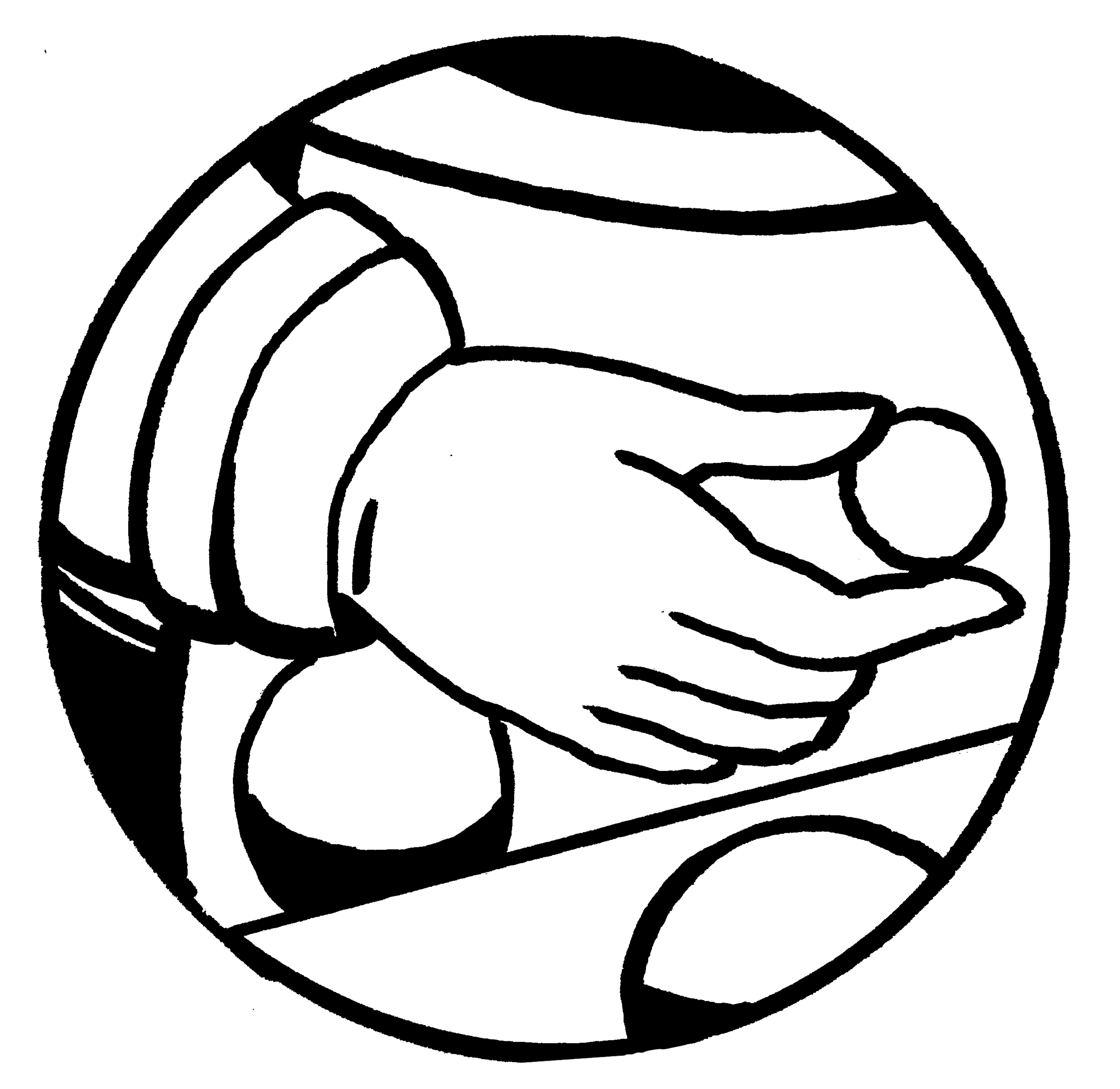}
\end{figure}
\end{minipage}
\begin{minipage}{0.7\textwidth} 
\begin{flushright}
Ars Inveniendi Analytica (2021), Paper No. 6, 47 pp.
\\
DOI 10.15781/64dc-7z92
\end{flushright}
\end{minipage}

\ccnote

\vspace{1cm}


\begin{center}
\begin{huge}
\textit{Boundary Layer Expansions for the Stationary Navier-Stokes Equations}


\end{huge}
\end{center}

\vspace{1cm}


\begin{minipage}[t]{.28\textwidth}
\begin{center}
{\large{\bf{Sameer Iyer}}} \\
\vskip0.15cm
\footnotesize{University of California, Davis}
\end{center}
\end{minipage}
\hfill
\noindent
\begin{minipage}[t]{.38\textwidth}
\begin{center}
{\large{\bf{Nader Masmoudi}}} \\
\vskip0.15cm
\footnotesize{New York University Abu Dhabi}
\\
\footnotesize{$\&\,\,\,$ New York University}
\end{center}
\end{minipage}

\vspace{1cm}


\begin{center}
\noindent \em{Communicated by Vlad Vicol}
\end{center}
\vspace{1cm}


\noindent \textbf{Abstract.} \textit{This is the first part of a two paper sequence in which we prove the global-in-$x$ stability of the classical Prandtl boundary layer for the 2D, stationary Navier-Stokes equations. In this part, we provide a construction of an approximate Navier-Stokes solution, obtained by a classical Euler-Prandtl asymptotic expansion. We develop here sharp decay estimates on these quantities. Of independent interest, we establish \textit{without} using the classical von-Mise change of coordinates, proofs of global in x regularity of the Prandtl system. The results of this paper are used in the second part of this sequence, \cite{IM2} to prove the asymptotic stability of the boundary layer as $\eps \rightarrow 0$ and $x \rightarrow \infty$.}
\vskip0.3cm

\vspace{0.5cm}


\section{Introduction}

We consider the Navier-Stokes (NS) equations posed on the domain $\mathcal{Q} := (0, \infty) \times (0, \infty)$: 
\begin{align} \label{NS:1}
&u^\eps u^\eps_x + v^\eps u^\eps_Y + P^\eps_x = \eps \Delta u^\eps, \\ \label{NS:2}
&u^\eps v^\eps_x + v^\eps v^\eps_Y + P^\eps_Y = \eps \Delta v^\eps, \\ \label{NS:3}
&u^\eps_x + v^\eps_Y = 0,
\end{align}
We are taking the following boundary conditions in the vertical direction
\begin{align} \label{BC:vert:intro}
&[u^\eps, v^\eps]|_{Y = 0} = [0,0], \qquad [u^\eps(x,Y), v^\eps(x,Y)] \xrightarrow{Y \rightarrow \infty} [u_E(x,\infty), v_E(x,\infty)].
\end{align}
which coincide with the classical no-slip boundary condition at $\{Y =0\}$ and the Euler matching condition as $Y \uparrow \infty$. We now fix the vector field 
\begin{align}
[u_E, v_E] := [1, 0], \qquad P_E  = 0
\end{align}
as a solution to the steady, Euler equations ($\eps = 0$ in \eqref{NS:1} - \eqref{NS:3}), upon which the matching condition above reads $[u^\eps, v^\eps] \xrightarrow{Y \rightarrow \infty} [1, 0]$.  

Generically, there is a mismatch at $Y = 0$ between the boundary condition \eqref{BC:vert:intro} and that of an inviscid Eulerian fluid, which typically satisfies the no-penetration condition, $v^E|_{Y = 0} = 0$. Given this, it is not possible to demand convergence of the type $[u^\eps, v^\eps] \rightarrow [1, 0]$ as $\eps \rightarrow \infty$ in suitably strong norms, for instance in the $L^\infty$ sense. To rectify this mismatch, Ludwig Prandtl proposed in his seminal 1904 paper, \cite{Prandtl}, that one needs to \textit{modify} the limit of $[u^\eps, v^\eps]$ by adding a corrector term to $[1, 0]$, which is effectively supported in a thin layer of size $\sqrt{\eps}$ near $\{Y = 0\}$. Mathematically, this amounts to proposing an asymptotic expansion of the type 
\begin{align} \label{ansatz:1:1}
u^\eps(x, Y) = 1 + u^0_p(x, \frac{Y}{\sqrt{\eps}}) + O(\sqrt{\eps}) = \bar{u}^0_p(x, \frac{Y}{\sqrt{\eps}}) + O(\sqrt{\eps}), 
\end{align} 
where the rescaling $\frac{Y}{\sqrt{\eps}}$ ensures that the corrector, $u^0_p$, is supported effectively in a strip of size $\sqrt{\eps}$. The quantity $\bar{u}^0_p$ is classically known as the Prandtl boundary layer, whereas the $O(\sqrt{\eps})$ term will be referred to in our paper as ``the remainder". Motivated by this ansatz, we introduce the Prandtl rescaling 
\begin{align}
y := \frac{Y}{\sqrt{\eps}}. 
\end{align}
We now rescale the solutions via 
\begin{align}
U^\eps(x, y) := u^\eps(x, Y), \qquad V^\eps(x, y) := \frac{v^\eps(x, Y)}{\sqrt{\eps}}
\end{align}
which satisfy the following system 
\begin{align} \label{eq:NS:1}
&U^\eps U^\eps_x + V^\eps U^\eps_y + P^\eps_x = \Delta_\eps U^\eps, \\  \label{eq:NS:2}
&U^\eps V^\eps_x + V^\eps V^\eps_y + \frac{P^\eps_y}{\eps} = \Delta_\eps V^\eps, \\  \label{eq:NS:3}
&U^\eps_x + V^\eps_y = 0. 
\end{align}

\noindent Above, we have denoted the scaled Laplacian operator, $\Delta_\eps := \p_y^2 + \eps \p_x^2$. 

We now expand the rescaled solution in a more detailed manner compared to \eqref{ansatz:1:1}, as 
\begin{align}
\begin{aligned} \label{exp:base}
&U^\eps := 1 + u^0_p + \sum_{i = 1}^{N_{1}} \eps^{\frac{i}{2}} (u^i_E + u^i_P) + \eps^{\frac{N_2}{2}} u =: \bar{u} +  \eps^{\frac{N_2}{2}} u \\
&V^\eps := v^0_p + v^1_E + \sum_{i = 1}^{N_{1}-1} \eps^{\frac i 2} (v^i_P + v^{i+1}_E) + \eps^{\frac{N_{1}}{2}} v^{N_{1}}_p +  \eps^{\frac{N_2}{2}} v =: \bar{v} + \eps^{\frac{N_2}{2}} v, \\
&P^\eps := \sum_{i = 0}^{N_1+1} \eps^{\frac{i}{2}} P^i_p +  \sum_{i = 1}^{N_{1}} \eps^{\frac{i}{2}} P^i_E + \eps^{\frac{N_2}{2}} P.
\end{aligned}
\end{align}

\noindent Above, 
\begin{align}
[u^0_E, v^0_E] := [1, 0], \qquad [u^i_p, v^i_p] = [u^i_p(x, y), v^i_p(x, y)], \qquad [u^i_E, v^i_E] = [u^i_E(x, Y), v^i_E(x, Y)], 
\end{align}
and the expansion parameters $N_{1}, N_2$ will be specified in Theorem \ref{thm:approx} for the sake of precision. We note that these are not optimal choices of these parameters, but we chose them large for simplicity. Certainly, it will be possible to bring these numbers significantly smaller.  

Motivated by the leading order of the right-hand side of the first two equations of \eqref{exp:base}, we make the following definitions:
\begin{align}
\bar{u}^0_p := 1 + u^0_p, \qquad \bar{v}^0_p := - \int_0^y \p_x \bar{u}^0_p
\end{align}
The divergence-free vector field $[\bar{u}^0_p, \bar{v}^0_p]$ solve the classical stationary Prandtl system, which we will now introduce. On the other hand, the divergence-free vector field $[u^0_p, v^0_p]$ are homogenized so as to decay as $y \rightarrow \infty$. The effectiveness of the ansatz \eqref{exp:base}, and the crux of Prandtl's revolutionary idea, is that the leading order term $\bar{u}^0_p$ (and its divergence-free counterpart, $\bar{v}^0_p$) satisfy a much simpler equation than the full Navier-Stokes system, known as the Prandtl system, 
\begin{align}  \label{BL:0:intro:intro}
&\bar{u}^0_p \p_x \bar{u}^0_p + \bar{v}^0_p \p_y \bar{u}^0_p - \p_y^{2} \bar{u}^0_p + P^{0}_{px} = 0, \qquad P^0_{py} = 0, \qquad \p_x \bar{u}^0_p + \p_y \bar{v}^0_p = 0,  
\end{align}
which are supplemented with the boundary conditions
\begin{align} \label{BL:1:intro:intro}
&\bar{u}^0_p|_{x = 0} = \bar{U}^0_p(y), \qquad \bar{u}^0_p|_{y = 0} = 0, \qquad \bar{u}^0_p|_{y = \infty} = 1, \qquad \bar{v}^0_p = - \int_0^y \p_x \bar{u}^0_p. 
\end{align}
This system is simpler than \eqref{eq:NS:1} - \eqref{eq:NS:3} in several senses. First, due to the condition $P^0_{py} = 0$, we obtain that the pressure is constant in $y$ (and then in $x$ due to the Bernoulli's equation), and hence \eqref{BL:0:intro:intro} is really a scalar equation. 

In addition to this, by temporarily omitting the transport term $\bar{v}^0_p \p_y \bar{u}^0_p$, one can make the formal identification that $\bar{u}^0_p \p_x \approx \p_{yy}$, which indicates that \eqref{BL:1:intro:intro} is really a \textit{degenerate, parabolic} equation, which is in stark contrast to the elliptic system \eqref{eq:NS:1} - \eqref{eq:NS:3}. From this perspective, $x$ acts as a time-like variable, whereas $y$ acts as a space-like variable. We thus treat \eqref{BL:0:intro:intro} - \eqref{BL:1:intro:intro} as one would a typical Cauchy problem. Indeed, one can ask questions of local (in $x$) wellposedness, global (in $x$) wellposedness, finite-$x$ singularity formation, decay and asymptotics, etc,... This perspective will be emphasized more in Section \ref{subsection:asy:x}.

The main purposes of this work is to

\begin{itemize}
\item[(1)] Provide a rigorous construction of each of the terms in the approximate solution, $[\bar{u}, \bar{v}]$ from \eqref{exp:base} with precise estimates on the error created by such an approximation: 
\begin{align}
\begin{aligned} \label{forcing:remainder}
F_R := &(U^\eps U^\eps_x + V^\eps U^\eps_y + P^\eps_x - \Delta_\eps U^\eps ) - (\bar{u} \bar{u}_x + \bar{v} \bar{u}_y + \bar{P}_x - \Delta_\eps \bar{u}) \\
G_R := &(U^\eps V^\eps_x + V^\eps V^\eps_y + \frac{P^\eps_y}{\eps} - \Delta_\eps V^\eps ) -  (\bar{u} \bar{v}_x + \bar{v} \bar{v}_y + \frac{\bar{P}_y}{\eps} - \Delta_\eps \bar{v}). 
\end{aligned}
\end{align}
\item[(2)] Provide a proof of global-in-$x$ regularity and decay for the nonlinear and linearized Prandtl equations in physical $(x, y)$ coordinates and in particular \textit{does not rely on the von-Mise change of coordinates}.  
\end{itemize}

The role that $x$ plays as a ``time" variable will be discussed from a mathematical standpoint in Section \ref{subsection:asy:x}. The present paper can be regarded as a global-in-$x$ version of the constructions obtained in \cite{GI2}, all of which required $x << 1$ (local in $x$). Another purpose of this article is to provide \textit{in physical $(x, y)$ coordinates}, a proof of global regularity of the nonlinear and linearized Prandtl equations, as compared to \cite{Serrin}, \cite{IyerBlasius}, both of whom relied on the ``von-Mise" transform, which has little hope to be useful in the Navier-Stokes setting. 

We draw a parallel to the unsteady setting in which  \cite{MW} provided an energy-estimate based proof of local in time existence to the Prandtl equations, \textit{without} using the Crocco-transform. The $(x, y)$ coordinate system (and energy estimates in this coordinate system) have proven to translate most effectively to the Navier-Stokes setting than these classical transforms (von-Mise transform, Crocco transform) that are particular to the simplified Prandtl setting. Indeed, the construction of the approximate solution, $[\bar{u}, \bar{v}]$, that we provide in this article is used in the companion paper, \cite{IM2}, in which we prove the stability of $[\bar{u}, \bar{v}]$.

\subsection{Asymptotic Expansion and Datum}

First, we note that, given the expansion \eqref{exp:base}, we enforce the vertical boundary conditions for $i = 0,...N_1$, $j = 1, ... N_1$,
\begin{align}
&u^i_p(x, 0) = - u^i_E(x, 0),  & &v^j_E(x, 0) = - v^{j-1}_p(x, 0), & &[u, v]|_{y = 0} = 0, \\
&v^i_p(x, \infty) = u^i_p(x, \infty) = 0, & & u^j_E(x, \infty) = v^j_E(x, \infty) = 0, & &[u, v]|_{y = \infty} = 0. 
\end{align}
which ensure the no-slip boundary condition for $[U^\eps, V^\eps]$ at $y = 0, \infty$. As we do not include an Euler field to cancel out last $v^{N_1}_p$, we need to also enforce the condition $v^{N_1}_p|_{y = 0} = 0$, which is a particularity for the $N_1$'th boundary layer. 

The side $\{x =0\}$ has a distinguished role in this setup as being where ``in-flow" datum is prescribed. This datum is prescribed ``at the level of the expansion", in the sense described below (see the discussion surrounding equations \eqref{datum:given}). Moreover, as $x \rightarrow \infty$, one expects the persistence of just one quantity from \eqref{exp:base}, which is the leading order boundary layer, $[\bar{u}^0_p, \bar{v}^0_p]$, and the remaining terms from \eqref{exp:base} are expected to decay in $x$. This decay will be established rigorously in our main result. 

We shall now describe the datum that we take for our setting.  First, we take the outer Euler profile to be 
\begin{align} \label{choice:Euler}
[u^0_E, v^0_E, P^0_E] := [1, 0, 0]
\end{align}
as given. We have selected \eqref{choice:Euler} as the simplest shear flow with which to work. Our analysis can be extended with relatively small (but cumbersome) modifications to general Euler shear flows of the form 
\begin{align}
[u^0_E, v^0_E, P^0_E] := [u^0_E(Y), 0, 0],
\end{align}
under mild assumptions on the shear profile $u^0_E(Y)$. 

Apart from prescribing the outer Euler profile, we also get to pick ``Initial datum", that is at $\{x = 0\}$ of various terms in the expansion \eqref{exp:base}. Specifically, the prescribed initial datum comes in the form of the functions: 
\begin{align} \label{datum:given}
u^i_P|_{x = 0} =: U^i_P(y), \qquad v^j_E|_{x = 0} =: V^j_E(Y),
\end{align}
for $i = 0,...,N_1$, and $j = 1,..,N_1$. Note that we do not get to prescribe, at $\{x = 0\}$, all of the terms appearing in \eqref{exp:base}. On the one hand, to construct $[u^i_p, v^i_p]$, we use that   $u^i_p$ obeys a degenerate parabolic equation, with $x$ occupying the time-like variable and $y$ the space-like variable and that $v^i_p$ can be  recovered from $u^i_p$ via the divergence-free condition. Therefore, only $u^i_p|_{x = 0}$ is necessary to determine these quantities.  

It is useful at this point to introduce the systems satisfied by the ``leading order" boundary layer ($[\bar{u}^0_p, \bar{v}^0_p]$) and the ``intermediate" boundary layers, $[u^i_p, v^i_p]$, for $i = 1, \dots, N_1$. We note that these systems are derived rigorously in Appendix \ref{app:A} by substituting the expansion \eqref{exp:base} into the system \eqref{eq:NS:1} -- \eqref{eq:NS:3} and systematically collecting terms contributing to each order in $\sqrt{\eps}$. 

First, as we have discussed in \eqref{BL:0:intro:intro} -- \eqref{BL:1:intro:intro}, the vector-field $[\bar{u}^0_p, \bar{v}^0_p]$ satisfies the pressure-free Prandtl equations, which read: 
\begin{subequations}
\begin{align}
&\bar{u}^0_p \p_x \bar{u}^0_p + \bar{v}^0_p \p_y \bar{u}^0_p - \p_y^2 \bar{u}^0_p = 0, \qquad (x, y) \in (0, \infty) \times (0, \infty) \\
&\bar{v}^0_p = - \int_0^y \p_x \bar{u}^0_p, \\
&\bar{u}^0_p|_{x = 0} = \bar{U}^0_p(y) := 1 + U^0_P(y), \qquad \bar{u}^0_p|_{y = 0} = 0, \qquad \bar{u}^0_p|_{y = \infty} = 1. 
\end{align}
\end{subequations}
In the main theorem below, Theorem \ref{thm:approx}, we will state a smallness condition on the perturbation of the initial datum from a particular self-similar profile, the Blasius profile $\bar{u}_\ast$, ( introduced below in \eqref{Blasius:1} -- \eqref{Blasius:3}). Indeed, the object we will ultimately analyze here will be the perturbation: 
\begin{align}
\tilde{u}^0_p := \frac{\bar{u}^0_p - \bar{u}_\ast}{\delta_{\ast}}, \qquad \tilde{v}^0_p := \frac{\bar{v}^0_p - \bar{v}_\ast}{\delta_{\ast}}. 
\end{align}
Above, $0 < \delta_{\ast} << 1$ is a small parameter introduced by hypothesis \eqref{near:blasius} in the main result below. The quantities above satisfy the equation 
\begin{subequations}
\begin{align} \label{Pr:leading:1:1:1}
&\bar{u}_\ast \p_x \tilde{u}^0_p +\tilde{u}^0_p \p_x \bar{u}_{\ast}  + \bar{v}_\ast \p_y \tilde{u}^0_p + \tilde{v}^0_p \p_y \bar{u}_\ast - \p_y^2 \tilde{u}^0_p = \delta_{\ast} Q(\tilde{u}^0_p, \tilde{v}^0_p), \\ \label{Pr:leading:1:1:2}
&\tilde{v}^0_p = - \int_0^y \p_x \tilde{u}^0_p, \\ \label{Pr:leading:1:1:3}
&\tilde{u}^0_p|_{x = 0} = \frac{1}{\delta_{\ast}} (\bar{U}^0_P(y) - \bar{u}_\ast(0, y)), \qquad \tilde{u}^0_p|_{y = 0} = 0, \qquad \tilde{u}^0_p|_{y = \infty} = 0. 
\end{align}
 \end{subequations}
 Above, the forcing $Q(\tilde{u}^0_p, \tilde{v}^0_p)$ contain quadratic terms, and will be perturbative using the smallness of $\delta_\ast << 1$ in the main theorem below. 
 
The intermediate boundary layers, $[u^i_p, v^i_p]$, $i = 1, \dots N_1$, satisfy the linearized Prandtl system (with contributed forcing):
\begin{subequations}
\begin{align} \label{eq:lin:Pr:HYU}
&\bar{u}^0_p \p_x u^i_p + u^i_p  \bar{u}^0_{px} + \bar{v}^0_p \p_y u^i_p + \bar{v}^i_p \p_y \bar{u}^0_p - \p_y^2 u^i_p = F^{(i)}_p, \\ \label{abs:intro:2}
&v^i_p =  \int_y^\infty \p_x u^i_p, \qquad \bar{v}^i_p = v^i_p - v^i_p|_{y = 0}, \\ \label{abs:intro:3}
&u^i_p|_{x = 0} = U^i_P(y), \qquad u^i_p|_{y = 0} = - u^i_E|_{y = 0}, \qquad u^i_p|_{y = \infty} = 0. 
\end{align}
\end{subequations}
 We note that, by our convention, $u^i_p$ are the boundary layer \textit{correctors}. This means that they vanish as $y \rightarrow \infty$, and cancel out the $i$'th Euler contribution at $\{y = 0\}$. Above, the forcing term $F^{(i)}_p$ is due to contributions of size $O(\eps^{\frac i2})$ of lower order boundary layers ($[u^j_p, v^j_p]$,  $j =0, \dots i - 1$) which need to be cancelled out through the introduction of $[u^i_p, v^i_p]$.

On the other hand,  to construct  the Euler profiles $[u^i_E, v^i_E]$ for $i = 1,..,N_1$, we use an  elliptic problem for $v^i_E$ (in the special case of \eqref{choice:Euler}, it is in fact $\Delta v^i_E = 0$). As such, we prescribe the datum for $v^i_E$, as is displayed in \eqref{datum:given}, and then recover $u^i_E$ via the divergence-free condition. Therefore, only $v^i_E|_{x = 0}$ is necessary to determine these quantities.

\subsection{Asymptotics as $x \rightarrow \infty$} \label{subsection:asy:x}

We will now discuss more precisely the role of the $x$-variable, specifically emphasizing the role that $x$ plays as a ``time-like" variable, controlling the ``evolution" of the fluid. The importance of studying the large $x$ behavior of both the Prandtl equations and the Navier-Stokes equations is not just mathematical (in analogy with proving global wellposedness/ decay versus finite-$x$ blowup), but is also of importance physically due to the possibility of boundary layer separation, which is a large $x$ phenomena (which was noted by Prandtl himself in his original 1904 paper). 

We shall discuss first the large-$x$ asymptotics at the level of the Prandtl equations, \eqref{BL:0:intro:intro} - \eqref{BL:1:intro:intro}, which govern $[\bar{u}^0_p, \bar{v}^0_p]$. It turns out that there are two large-$x$ regimes for $[\bar{u}^0_p, \bar{v}^0_p]$, depending on the sign of the Euler pressure gradient: 

\begin{itemize}
\item[(1)] Favorable pressure gradient, $\p_x P^E \le 0$: $[\bar{u}^0_p, \bar{v}^0_p]$ exists globally in $x$, and becomes asymptotically self-similar, 

\item[(2)] Unfavorable pressure gradient, $\p_x P^E > 0$, $[\bar{u}^0_p, \bar{v}^0_p]$ may form a finite-$x$ singularity, known as ``separation".
\end{itemize}
In this work, our choice of $[1, 0]$ for the outer Euler flow guarantees that we are in setting (1), that of a favorable pressure gradient. 

This dichotomy above was introduced by Oleinik, \cite{Oleinik}, \cite{Oleinik1}, who established the first mathematically rigorous results on the Cauchy problem \eqref{BL:0:intro:intro} - \eqref{BL:1:intro:intro}. Indeed, Oleinik established that solutions to  \eqref{BL:0:intro:intro} - \eqref{BL:1:intro:intro} are locally (in $x$) well-posed in both regimes (1) and (2), and globally well-posed in regime (1) (under suitable hypothesis on the datum, which we do not delve into at this stage). 

Now, we investigate what  it means for $[\bar{u}^0_p, \bar{v}^0_p]$ to become asymptotically self-similar.  In order to describe this behavior more quantitatively, we need to introduce the Blasius solutions. Four years after Prandtl's seminal 1904 paper, H. Blasius introduced the (by now) famous ``Blasius boundary layer" in \cite{Blasius}, which takes the following form 
\begin{align} \label{Blasius:1}
&[\bar{u}_\ast^{x_0}, \bar{v}_\ast^{x_0}] = [f'(z), \frac{1}{\sqrt{x + x_0}} (z f'(z) - f(z)) ], \\ \label{Blasius:2}
&z := \frac{y}{\sqrt{x + x_0}} \\ \label{Blasius:3}
&ff'' + f''' = 0, \qquad f'(0) = 0, \qquad f'(\infty) = 1, \qquad \frac{f(z)}{z} \xrightarrow{\eta \rightarrow \infty} 1, 
\end{align}
where above, $f' = \p_z f(z)$ and  $x_0$ is a free parameter. Physically, $x_0$ has the meaning that at $x = - x_0$, the fluid interacts with the ``leading edge" of, say, a plate (hence the singularity at $x = - x_0$). Our analysis will treat any fixed $x_0 > 0$ (one can think, without loss of generality, that $x_0 = 1$). In fact, we will make the following notational convention which enables us to omit rewriting $x_0$ repeatedly:
\begin{align}
[\bar{u}_\ast, \bar{v}_\ast] := [\bar{u}_\ast^{1}, \bar{v}_\ast^{1}]. 
\end{align}
We emphasize that the choice of $1$ above could be replaced with any positive number, without loss of generality. 

The Blasius solutions, $[\bar{u}_\ast^{x_0}, \bar{v}_\ast^{x_0}]$ are distinguished solutions to the Prandtl equations in several senses. First, physically, they have demonstrated remarkable agreement with experiment (see \cite{Schlicting} for instance). Mathematically, their importance is two-fold. First, they are self-similar, and second, they act as large-$x$ attractors for the Prandtl dynamic. Indeed, a classical result of Serrin, \cite{Serrin}, states: 
\begin{align}
\lim_{x \rightarrow \infty} \| \bar{u}^0_p - \bar{u}_\ast \|_{L^\infty_y} \rightarrow 0 
\end{align}
for a general class of solutions, $[\bar{u}^0_p, \bar{v}^0_p]$ of \eqref{BL:0:intro:intro}.

This was revisited by the first author in the work \cite{IyerBlasius}, who established a refined description of the above asymptotics, in the sense 
\begin{align} \label{asy:blas:1}
\| \bar{u}^0_p - \bar{u}_\ast \|_{L^\infty_y} \lesssim o(1) \langle x \rangle^{- \frac 1 2 + \sigma_\ast}, \text{ for any } 0 < \sigma_\ast << 1,  
\end{align}
which is the essentially optimal decay rate from the point of view of regarding $\bar{u}^0_p$ as a parabolic equation with one spatial dimension. Both the work of \cite{IyerBlasius} and the work of \cite{Serrin} relied crucially on the ``von-Mise" change of variables, which interacts poorly with the Navier-Stokes equations. The purpose of the present article is present an alternative proof of these results, using only energy methods in the physical $(x, y)$ coordinate system.   

The case of (2) above (the setting of unfavorable pressure gradient) has been treated in the work of \cite{MD} as well as in the paper of \cite{Zhangsep} for the Prandtl equation with $\p_x P^E > 0$ (which appears as a forcing term on the right-hand side of \eqref{BL:0:intro:intro} - \eqref{BL:1:intro:intro} in their setting). These results establish the physically important phenomenon of separation, which occurs when $\p_y \bar{u}^0_p(x, 0) = 0$ for some $x > 0$, even though the datum starts out with the monotonicity $\p_y \bar{u}^0_p(0, 0) > 0$.

\subsection{Main Theorem}

The theorem we prove here is 
\begin{theorem}[Construction of Approximate Solution] \label{thm:approx}  Fix $N_{1} = 400$ and $N_2 = 200$ in \eqref{exp:base}. Fix the leading order Euler flow to be 
\begin{align}
[u^0_E, v^0_E, P^0_E] := [1, 0, 0].
\end{align}
Assume the following pieces of initial data at $\{x = 0\}$ are given for $i = 0,...,N_{1}$, and $j = 1,...N_1$,
\begin{align} \label{datum:in}
 u^i_p|_{x = 0} =: U^i_p(y), && v^j_E|_{x = 0} =: V_E^j(Y)
\end{align}
where we make the following assumptions on the initial datum \eqref{datum:in}: 
\begin{itemize}
\item[(1)] For $i = 0$, the boundary layer datum $\bar{U}^0_p(y)$ is in a neighborhood of  Blasius, defined in \eqref{Blasius:1}. More precisely, we will assume 
\begin{align} \label{near:blasius}
\| (\bar{U}^0_p(y) - \bar{u}_\ast(0, y) ) \langle y \rangle^{m_0} \|_{C^{\ell_0}} \le \delta_\ast, 
\end{align}
where $0 < \delta_\ast << 1$ is small relative to universal constants, where $m_0, \ell_0$, are large, explicitly computable numbers. Assume also the difference $\bar{U}^0_p(y) - \bar{u}_\ast(0, y)$ satisfies generic parabolic compatibility conditions at $y = 0$ as specified in Definition \ref{para:compat:def}.

\item[(2)] For $i = 1,..,N_1$, the boundary layer datum, $U^i_p(\cdot)$ is sufficiently smooth and decays rapidly:
\begin{align} \label{hyp:1}
\| U^i_p \langle y \rangle^{m_i} \|_{C^{\ell_i}} \lesssim 1,
\end{align}
where $m_i, \ell_i$ are large, explicitly computable constants (for instance, we can take $m_0 = 10,000$, $\ell_0 = 10,000$ and $m_{i+1} = m_i - 5$, $\ell_{i + 1} = \ell_i - 5$), and satisfies generic parabolic compatibility conditions at $y = 0$ (described specifically in Definition \ref{para:compat:def}).

\item[(3)] The Euler datum $V^i_E(Y)$ satisfies generic elliptic compatibility conditions, which we define specifically in Definition \ref{defn:compatible:harmonic}. 

\end{itemize}

Define $\sigma_\ast = \frac{1}{10,000}$. Then for $i = 1,...,N_1$, for $M \le m_i$ and $2k+j \le \ell_i$, the quantities $[u^i_p, v^i_p]$ and $[u^i_E, v^i_E]$ exist globally, $x > 0$, and the following estimates are valid
\begin{align} \label{blas:conv:1}
&\| \p_x^k \p_y^j ( \bar{u}^0_p - \bar{u}_\ast) \langle z \rangle^M \|_{L^\infty_y} \lesssim \delta_\ast \langle x \rangle^{- \frac 1 4 - k - \frac j 2 + \sigma_\ast}, \\ \label{blas:conv:2}
&\| \p_x^k \p_y^j ( \bar{v}^0_p - \bar{v}_\ast) \langle z \rangle^M \|_{L^\infty_y} \lesssim \delta_\ast \langle x \rangle^{- \frac 3 4 - k - \frac j 2 + \sigma_\ast} \\ \label{water:65}
&\| \langle z \rangle^M \p_x^k \p_y^j u_p^{i} \|_{L^\infty_y} \lesssim \langle x \rangle^{- \frac 1 4 - k - \frac j 2 + \sigma_\ast} 
\end{align}
\begin{align}
\label{water:54}
&\| \langle z \rangle^M \p_x^k \p_y^j v_p^{i} \|_{L^\infty_y} \lesssim \langle x \rangle^{- \frac 3 4 - k - \frac j 2 + \sigma_\ast}, \\  \label{water:78}
&\|(Y \p_Y)^l \p_x^k \p_Y^j v^{i}_E \|_{L^\infty_Y} \lesssim \langle x \rangle^{- \frac 1 2 - k - j}, \\ \label{water:88}
&\| (Y \p_Y)^l \p_x^k \p_Y^j u^{i}_E \|_{L^\infty_Y} \lesssim \langle x \rangle^{- \frac 1 2 - k - j}.
\end{align}
The following estimates hold on the contributed forcing: 
\begin{align} \label{est:forcings:part1}
\| \p_x^j \p_y^k F_R \langle x \rangle^{\frac{11}{20}+ j + \frac k 2} \| + \sqrt{\eps} \|  \p_x^j \p_y^k G_R \langle x \rangle^{\frac{11}{20}+ j + \frac k 2} \| \le \eps^{5}. 
\end{align}
\end{theorem}

\begin{remark} Although we do not state it as part of our main construction, a consequence to the our construction will be Lemmas \ref{lemma:cons:1} -- \ref{lemma:cons:3}, which are ``aggregated estimates", that is, estimates on the full approximate solution $[\bar{u}, \bar{v}]$ which follow from those stated above in Theorem \ref{thm:approx} and by taking summation according to \eqref{exp:base}.
\end{remark}

\subsection{Existing Literature} \label{existing}

The boundary layer theory originated with Prandtl's seminal 1904 paper, \cite{Prandtl}. First, we would like to emphasize that this paper presented the boundary layer theory in precisely the present setting: for 2D, steady flows over a plate (at $Y = 0$). In addition, Prandtl's original paper discussed the physical importance of understanding \eqref{exp:base} for large $x$, due to the possibility of boundary layer separation.

We will distinguish between two types of questions that are motivated by the ansatz, \eqref{ansatz:1:1}. First, there are questions regarding the description of the leading order boundary layer, $[\bar{u}^0_p, \bar{v}^0_p]$, and second, there are questions regarding the study of the $O(\sqrt{\eps})$ remainder, which, equivalently, amounts to questions regarding the validity of the asymptotic expansion \eqref{ansatz:1:1}.

A large part of the results surrounding the system \eqref{BL:0:intro:intro} - \eqref{BL:1:intro:intro} were already discussed in Section \ref{subsection:asy:x}, although the results discussed there were more concerned with the large $x$ asymptotic  behavior. We point the reader towards \cite{MD} for a study of separation in the steady setting, using modulation and blowup techniques. For local-in-$x$ behavior, let us supplement the references from Section \ref{subsection:asy:x} with the results of \cite{GI2}, which established higher regularity for \eqref{BL:0:intro:intro} - \eqref{BL:1:intro:intro} through energy methods, and the recent work of \cite{Zhifei:smooth} which obtains global $C^\infty$ regularity using maximum principle methods. One can think of the present article as a global version of the article \cite{GI2}.

We now discuss the validity of the ansatz (\ref{ansatz:1:1}). The classical setting we consider here, with the no-slip condition, was first treated, locally in the $x$ variable by the works \cite{GI1} - \cite{GI2}, \cite{Varet-Maekawa}, and the related work of \cite{GI3}. These works of \cite{GI1} - \cite{GI2} are distinct from that of \cite{Varet-Maekawa} in the sense that the main concern of \cite{GI1} - \cite{GI2} are $x$-dependent boundary layer profiles, and in particular addresses the classical Blasius solution. On the other hand, the work of \cite{Varet-Maekawa} is mainly concerned with shear solutions $(U(y), 0)$ to the forced Prandtl equations (shears flows are not solutions to the homogeneous Prandtl equations), which allows for Fourier analysis in the $x$ variable. Both of these works are \textit{local-in-$x$} results, which can demonstrate the validity of expansion \eqref{exp:base} for $0 \le x \le L$, where $L << 1$ is small (but of course, fixed relative to the viscosity $\eps$). We also mention the relatively recent work of \cite{Gao-Zhang}, which has generalized the work of \cite{GI1} - \cite{GI2} to the case of Euler flows which are not shear for $0 < x < L << 1$. 

We also point the reader towards the works \cite{GN}, \cite{Iyer}, \cite{Iyer2a} - \cite{Iyer2c}, and \cite{Iyer3}. All of these works are under the assumption of a moving boundary at $\{Y = 0\}$, which while they face the difficulty of having a transition from $Y = 0$ to $Y = \infty$, crucially eliminate the degeneracy of $\bar{u}^0_p$ at $\{Y =0\}$, which is a major difficulty posed by the boundary layer theory. The work of \cite{Iyer2a} - \cite{Iyer2c} is of relevance to this paper, as the question of global in $x$ stability was considered (again, under the assumption of a moving $\{Y = 0\}$ boundary, which significantly simplifies matters). 

For unsteady flows, there is also a  large literature studying  expansions of the type \eqref{exp:base}. We refrain from discussing this at too much length because the unsteady setting is quite different from the steady setting. Rather, we point the reader to (an incomplete) list of references. First, in the analyticity setting, for small time, the seminal works of \cite{Caflisch1}, \cite{Caflisch2} establish the stability of expansions \eqref{exp:base}. This was extended to the Gevrey setting in \cite{DMM}, \cite{DMM20}. The work of \cite{Mae} establishes stability under the assumption of the initial vorticity being supported away from the boundary.  The reader should also see the related works \cite{Asano}, \cite{LXY}, \cite{Taylor}, \cite{TWang}, \cite{Wang}. 

When the regularity setting goes from analytic/ Gevrey to Sobolev, there have also been several works in the opposite direction, which demonstrate, again in the unsteady setting, that expansion of the type \eqref{exp:base} should not be expected. A few works in this direction are \cite{Grenier}, \cite{GGN1}, \cite{GGN2}, \cite{GGN3}, \cite{GN2}, as well as the remarkable series of recent works of \cite{GrNg1}, \cite{GrNg2}, \cite{GrNg3} which settle the question and establish invalidity in Sobolev spaces of expansions of the type (\ref{exp:base}). The related question of $L^2$ (in space) convergence of Navier-Stokes flows to Euler has been investigated by many authors, for instance in \cite{CEIV}, \cite{CKV}, \cite{CV}, \cite{Kato}, \cite{Masmoudi98}, and \cite{Sueur}. 

There is again the related matter of wellposedness of the unsteady Prandtl equation. This investigation was initiated by \cite{Oleinik}, who obtained global in time solutions on $[0, L] \times \mathbb{R}_+$ for $L << 1$ and local in time solutions for any $L < \infty$, under the crucial monotonicity assumption $\p_y u|_{t = 0} > 0$. The $L << 1$ was removed by \cite{Xin} who obtained global in time weak solutions for any $L < \infty$. These works relied upon the Crocco transform, which is only available under the monotonicity condition. Also under the monotonicity condition, but without using the Crocco transform, \cite{AL} and \cite{MW} obtained local in time existence. \cite{AL} introduced a Nash-Moser type iterative scheme, whereas \cite{MW} introduced a good unknown which enjoys an extra cancellation and obeys good energy estimates. The related work of \cite{KMVW} removes monotonicity and replaces it with multiple monotonicity regions.  

Without monotonicity of the datum, the wellposedness results are largely in the analytic or Gevrey setting. Indeed, \cite{GVDie},  \cite{GVM}, \cite{Vicol}, \cite{Kuka}, \cite{Lom}, \cite{LMY}, \cite{Caflisch1} - \cite{Caflisch2}, \cite{IyerVicol} are some results in this direction. Without assuming monotonicity, in Sobolev spaces, the unsteady Prandtl equations are, in general, illposed: \cite{GVD}, \cite{GVN}. Finite time blowup results have also been obtained in \cite{EE}, \cite{KVW}, \cite{Hunter}. Moreover, the issue of boundary layer separation in the unsteady setting has been tackled by the series of works \cite{Collot1}, \cite{Collot2}, \cite{Collot3} using modulation and blowup techniques.  

The above discussion is not comprehensive, and we have elected to provide a more in-depth of the steady theory due to its relevance to the present paper. We refer to the review articles, \cite{E}, \cite{Temam} and references therein for a more complete review of other aspects of the boundary layer theory. 

\subsection{Notational Conventions}

We first define (in contrast with the typical bracket notation) $\langle x \rangle := 1+ x$. We also define the quantity 
\begin{align} \label{z:choice}
z := \frac{y}{\sqrt{x + 1}} = \frac{y}{\sqrt{\langle x \rangle}},
\end{align}
due to our choice that $x_0 = 1$ (which we are again making without loss of generality). The cut-off function $\chi(\cdot): \mathbb{R}_+ \rightarrow \mathbb{R}$ will be reserved for a particular decreasing function, $0 \le \chi \le 1$, satisfying 
\begin{align} \label{def:chi}
\chi(z) = \begin{cases} 1 \text{ for } 0 \le z \le 1 \\ 0 \text{ for } 2 \le z < \infty  \end{cases}
\end{align}
Regarding norms, we define for functions $u(x, y)$, 
\begin{align}
\| u \| := \|u \|_{L^2_{xy}} = \Big( \int u^2 \ud x \ud y \Big)^{\frac 1 2}, \qquad \|u \|_\infty := \sup_{(x,y) \in \mathcal{Q}} |u(x, y)|. 
\end{align}
We will often need to consider ``slices", whose norms we denote in the following manner
\begin{align}
\| u \|_{L^p_y} := \Big( \int u(x, y)^p \ud y \Big)^{\frac 1 p}.
\end{align}
We use the notation $a \lesssim b$ to mean $a \le Cb$ for a constant $C$, which is independent of the parameters $\eps, \delta_\ast$. We define the following scaled differential operators 
\begin{align}
\nabla_\eps := \begin{pmatrix} \p_x \\ \frac{\p_y}{\sqrt{\eps}} \end{pmatrix}, \qquad \Delta_\eps := \p_{yy} + \eps \p_{xx}.
\end{align}
For derivatives, we will use both $\p_x f$ and $f_x$ to mean the same thing. For integrals, we use the notation $\int f := \int_0^\infty f(y) \ud y$, that is, it is a one-dimensional integral, and moreover, we will omit repeating $\ud y$. These conventions are taken unless otherwise specified (by appending a $\ud y$ or a $\ud x$), which we sometimes need to do. We will often use the parameter $\delta$ to be a generic small parameter, that can change in various instances. The constant $C_\delta$ will refer to a number that may grow to $\infty$ as $\delta \downarrow 0$. 

 \subsection{Overview of Strategy}

The first step we take is to insert the proposed expansion of our approximate solution, \eqref{exp:base}, into the rescaled Navier-Stokes equations \eqref{eq:NS:1} -- \eqref{eq:NS:3}. The multiscale nature of the alternating Euler-Prandtl terms in \eqref{exp:base} requires us to carefully group terms based on both (1) order of $\eps^{\frac 1 2}$ and (2) the scale on which they vary $y = \frac{Y}{\eps^{\frac 12}}$ versus $Y$. This produces a sequence of parabolic systems for the boundary layer terms, \eqref{BL:0} -- \eqref{BL:1} for the nonlinear Prandtl layer $(i = 0)$, \eqref{rustic:1} -- \eqref{rustic:3} for the linearized Prandtl layers ($i = 1, \dots, N_1)$. This procedure also produces elliptic equations for each of the Euler terms, \eqref{Eul:harm} (in fact, these are harmonic). Our analysis is concerned with analyzing the boundary layer equations, \eqref{BL:0} -- \eqref{BL:1} and \eqref{rustic:1} -- \eqref{rustic:3}. 

On one hand, one observes from \eqref{rustic:1} -- \eqref{rustic:3} that the system governing all of the boundary layers (including the nonlinear Prandtl layer after subtracting out the Blasius flow) is the linearized system displayed below, which we write in abstract form:
\begin{subequations} 
\begin{align} \label{intro:abs:AI:a}
&\bar{u}_B \p_x u + u \p_x \bar{u}_B + \bar{v}_B \p_y u + \bar{v} \p_y \bar{u}_B - \p_y^2 u = F, \\ \label{intro:abs:AI:b}
&\bar{v} := - \int_0^y \p_x u, \\ \label{intro:abs:AI:c}
&u|_{x = 0} = u_0(y), \qquad u|_{y = 0} = g(x), \qquad u|_{y = \infty} = 0.  
\end{align}
\end{subequations}
Above, $[\bar{u}_B, \bar{v}_B]$ stands for a \textit{background} boundary layer profile. In almost all cases, we will take $[\bar{u}_B, \bar{v}_B] = [\bar{u}^0_p, \bar{v}^0_p]$. However, in a few cases, we also want to be able to take just the Blasius flow $[\bar{u}_B, \bar{v}_B] = [\bar{u}_{\ast}, \bar{v}_{\ast}]$, so we work in this abstract framework to unify these treatments.

On the other hand, one observes that the precise form of the forcing $F$ and the boundary condition $g(x)$ in the abstract problem \eqref{intro:abs:AI:a} -- \eqref{intro:abs:AI:c} (which play the role of $F_p^{(i)}$ in \eqref{rustic:1}, which itself is defined in \eqref{Fpi}, and $u_E(x)$ in \eqref{rustic:3}) depends in a delicate manner on the lower order terms in the expansion \eqref{exp:base} for $i = 1, \dots, N_1$. Therefore, the core of our strategy involves propagating an induction which establishes the bounds \eqref{blas:conv:1} -- \eqref{water:88} iteratively: we prove \eqref{blas:conv:1} -- \eqref{blas:conv:2} which initiates the induction. Then, fixing $i \ge 1$, assuming the bounds \eqref{blas:conv:1} -- \eqref{water:88} hold for $0, \dots, i - 1$, we obtain strong enough bounds on \eqref{intro:abs:AI:a} -- \eqref{intro:abs:AI:c} to close the induction. We point the reader to estimates \eqref{forcing:est:PR:2} -- \eqref{forcing:est:Pr:lin}, which quantifies precisely (in terms of regularity, $x$ decay, and $z$ decay) how the inductive bounds on $0, \dots, i-1$ feed into the forcing and the boundary condition for the $i$'th boundary layer. 

To close the entire argument, given precise estimates on $F$ and $g(x)$, which we now can regard as inputs to \eqref{intro:abs:AI:a} -- \eqref{intro:abs:AI:c}, we need to establish strong enough bounds on the corresponding solution to the linearized problem, $[u, v]$. This is achieved in Section \ref{section:PL} through a sequence of coupled weighted energy estimates. These estimates rely on several ingredients: introducing a ``von-Mise" good unknown, \eqref{good:variables}, introducing several delicate weighted norms, \eqref{X00:norm} -- \eqref{normdefuB}, and a cascade of bounds obtained through carefully chosen multipliers (Lemmas \ref{lem:LP:1} -- \ref{lemma:LP4}) which enable us to control these norms on the solution. All the while, we need to strike a delicate balance between the inductively available estimates on the inputs $F, g(x)$ with controlling a strong enough norm to propagate the inductive hypothesis. 

As a final remark, let us discuss the organization of the article. As we have mentioned above, Appendix \ref{app:A} is devoted to carefully performing the multiscale analysis that is required in order to collect the equations governing the various terms appearing in \eqref{exp:base}. We also incidentally obtain estimates \eqref{water:78} -- \eqref{water:88}, which are straightforward consequences of the harmonicity in \eqref{Eul:harm}. Section \ref{section:PL} studies the abstract problem \eqref{intro:abs:AI:a} -- \eqref{intro:abs:AI:c}. In fact, we take $g(x) = 0$; every instance in which we invoke the estimates from Section \ref{section:PL}, we perform homogenizations to translate the boundary condition $g(x)$ into the forcing. In Section \ref{NL:Section}, we apply the results of Section \ref{section:PL} to the nonlinear Prandtl layer, thereby establishing the $i = 0$ case, \eqref{blas:conv:1} -- \eqref{blas:conv:2} of the main result. In Section \ref{Section:LinPRA} we study the linearized Prandtl layers, $i = 1, \dots, N_1 -1$ and close the bounds \eqref{water:65} -- \eqref{water:88}. Section \ref{cutoff:SEC} is devoted to studying the final Prandtl layer, $i = N_1$, which is largely similar to the $i = 1, \dots N_1 - 1$ case but requires an additional cutoff argument at $y = \infty$, which contributes extra error terms that need to be controlled carefully. Finally, Section \ref{subsection:background} is devoted to accumulating ``aggregated estimates" on the complete background profiles, $[\bar{u}, \bar{v}]$, which can be thought of as a consequence of the estimates of the individual components \eqref{blas:conv:1} -- \eqref{water:88} and the summation in \eqref{exp:base}.

\section{Abstract Linearized Prandtl Equations} \label{section:PL}

\subsection{Formulation and Good Variables}

In this section, we analyze an abstract formulation of the linearized Prandtl equations against a background vector-field $[\bar{u}_B, \bar{v}_B]$:
\begin{align} \label{eq:abs:1:a}
&\bar{u}_B u_x + u \p_x \bar{u}_B + \bar{v}_B \p_y u + \bar{v} \p_y \bar{u}_B - u_{yy} = F, \\ \label{eq:abs:1:b}
&\bar{v} := - \int_0^y \p_x u, \\ \label{eq:abs:1:c}
&u|_{y = 0} = \bar{v}|_{y = 0} = u|_{y = \infty}  = 0, \\ \label{eq:abs:1:d}
&u|_{x = 0} = u_0(y). 
\end{align}
In this article, we will be making two choices of the background field: either $[\bar{u}_B, \bar{v}_B] := [\bar{u}^0_p, \bar{v}^0_p] = [\bar{u}_\ast, \bar{v}_\ast] + \delta_\ast [\tilde{u}^0_p, \tilde{v}^0_p]$ or $[\bar{u}_B, \bar{v}_B] := [\bar{u}_\ast, \bar{v}_\ast]$. To thus unify the treatment of these two choices, we introduce a parameter $\delta_{\ast \ast}$ which we will select to either take the value $\delta_{\ast}$ in the first case or $0$ in the second case. Subsequently, we define 
\begin{align} \label{deltastarstar}
[\bar{u}_B, \bar{v}_B] := [\bar{u}_\ast, \bar{v}_\ast] + \delta_{\ast \ast} [\tilde{u}^0_p, \tilde{v}^0_p].
\end{align}
We note that importantly, $[\bar{u}_B, \bar{v}_B]$ will be solutions to the nonlinear, homogeneous, Prandtl equations:
\begin{align} \label{Pr:uB}
\bar{u}_B \p_x \bar{u}_B + \bar{v}_B \p_y \bar{u}_B - \p_y^2 \bar{u}_B = 0. 
\end{align}

Let now $\psi$ be the associated stream function to the unknowns $[u, \bar{v}]$. We introduce the good variables,  
\begin{align} \label{good:variables}
q := \frac{\psi}{\bar{u}_B}, \qquad U := \p_y q = \frac{1}{\bar{u}_B} (u - \frac{\bar{u}_{By}}{\bar{u}_B} \psi) , \qquad V := - \p_x q = \frac{1}{\bar{u}_B} (\bar{v} + \frac{\bar{u}_{Bx}}{\bar{u}_B} \psi).  
\end{align}
It will be necessary to invert the above transformations, that is go from $(U, V) \mapsto (u, v)$. To do this, we record the following:
\begin{subequations}
\begin{align} \label{inv:u}
&u = \p_y \psi = \p_y \{ \bar{u}_B q \} = \bar{u}_B U + \bar{u}_{By} q, \\ \label{inv:v}
&\bar{v} = - \p_x \psi = - \p_x \{ \bar{u}_B q  \} = \bar{u}_B V - \bar{u}_{Bx} q. 
\end{align}
\end{subequations} 
Inserting now \eqref{inv:u} -- \eqref{inv:v} into the left-hand side of \eqref{eq:abs:1:a} yields 
\begin{align*}
&\bar{u}_B u_x + \bar{u}_{Bx} u + \bar{v}_B \p_y u + \bar{v} \p_y \bar{u}_B - u_{yy} \\
= & \bar{u}_B \p_x \{ \bar{u}_B U + \bar{u}_{By} q \} + \bar{u}_{Bx} (\bar{u}_B U + \bar{u}_{By} q) + \bar{v}_B \p_y \{ \bar{u}_B U + \bar{u}_{By} q \} + \bar{u}_{By}(\bar{u}_B V - \bar{u}_{Bx}q)  - u_{yy}
\end{align*}
\begin{align*}
= & |\bar{u}_B|^2 U_x + 2 (\bar{u}_B \bar{u}_{Bx} + \bar{v}_B \bar{u}_{By}) U + \bar{u}_B \bar{v}_B U_y + (\bar{u}_B \bar{u}_{Bxy} + \bar{v}_B \bar{u}_{Byy}  )q - u_{yy} \\
= & \mathcal{T}[U]+ \bar{u}_{Byyy} q -  u_{yy},
\end{align*}
where the transport operator $\mathcal{T}[U]$ is defined by 
\begin{align} \label{def:Pr:T}
\mathcal{T}[U] := (\bar{u}_B)^2 U_x + \bar{u}_B \bar{v}_B U_y + 2 \bar{u}_{Byy} U.
\end{align}
We note that the Prandtl equation, \eqref{Pr:uB}, has been invoked on $\bar{u}_B$ to go from the third to the fourth line above. To summarize, we generate the following identity 
\begin{align} \label{Pr:af}
\mathcal{T}[U]+ \bar{u}_{Byyy} q - \p_y^2 u = F.
\end{align}

As our estimates will  be in terms of the good variables introduced above, we set the notation 
\begin{align} \label{notation:1}
(\p_x^k U)_0(y) := \p_x^k U|_{x = 0} \text{ and } U^{(k)} := \p_x^k U. 
\end{align}
As a final notational point, since our definition of $U$ depends on the choice of normalizing background flow, $\bar{u}_B$, when it is not clear from context we will emphasize this dependance by introducing the notation 
\begin{align} \label{notation:U:ub}
U[\bar{u}_B] := \frac{1}{\bar{u}_B} (u - \frac{\bar{u}_{By}}{\bar{u}_B} \psi). 
\end{align}
However, in the vast majority of cases, this will be clear from context and we will thus omit the dependence of $\bar{u}_B$.  
\subsection{Norms and Embeddings} \label{subsection:LP:embed}

\subsubsection{Norms on the Solutions}

First, we introduce some parameters which will dictate regularity and $z$-decay that we control in our energy norms. As such, we fix 
\begin{align*}
&\ell_{max} = \text{index of $y$ regularity} \\
&k_{max} = \text{index of $x$ regularity} = \lfloor \frac{\ell_{max}}{2} \rfloor, \\
&m_{max} = \text{index of $z$ decay}. 
\end{align*}
Each time we invoke the estimates we are about to develop in Sections \ref{NL:Section}, \ref{Section:LinPRA}, we will specify the choice of these regularity parameters, $\ell_{max}, k_{max}, m_{max}$. 

We will perform parabolic-type energy estimates on \eqref{Pr:af}. Our energies will capture the decay claimed in \eqref{water:65}. We will define the following spaces (recall the definition of $U^{(k)}$ from \eqref{notation:1}):
\begin{align}  \n
\| U \|_{X_{k,0}} := &  \| \bar{u}_B U^{(k)} \langle x \rangle^{k- \sigma_\ast}   \|_{L^\infty_x L^2_y} + \sigma_\ast \|  \bar{u}_B U^{(k)} \langle x \rangle^{k - \frac 1 2- \sigma_\ast} \|  + \| \sqrt{\bar{u}_B} U^{(k)}_y \langle x \rangle^{k-\sigma_\ast}   \| \\ \label{X00:norm}
&+ \| \sqrt{\bar{u}_{By}} U^{(k)}_y|_{y = 0} \langle x \rangle^{k- \sigma_\ast} \|_{L^2_x},\\ 
\nonumber
\| U \|_{X_{k,m}} := &  \| \bar{u}_B U^{(k)} \langle x \rangle^{k- \sigma_\ast}  z^m \|_{L^\infty_x L^2_y} + \| \bar{u}_B U^{(k)} \langle x \rangle^{k- \frac 1 2- \sigma_\ast} z^m \|  + \| \sqrt{\bar{u}_B} U^{(k)}_y \langle x \rangle^{k-\sigma_\ast}  z^m \|, 
\end{align}
for $m = 1,...,m_{max}$ and for $k = 0,..,k_{max} + 1$. Define also
\begin{align} \label{def:X12:m}
\| U \|_{X_{k + \frac 1 2, m}} := & \| \bar{u}_B \p_x U^{(k)} \langle x \rangle^{k + \frac 1 2 - \sigma_\ast}  z ^m \| + \| \sqrt{\bar{u}_B} U^{(k)}_y \langle x \rangle^{k + \frac 1 2- \sigma_\ast}  z^m  \|_{L^\infty_x L^2_y}, 
\end{align}
for $m = 0,...,m_{max}$, $k = 0,...,k_{max}$. Finally, define
\begin{align} \n
\| U \|_{Y_{k + \frac 1 2, 0}} := & \| \bar{u}_B U^{(k)}_y \langle x \rangle^{k + \frac 1 2 - \sigma_\ast}  \|_{L^\infty_x L^2_y} + \| \sqrt{\bar{u}_B} U^{(k)}_{yy} \langle x \rangle^{k + \frac 1 2-\sigma_\ast}  \| \\ \label{Yh0:norm}
& + \| \sqrt{\bar{u}_{By}} U^{(k)}_y \langle x \rangle^{k + \frac 1 2-\sigma_\ast}|_{y = 0} \|_{L^2_x}, \\ \label{Yhm:norm}
\| U \|_{Y_{k + \frac 1 2, m}} := & \|\bar{u}_B U^{(k)}_y \langle x \rangle^{k + \frac 1 2 - \sigma_\ast} z^m  \|_{L^\infty_x L^2_y} + \| \sqrt{\bar{u}_B} U^{(k)}_{yy} \langle x \rangle^{k+\frac 1 2-\sigma_\ast} z^m \|, 
\end{align}
for $m = 1,..,m_{max}$ and $k = 0,...,k_{max}$.

It is convenient also to introduce the following notation, which can compactify sums of norms
\begin{align} \n
&\|U \|_{\mathcal{X}_{\le k, m}} := \|U \|_{X_{k,m}} + \sum_{j = 0}^{{k-1}} \Big( \|  U \|_{X_{j,m}} + \| U \|_{X_{j + \frac 1 2,m}}  + \| U \|_{Y_{j + \frac 1 2,m}} \Big), \\ \n
&\|U \|_{\mathcal{X}_{\le k + \frac 1 2, m}} :=   \sum_{j = 0}^{{k}} \Big( \| U \|_{X_{j,m}} + \|  U \|_{X_{j + \frac 1 2,m}}  + \| U \|_{Y_{j + \frac 1 2,m}} \Big).
\end{align} 

On one occasion we will need to consider ``finite"-$x$ versions of the norms $X_{k,0}, X_{k,m}$ (to run a Gronwall-type argument precisely). We give the following notation for this. Fix $0 < X_{\ast} < \infty$. 
\begin{align}  \n
\| U \|_{X_{k,0, X_{\ast}}} := &  \sup_{0 \le x \le X_{\ast}} \| \bar{u}_B U^{(k)} \langle x \rangle^{k- \sigma_\ast}   \|_{L^2_y} + \sigma_\ast \|  \bar{u}_B U^{(k)} \langle x \rangle^{k - \frac 1 2- \sigma_\ast} \|_{L^2_x(0, X_\ast) L^2_y} \\ \label{X00:norm:finX}
& + \| \sqrt{\bar{u}_B} U^{(k)}_y \langle x \rangle^{k-\sigma_\ast}   \|_{L^2_x(0, X_\ast) L^2_y} + \| \sqrt{\bar{u}_{By}} U^{(k)}_y|_{y = 0} \langle x \rangle^{k- \sigma_\ast} \|_{L^2_x(0, X_{\ast}) L^2_y},\\ \n
\| U \|_{X_{k,m, X_{\ast}}} := &  \sup_{0 \le x \le X_{\ast}} \| \bar{u}_B U^{(k)} \langle x \rangle^{k- \sigma_\ast}  z^m \|_{L^2_y} + \| \bar{u}_B U^{(k)} \langle x \rangle^{k- \frac 1 2- \sigma_\ast} z^m \|_{L^2_x(0, X_{\ast}) L^2_y}  \\ \label{X0m:norm:finX} &+ \| \sqrt{\bar{u}_B} U^{(k)}_y \langle x \rangle^{k-\sigma_\ast}  z^m \|_{L^2_x(0, X_{\ast}) L^2_y}, 
\end{align}

As a final notational point, as the norms in \eqref{X00:norm} -- \eqref{Yhm:norm} depend on the choice of background flow, $\bar{u}_B$, we will on occasion choose to emphasize this dependance through the notation: 
\begin{align} \n
\| U \|_{X_{k,0}[\bar{u}_B]} := &  \| \bar{u}_B U^{(k)} \langle x \rangle^{k- \sigma_\ast}   \|_{L^\infty_x L^2_y} + \sigma_\ast \|  \bar{u}_B U^{(k)} \langle x \rangle^{k - \frac 1 2- \sigma_\ast} \|  + \| \sqrt{\bar{u}_B} U^{(k)}_y \langle x \rangle^{k-\sigma_\ast}   \| \\ \label{normdefuB}
&+ \| \sqrt{\bar{u}_{By}} U^{(k)}_y|_{y = 0} \langle x \rangle^{k- \sigma_\ast} \|_{L^2_x},
\end{align} 
and so on. However, in the vast majority of cases, it will be clear from context and we will suppress the dependence on $\bar{u}_B$. 
 \subsubsection{Norms on the Background}

In addition to norms on the actual solution (introduced above), we also introduce norms in which we keep track of the background, $[\bar{u}_B, \bar{v}_B]$. Define 
\begin{align} \label{norm:B:def}
\| \tilde{u}^0_p, \tilde{v}^0_p \|_{X_{B}} := &\sum_{k \le 10} \sum_{j \le 20} \| \p_y^j  \p_x^k \tilde{u}^0_p \langle x \rangle^{\frac 1 4 + k + \frac j 2 - 2 \sigma_{\ast}}  \langle z \rangle^{10}\|_{L^\infty_{xy}}, \\  \label{norm:B:def:finite}
\| \tilde{u}^0_p, \tilde{v}^0_p \|_{X_{B, X_{\ast}}} := & \sum_{k \le 10} \sum_{j \le 20} \sup_{0 \le x \le X_{\ast}} \| \p_y^j  \p_x^k \tilde{u}^0_p \langle x \rangle^{\frac 1 4 + k + \frac j 2 - 2 \sigma_{\ast}}  \langle z \rangle^{10}\|_{L^\infty_{y}} 
\end{align}
A few of our forthcoming estimates (for instance, \eqref{Hardy:1:here}, \eqref{emb:start:1}, \eqref{case:case:0}, \eqref{case:case:m}), are valid under the following hypothesis on the background $[\bar{u}_B, \bar{v}_B]$: 
\begin{align} \label{boots:1}
\| \tilde{u}^0_p, \tilde{v}^0_p \|_{X_B} \le \delta_{\ast \ast}^{- \frac 12}.
\end{align}
Note that we will explicitly say this in the hypothesis of each such lemma. Note also that we interpret \eqref{boots:1} as a hypothesis on $[\bar{u}_B, \bar{v}_B]$ due to \eqref{deltastarstar}.  

 On the other hand, certain of our other estimates (for instance: Lemma \ref{lem:LP:2}, Lemma \ref{lem:LP:3}, Lemma \ref{lemma:LP4})  hold only under stronger assumptions on the background. Specifically, we define:
\begin{definition} We refer to the following as our ``Strong Inductive Estimate":
\begin{align} \label{indu:1}
&\| \langle z \rangle^{m_0} \p_x^k \p_y^j ( \bar{u}_B - \bar{u}_\ast) \|_{L^\infty_y} \lesssim \delta_{\ast \ast}^{\frac 12} \langle x \rangle^{- \frac 1 4 - k - \frac j 2 + \sigma_\ast}, \qquad 2k + j \le \ell_{max} \\ \label{indu:2}
&\| \langle z \rangle^{m_0} \p_x^k \p_y^j ( \bar{v}_B - \bar{v}_\ast) \|_{L^\infty_y} \lesssim \delta_{\ast \ast}^{\frac 1 2} \langle x \rangle^{- \frac 3 4 - k - \frac j 2 + \sigma_\ast}, \qquad 2k + j \le \ell_{max}.
\end{align}
\end{definition}
Each time we assume \eqref{indu:1} -- \eqref{indu:2} as a hypothesis, we will explicitly state so.

\subsubsection{Embeddings}

We state now the following Hardy-type inequality:
\begin{lemma}  Let $[\bar{u}_B, \bar{v}_B]$ be as in \eqref{deltastarstar}, and assume that $[\tilde{u}^0_p, \tilde{v}^0_p]$ satisfies \eqref{boots:1}. For $0 < \gamma << 1$, and for any function $f \in H^1_y$, 
\begin{align} \label{Hardy:1:here}
\| f \|_{L^2_y}^2 \lesssim \gamma \| \sqrt{ \bar{u}_B}  f_y \langle x \rangle^{\frac 1 2} \|_{L^2_y}^2 + \frac{1}{\gamma^2} \| \bar{u}_B f \|_{L^2_y}^2. 
\end{align}
\end{lemma}
\begin{proof} We square the left-hand side and localize the integral based on $z$ via 
\begin{align}
\int f^2 \ud y = \int f^2 \chi(\frac{z}{\gamma}) \ud y + \int f^2 (1 - \chi(\frac{z}{\gamma})) \ud y. 
\end{align}
For the localized component, we integrate by parts in $y$ via 
\begin{align}
\int f^2 \chi(\frac{z}{\gamma}) \ud y = \int \p_y (y) f^2 \chi(\frac{z}{\gamma}) \ud y = - \int 2 y f f_y \chi(\frac{z}{\gamma}) \ud y - \frac{1}{\gamma} \int \frac{y}{\sqrt{x}} f^2 \chi'(\frac{z}{\gamma}) \ud y. 
\end{align}
We estimate each of these terms via 
\begin{align}
\Big| \int y f f_y \chi(\frac{z}{\gamma})\ud y \Big| \lesssim \| f  \|_{L^2_y} \|   \sqrt{x} \sqrt{ \bar{u}_B} \sqrt{\gamma} f_y \|_{L^2_y} \le \delta \| f \|_{L^2_y}^2 + C_\delta \gamma x \| \sqrt{ \bar{u}_B} f_y \|_{L^2_y}^2.
\end{align}
For the far-field term, we estimate again via 
\begin{align}
|\int f^2 (1 - \chi(\frac{z}{\gamma})) \ud y| = |\int \frac{1}{|\bar{u}_B|^2} |\bar{u}_B|^2 f^2 (1 - \chi(\frac{z}{\gamma})) \ud y| \lesssim \frac{1}{\gamma^2} \| \bar{u}_B f \|_{L^2_y}^2.
\end{align}
We have thus obtained 
\begin{align}
\| f \|_{L^2_y}^2 \le \delta \| f \|_{L^2_y}^2 + C_\delta \gamma x \| \sqrt{ \bar{u}_B } f_y \|_{L^2_y}^2 + \frac{C}{\gamma^2} \|  \bar{u}_B f \|_{L^2_y}^2, 
\end{align}
and the desired result follows from taking $\delta$ small relative to universal constants and absorbing to the left-hand side. 
\end{proof}

\subsection{Energy Estimates} \label{subsection:LP:energy}

We recall that the equations \eqref{eq:lin:Pr}, \eqref{eq:lin:Pr:hom}, and \eqref{Pr:af} are all equivalent. We will now perform energy estimates in the formulation \eqref{Pr:af}. We recall the notational convention used in this section that, when unspecified, $\int f := \int_0^\infty f \ud y$. 

\begin{lemma} \label{lem:LP:1}  Let $[\bar{u}_B, \bar{v}_B]$ be as in \eqref{deltastarstar}, and assume that $[\tilde{u}^0_p, \tilde{v}^0_p]$ satisfies \eqref{boots:1}. Let $U$ be a solution to \eqref{Pr:af}. Then the following estimates are valid
\begin{align} \label{case:case:0}
\|U \|_{X_{0,0}}^2 \lesssim & \| \bar{u}_B|_{x = 0} U_0 \|_{L^2_y}^2 + \| F \langle x \rangle^{\frac 1 2- \sigma_\ast} \|^2, \\ \label{case:case:m}
\| U \|_{X_{0, m}}^2 \lesssim & \|U \|_{X_{0, m-1}}^2 +  \| \bar{u}_B|_{x = 0} U_0 y^m \|_{L^2_y}^2 +   \| F \langle x \rangle^{\frac 1 2- \sigma_\ast} z^m \|^2,
\end{align}
for $1 \le m \le m_{max}$.  
\end{lemma}
\begin{proof} We multiply by $U \langle x \rangle^{-2\sigma_\ast} z^{2m}$, for $m = 0,...,m_{max}$, which produces the following identity 
\begin{align} \label{yoyoyo}
\int \mathcal{T}[U] U \langle x \rangle^{-2\sigma_\ast} z^{2m}  + \int (- \p_y^2 U + \bar{u}_{Byyy} q) U \langle x \rangle^{-2\sigma_\ast} z^{2m} = \int F^{(i)} U \langle x \rangle^{-2\sigma_\ast} z^{2m}. 
\end{align}
We first analyze the transport terms, which are energetic:
\begin{align} \n
\int \mathcal{T}[U] U  \langle x \rangle^{-2\sigma_\ast} z^{2m} \ud y = & \frac{\p_x}{2} \int |\bar{u}_B|^2 U^2 \langle x \rangle^{-2\sigma_\ast} z^{2m} + (\sigma_\ast + \frac m 2) \int \bar{u}_B^2 U^2 \langle x \rangle^{-2\sigma_\ast-1} z^{2m} \\ \n
& - \frac 1 2 \int (\bar{u}_B\bar{u}_{Bx} + \bar{v}_B \bar{u}_{By}) U^2 \langle x \rangle^{-2\sigma_\ast} z^{2m} - m \int \bar{u}_B \bar{v}_B U^2 \langle x \rangle^{-2\sigma_\ast - \frac 1 2} z^{2m-1} \\ \n
& + \int 2 \bar{u}_{Byy} U^2 \langle x \rangle^{-2\sigma_\ast} z^{2m} \\ \n
= &  \frac{\p_x}{2} \int |\bar{u}_B|^2 U^2 \langle x \rangle^{-2\sigma_\ast} z^{2m} + (\sigma_\ast + \frac m 2) \int |\bar{u}_B|^2 U^2 \langle x \rangle^{-2\sigma_\ast-1} z^{2m} \\ \label{ident:trans:energy}
& + \frac 3 2 \int  \bar{u}_{Byy} U^2 \langle x \rangle^{-2\sigma_\ast} z^{2m} - m \int \bar{u}_B \bar{v}_B U^2 \langle x \rangle^{-2\sigma_\ast - \frac 1 2} z^{2m-1},
\end{align}
where we have invoked the nonlinear Prandtl equation satisfied by the background, $[\bar{u}^0_p, \bar{v}^0_p]$. The first two terms are energetic contributions, whereas the third term will be cancelled out below, see \eqref{damping:1}. For the fourth term, we estimate after integration over $x$ via 
\begin{align}
\int \int |\bar{u}_B \bar{v}_B U^2 \langle x \rangle^{-2\sigma_\ast - \frac 1 2} z^{2m-1}| \lesssim \| \bar{u}_B U \langle x \rangle^{- \frac 1 2 - \sigma_\ast} z^{m-1} \| \| \bar{u}_B U \langle x \rangle^{- \frac 1 2 - \sigma_\ast} z^{m} \|, 
\end{align}
the former term has been controlled inductively by $\| U \|_{X_{0, m-1}}$.

We next analyze the diffusive term. Due to boundary contributions which exist only when $m = 0$, it is convenient to first display the $m = 0$ calculation, which gives 
\begin{align} \n
- \int \p_y^2 u U \langle x \rangle^{-2\sigma_\ast} \ud y = & u_y U(x, 0) \langle x \rangle^{-2\sigma_\ast} + \int u_y U_y \langle x \rangle^{-2\sigma_\ast}\ud y \\ \n
= & \p_y (\bar{u}_B U + \bar{u}_{By} q) U(x, 0) \langle x \rangle^{-2\sigma_\ast} + \int \p_y (\bar{u}_B U + \bar{u}_{By} q) U_y \langle x \rangle^{-2\sigma_\ast} \ud y
\end{align}
\begin{align}
\n
= & 2\bar{u}_{By} U^2(x, 0) \langle x \rangle^{-2\sigma_\ast}  + \int \bar{u}_B U_y^2 \langle x \rangle^{-2\sigma_\ast}\ud y + 2 \int \bar{u}_{By} U U_y \langle x \rangle^{-2\sigma_\ast}\ud y \\ \n
& + \int \bar{u}_{Byy} q U_y \langle x \rangle^{-2\sigma_\ast} \ud y \\ \n
= &  \bar{u}_{By} U^2(x, 0) \langle x \rangle^{-2\sigma_\ast}+ \int \bar{u}_B U_y^2 \langle x \rangle^{-2\sigma_\ast} - 2  \int \bar{u}_{Byy} U^2 \langle x \rangle^{-2\sigma_\ast}\\ \label{jmjm}
&+ \frac 1 2 \int \bar{u}_{Byyyy} q^2 \langle x \rangle^{-2\sigma_\ast}, 
\end{align}
and the contribution from the $q$ term is 
\begin{align} \label{q:term:1}
\int\bar{u}_{Byyy} q U \langle x \rangle^{-2\sigma_\ast} = - \frac 1 2 \int \bar{u}_{Byyyy} q^2 \langle x \rangle^{-2\sigma_\ast}. 
\end{align}
This term cancels exactly the fourth term from \eqref{jmjm}. The first term from \eqref{jmjm} is positive due to 
\begin{align}
\bar{u}_{By}(x, 0) = \p_y \bar{u}_{\ast}(x, 0) + \delta_{\ast \ast} \widetilde{u}^0_{py}(x, 0) \gtrsim \langle x \rangle^{- \frac 1 2} - \delta_{\ast \ast}^{\frac 12} \langle x \rangle^{- \frac 1 2} \gtrsim \langle x \rangle^{- \frac 1 2},
\end{align}
where we have invoked the bootstrap \eqref{boots:1}. 

The third term from \eqref{jmjm} combines with the third term from \eqref{ident:trans:energy}, which contributes the following damping term 
\begin{align} \label{damping:1}
- \frac 1 2 \int \bar{u}_{Byy} U^2 \langle x \rangle^{-2\sigma_\ast} = - \frac 1 2 \int \bar{u}_{\ast yy} U^2 \langle x \rangle^{-2\sigma_\ast} - \frac 1 2 \int (\bar{u}_{Byy}  - \bar{u}_{\ast yy} ) U^2 \langle x \rangle^{-2\sigma_\ast},
\end{align}
and we estimate the latter contribution after integration in $x$ via 
\begin{align} \n
\int \int |(\bar{u}_{Byy}  - \bar{u}_{\ast yy} )| U^2 \langle x \rangle^{-2\sigma_\ast} \lesssim & \delta_{\ast \ast}^{\frac 12} \| \bar{u}_B U \langle x \rangle^{- \frac{9}{16}} \|^2 + \delta_{\ast \ast}^{\frac 12} \| \sqrt{\bar{u}_B} U_y \langle x \rangle^{- \frac{1}{16}} \|^2, 
\end{align}
where we have used the estimate \eqref{blas:conv:1}, and the Hardy type inequality \eqref{Hardy:1:here}. Due to the smallness of $\delta_{\ast \ast}$, these terms are absorbed into the left-hand side. 

Next, we treat the case when $m = 1,...,m_{max}$, which gives 
\begin{align} \n
- \int \p_y^ 2u U \langle x \rangle^{-2\sigma_\ast} z^{2m} = & \int u_y U_y \langle x \rangle^{-2\sigma_\ast} z^{2m} + 2m \int u_y U z^{2m-1} \langle x \rangle^{-2\sigma_\ast - \frac 1 2} \\ \n
= & \int \bar{u}_B U_y^2 \langle x \rangle^{-2\sigma_\ast} z^{2m} - \int 2 \bar{u}_{Byy} U^2 \langle x \rangle^{-2\sigma_\ast} z^{2m} + \frac 1 2 \int \bar{u}_{Byyyy} q^2 \langle x \rangle^{-2\sigma_\ast} z^{2m} \\ \n
& + 2m \int \bar{u}_{Byyy} q^2 \langle x \rangle^{-2\sigma_\ast - \frac 1 2} z^{2m-1} - m (2m-1) \int \bar{u}_B U^2 z^{2m-2} \langle x \rangle^{-2\sigma_\ast - 1} \\ \label{qtermpr2}
& + m(2m-1) \int \bar{u}_{Byy} q^2 \langle x \rangle^{-2\sigma_\ast - 1} z^{2m-2},
\end{align}
and the contribution from the $q$ term is 
\begin{align} \label{qtermpr1}
\int \bar{u}_{Byyy} q U \langle x \rangle^{-2\sigma_\ast} z^{2m} = & - \frac 1 2 \int \bar{u}_{Byyyy} q^2 \langle x \rangle^{-2\sigma_\ast} z^{2m} - m \int \bar{u}_{Byyy} q^2 \langle x \rangle^{-2\sigma_\ast - \frac 1 2} z^{2m-1}.
\end{align}
The first term on the right-hand side of \eqref{qtermpr1} cancels the third term from \eqref{qtermpr2}. The remaining terms from \eqref{qtermpr2} and \eqref{qtermpr1} are:
\begin{align}
\mathcal{I}_{m,q} := m \int \bar{u}_{Byyy} q^2 \langle x \rangle^{-2\sigma_\ast - \frac 1 2} z^{2m-1} + m (2m -1) \int \bar{u}_{Byy} q^2 \langle x \rangle^{-2\sigma_\ast - 1} z^{2m-2}
\end{align} 
For the $m = 0$ case, $I_{0,q} = 0$, and therefore we do not need to estimate this term to prove \eqref{case:case:0}. Assume now $m \ge 1$. We can estimate the first term above via 
\begin{align*}
| \int \bar{u}_{Byyy} q^2 \langle x \rangle^{-2\sigma_\ast - \frac 1 2} z^{2m-1}| \lesssim & \| \bar{u}_{Byyy} z^{2m-1} \langle x \rangle^{\frac 12} y^2 \|_{L^\infty_{xy}} \| U \|_{L^2_y}^2 \langle x \rangle^{-1-2\sigma_\ast} \\
\lesssim & \| \bar{u}_B U \|_{L^2_y}^2 \langle x \rangle^{-1-2\sigma_\ast} + \| \sqrt{\bar{u}_B} U_y \|_{L^2_y}^2 \langle x \rangle^{-2\sigma_\ast},
\end{align*}
both of which are integrable according to the $X_{0,0}$ norm, defined in \eqref{X00:norm}, and which has been controlled inductively for $m \ge 1$. The second term from $\mathcal{I}_{m,q}$ is estimated in a nearly identical manner, when $m \ge 2$.

Finally, upon integrating in $x$, we estimate the forcing terms via 
\begin{align}
|\int \int F U \langle x \rangle^{-2\sigma_\ast} z^{2m} \ud y \ud x| \lesssim \| F \langle x \rangle^{\frac 1 2 - \sigma_\ast} z^m \| \| U \langle x \rangle^{- \frac 1 2 - \sigma_\ast} z^{m} \|. 
\end{align}
To conclude the proof of the lemma, we integrate \eqref{yoyoyo} over $x \in [0, X_0]$ and then take the supremum over $X_0$. Upon doing so, the first term from \eqref{ident:trans:energy} yields the first quantity in \eqref{X00:norm}, the second term from \eqref{ident:trans:energy} yields the second quantity in \eqref{X00:norm}, the first and second terms from \eqref{jmjm} give the third and fourth terms from \eqref{X00:norm}, respectively. Upon using the elementary inequality $\sup_{X_0} |\int_0^{X_0} \int g \ud y \ud x| \le \int \int |g| \ud y \ud x$, we appeal to the estimation of the error terms above.   
\end{proof}

It is useful for us to state a ``finite-x" version of the above estimate, which will be useful on one occasion in the future. We formulate this as follows:
\begin{corollary}  Let $[\bar{u}_B, \bar{v}_B]$ be as in \eqref{deltastarstar}, and assume that $[\tilde{u}^0_p, \tilde{v}^0_p]$ satisfies \eqref{boots:1}. Fix any $X_\ast > 0$. Assume the bootstrap bound 
\begin{align} \label{boots:2}
\sup_{0 < x \le X_\ast} \sum_{k \le 10} \sum_{j \le 20} \|\p_x^k \p_y^j \tilde{u}^0_p(x, \cdot) \|_{L^\infty_y} \langle x \rangle^{\frac 14 - 2 \sigma_\ast + k + \frac j 2} \le \delta_{\ast \ast}^{- \frac 12}.   
\end{align}
Then 
\begin{align} \label{case:case:0:fin}
\|U \|_{X_{0,0, X_{\ast}}}^2 \lesssim & \| \bar{u}_B|_{x = 0} U_0 \|_{L^2_y}^2 + \| F \langle x \rangle^{\frac 1 2- \sigma_\ast} \|_{L^2_x(0, X_{\ast}) L^2_y}^2, \\ \label{case:case:m:fin}
\| U \|_{X_{0, m, X_{\ast}}}^2 \lesssim & \|U \|_{X_{0, m-1}}^2 +  \| \bar{u}_B|_{x = 0} U_0 y^m \|_{L^2_y}^2 +   \| F \langle x \rangle^{\frac 1 2- \sigma_\ast} z^m \|_{L^2_x(0, X_{\ast}) L^2_y}^2,
\end{align}
for $1 \le m \le m_{max}$. 
\end{corollary}
\begin{proof} This follows upon repeating the proof of the previous lemma.  
\end{proof}

We now estimate the $X_{\frac 1 2, m}$ scale of norms, defined in \eqref{def:X12:m}.
\begin{lemma} \label{lem:LP:2}Assume the inductive hypotheses \eqref{indu:1} -- \eqref{indu:2} . Let $U$ be a solution to \eqref{Pr:af}. Then the following estimates are valid for $m = 1,...,m_{max}$, and any $0 < \delta << 1$, 
\begin{align}
\| U \|_{X_{\frac 1 2, 0}}^2 \lesssim & C_\delta \| U \|_{X_{0,0}}^2 + \delta \|U \|_{X_{1,0}}^2 + \delta \|U \|_{Y_{\frac 1 2,0}}^2 + \| F \langle x \rangle^{\frac 1 2- \sigma_\ast}  \|^2 + \| \sqrt{\bar{u}_B} \p_y U_0 \|_{L^2_y}^2, \\ \n
\| U \|_{X_{\frac 1 2, m }}^2 \lesssim &\| U \|_{X_{0,m}}^2 + \| U \|_{X_{\frac 1 2, m -1 }}^2 +   \delta \|U \|_{X_{1,0}}^2 + \delta \|U \|_{Y_{\frac 1 2,0}}^2 + C_\delta \|U \|_{X_{0,0}}^2 \\
& + \| F \langle x \rangle^{\frac 1 2- \sigma_\ast} z^m \|^2 + \| \sqrt{\bar{u}^0_p} \p_y U_0  y^m\|_{L^2_y}^2.
\end{align}
\end{lemma}
\begin{proof} We apply the multiplier $U_x \langle x \rangle^{1-2\sigma_\ast} z^{2m}$ for $m = 0,...,m_{max}$, which generates the identity 
\begin{align} \n
&\int \mathcal{T}[U] U_x  \langle x \rangle^{1-2\sigma_\ast} z^{2m} \ud y - \int \p_y^2 u U_x  \langle x \rangle^{1-2\sigma_\ast} z^{2m} \ud y \\ \label{energetic:2}
+& \int \bar{u}_{Byyy} q U_x  \langle x \rangle^{1-2\sigma_\ast} z^{2m} \ud y = \int F^{(i)} U_x  \langle x \rangle^{1-2\sigma_\ast} z^{2m} \ud y. 
\end{align}
We will first analyze those terms coming from the first integral above in \eqref{energetic:2}. We have 
\begin{align} \n
\int \mathcal{T}[U] U_x  \langle x \rangle^{1-2\sigma_\ast} z^{2m} = &\int |\bar{u}_B|^2 U_x^2  \langle x \rangle^{1-2\sigma_\ast} z^{2m}  + \int \bar{u}_B \bar{v}^0_p U_y U_x  \langle x \rangle^{1-2\sigma_\ast} z^{2m}  \\ \label{treL1}
& + \int 2 \bar{u}_{Byy} U U_x  \langle x \rangle^{1-2\sigma_\ast} z^{2m}. 
\end{align}
The first term from \eqref{treL1} is a positive contribution to the left-hand side. The second term, we estimate via 
\begin{align}
|\int \bar{u}_B \bar{v}_B U_y U_x \langle x \rangle^{1-2\sigma_\ast} z^{2m}| \lesssim \| \frac{\bar{v}_B}{\bar{u}_B} \langle x \rangle^{\frac 1 2} \|_\infty \| \sqrt{\bar{u}_B} U_y \langle x \rangle^{-\sigma_\ast} z^{m} \|_{L^2_y} \| \bar{u}_B U_x \langle x \rangle^{\frac 1 2 - \sigma_\ast} z^m \|_{L^2_y}, 
\end{align}
where we have invoked the decay estimate on $\bar{v}_B$ from \eqref{v:blasius} and \eqref{indu:2}. For the third term from \eqref{treL1}, we obtain 
\begin{align} \label{identical:1}
|\int \bar{u}_{Byy} UU_x \langle x \rangle^{1-2\sigma_\ast} z^{2m}| \lesssim \| \frac{\bar{u}_{Byy}}{\bar{u}_B} \langle x \rangle \|_\infty \| U \langle x \rangle^{- \frac 1 2 - \sigma_\ast} z^m \|_{L^2_y} \| \bar{u}_B U_x \langle x \rangle^{\frac 1 2 - \sigma_\ast} z^m \|_{L^2_y}, 
\end{align}
where we have invoked the decay estimate on $\bar{u}^0_{Byy}$ which results from combining \eqref{water:1} and \eqref{indu:1}. 

We now address the second term, containing $- \p_{yy} u$, from \eqref{energetic:2}. For this, it is more convenient to split into the case when $m = 0$ and when $m \ge 1$. We first treat the case when $m = 0$, in which case we generate the following identity 
\begin{align} \label{pacha:1}
- \int \p_y^2 u U_x \langle x \rangle^{1-2\sigma_\ast} = \int u_y U_{xy} \langle x \rangle^{1-2\sigma_\ast} \ud y + u_y U_x \langle x \rangle^{1-2\sigma_\ast}(x, 0). 
\end{align}
We first treat the boundary contribution from above by expanding 
\begin{align} \n
u_y U_x \langle x \rangle^{1-2\sigma_\ast}(x,0) = & 2 \bar{u}_{By} UU_x \langle x \rangle^{1-2\sigma_\ast}(x, 0) \\ \label{identical:2}
= & \p_x (\bar{u}_{By} U^2 \langle x \rangle^{1-2\sigma_\ast}) - \bar{u}_{Bxy} U^2 \langle x \rangle^{1-2\sigma_\ast} - (1-2\sigma_\ast) \bar{u}_{By} U^2 \langle x \rangle^{-2\sigma_\ast},
\end{align}
which, upon further integrating in $x$, is bounded above by $\| U \|_{X_{0,0}}^2$. 

For the first integral in \eqref{pacha:1}, we have
\begin{align} \label{pacha:2}
\int u_y U_{xy} \langle x \rangle^{1-2\sigma_\ast} = \int (\bar{u}_B U_y + 2 \bar{u}_{By} U + \bar{u}_{Byy}q) U_{xy} \langle x \rangle^{1-2\sigma_\ast}.
\end{align}
We treat each of the three terms above individually. The first term gives 
\begin{align} \label{energy2222}
\int \bar{u}_B U_y U_{xy} \langle x \rangle^{1-2\sigma_\ast} = \frac{\p_x}{2} \int \bar{u}_B U_y^2 \langle x \rangle^{1-2\sigma_\ast} -  \frac{1}{2} \int \bar{u}_{Bx} U_y^2 \langle x \rangle^{1-2\sigma_\ast} - \frac{1-2\sigma_\ast}{2} \int \bar{u}_B U_y^2 \langle x \rangle^{-2\sigma_\ast}.  
\end{align}
The first term above is energetic, while the third is clearly controlled by $\| U \|_{X_{0,0}}^2$ upon integrating in $x$.  We estimate the second term, after integration in $x$, via 
\begin{align}
|\int \int \bar{u}_{Bx} U_y^2 \langle x \rangle^{1-2\sigma_\ast}| \lesssim \| \frac{ \bar{u}_{Bx} }{\bar{u}_B} \langle x \rangle \|_\infty \| \sqrt{\bar{u}_B} U_y \langle x \rangle^{-\sigma_\ast} \|^2 \lesssim \| U \|_{X_{0,0}}^2, 
\end{align}
where we have invoked the decay estimate \eqref{water:1} as well as \eqref{indu:1}. 

For the second term from \eqref{pacha:2}, we integrate by parts in $y$ which gives 
\begin{align} 
\int 2 \bar{u}_{By} U U_{xy} \langle x \rangle^{1-2\sigma_\ast} =& - \int \bar{u}_{Byy} U U_x \langle x \rangle^{1-2\sigma_\ast} - \int \bar{u}_{By} U_y U_x \langle x \rangle^{1-2\sigma_\ast}  - \bar{u}_{By} UU_x \langle x \rangle^{1-2\sigma_\ast}(x, 0). 
\end{align}
The first term above is identical to \eqref{identical:1}, while the boundary contribution above is identical to \eqref{identical:2}. For the second term, we localize based on the value of $z$. First, when $z \ge 1$, we have (again integrating in $x$)
\begin{align} \n
|\int \int \bar{u}_{By} U_y U_x \langle x \rangle^{1-2\sigma_\ast} (1 - \chi(z))| \lesssim & \| \bar{u}_{By} \langle x \rangle^{\frac 1 2} \|_\infty \| \sqrt{\bar{u}}_B U_y \langle x \rangle^{-\sigma_\ast} \| \| \bar{u}_B U_x \langle x \rangle^{\frac 1 2 - \sigma_\ast} \| \\ \label{lights:1}
\lesssim & \| U \|_{X_{0,0}} \| U \|_{X_{\frac 1 2, 0}},
\end{align}
where we have used that $z \ge 1$ to insert factors of $\bar{u}_B$ due to the boundedness of $\bar{u}_B^{-1}$ when $z \ge 1$, as well as the pointwise decay of $\bar{u}_{By}$ from \eqref{water:1} and \eqref{indu:1}. For the case when $z < 1$, we need to invoke the norms $Y_{\frac 1 2, 0}, X_{1,0}$. Indeed, we estimate (again after integration in $x$)
\begin{align} \n
&|\int \int \bar{u}_{By} U_y U_x \langle x \rangle^{1-2\sigma_\ast} \chi(z)| \lesssim  \| \bar{u}_{By} \langle x \rangle^{\frac 1 2} \|_\infty \| U_y \langle x \rangle^{-\sigma_\ast} \| \| U_x \langle x \rangle^{\frac 1 2 - \sigma_\ast}  \| \\ \n
\le & (C_{\delta_1} \| \sqrt{\bar{u}_B} U_y \langle x \rangle^{-\sigma_\ast} \| + \delta_1 \| \sqrt{\bar{u}_B} U_{yy} \langle x \rangle^{\frac 1 2 - \sigma_\ast} \| ) ( C_{\delta_2} \| \bar{u}_B U_x \langle x \rangle^{\frac 1 2 - \sigma_\ast}  \| + \delta_2 \| \sqrt{\bar{u}_B} U_{xy} \langle x \rangle^{1-\sigma_\ast} \| ) \\ \label{lights:2}
\le & C_{\delta} \|U \|_{X_{0,0}}^2 + \delta \|U \|_{X_{\frac 1 2, 0}}^2 + \delta \|U \|_{X_{1,0}}^2 +  \delta \|U \|_{Y_{\frac 1 2, 0}}^2.
\end{align}

For the third term from \eqref{pacha:2}, we again integrate by parts in $y$ which gives 
\begin{align}
\int \bar{u}_{Byy} q U_{xy} \langle x \rangle^{1-2\sigma_\ast} = - \int \bar{u}_{Byyy} q U_x \langle x \rangle^{1-2\sigma_\ast} - \int \bar{u}_{Byy} UU_x \langle x \rangle^{1-2\sigma_\ast}, 
\end{align}
where the first term above cancels the third term from \eqref{energetic:2}, and the second term above is identical to \eqref{identical:1}.

We now handle the case when $m = 1,...,m_{max}$. Here, we do not acquire boundary contributions at $y = 0$ as in the $m = 0$ case, but we generate commutators in $z$.  
\begin{align} \n
- \int u_{yy} U_x \langle x \rangle^{1-2\sigma_\ast} z^{2m} = & \int u_y U_{xy} \langle x \rangle^{1-2\sigma_\ast} + 2m \int u_y U_x \langle x \rangle^{\frac 1 2 - 2 \sigma_\ast} z^{2m-1} \\ \n
= & \int \bar{u}_B U_y U_{xy} \langle x \rangle^{1-2\sigma_\ast} z^{2m} + \int 2 \bar{u}_{By} U U_{xy} \langle x \rangle^{1-2\sigma_\ast} z^{2m} \\ \n
&+ \int \bar{u}_{Byy} q U_{xy} \langle x \rangle^{1-2\sigma_\ast} z^{2m} + 2m \int \bar{u}_B U_y U_x \langle x \rangle^{\frac 1 2 - 2 \sigma_\ast} z^{2m-1} \\ \label{honne:j:1}
&+ 4 m\int \bar{u}_{By} U U_x \langle x \rangle^{\frac 1 2 - 2 \sigma_\ast} z^{2m-1} + 2m \int \bar{u}_{Byy} q U_x \langle x \rangle^{\frac 1 2 - 2 \sigma_\ast} z^{2m-1}.
\end{align}
For the first term in \eqref{honne:j:1}, we have
\begin{align} \label{Baleaf}
\int \bar{u}_B U_y U_{xy} \langle x \rangle^{1-2\sigma_\ast} z^{2m} = \frac{\p_x}{2} \int \bar{u}_B U_y^2 \langle x \rangle^{1-2\sigma_\ast} z^{2m} - \int \bar{u}_{Bx} U_y^2 \langle x \rangle^{1-2\sigma_\ast} z^{2m} + \frac m 2 \int \bar{u}_B U_y^2 \langle x \rangle^{-2\sigma_\ast} z^{2m}. 
\end{align}
The third term from \eqref{Baleaf} is clearly bounded by $\|U \|_{X_{0,m}}^2$ by definition. We estimate the middle term from \eqref{Baleaf}, upon integration over $x$, via 
\begin{align}
|\int \int \bar{u}_{Bx} U_y^2 \langle x \rangle^{1-2\sigma_\ast} z^{2m}| \lesssim \| \frac{\bar{u}_{Bx}}{\bar{u}_B} \langle x \rangle \|_\infty \| \sqrt{\bar{u}_B} U_y \langle x \rangle^{- \sigma_\ast} z^{m} \|^2 \lesssim \| U \|_{X_{0, m}}^2,  
\end{align}
where we have invoked the decay estimate from \eqref{water:1}. 

For the second term from \eqref{honne:j:1}, we have 
\begin{align} \n
\int 2 \bar{u}_{By} U U_{xy} \langle x \rangle^{1-2\sigma_\ast} z^{2m} = &- \int 2 \bar{u}_{Byy} UU_x \langle x \rangle^{1-2\sigma_\ast} z^{2m} - \int 2 \bar{u}_{By} U U_x \langle x \rangle^{\frac 1 2 - 2 \sigma_\ast} z^{2m-1} \\
& - \int 2 \bar{u}_{By} U_y U_x \langle x \rangle^{1-2\sigma_\ast} z^{2m}.
\end{align}
Above, the first two terms are essentially estimated in an identical manner to \eqref{identical:1}, whereas the third term is estimated in an identical manner to \eqref{lights:1} - \eqref{lights:2}.
 
For the third term from \eqref{honne:j:1}, we have 
\begin{align} \n
\int \bar{u}_{Byy} q U_{xy} \langle x \rangle^{1-2\sigma_\ast} z^{2m} = & - \int \bar{u}_{Byyy} q U_x \langle x \rangle^{1-2\sigma_\ast} z^{2m} - \int \bar{u}_{Byy} UU_x \langle x \rangle^{1-2\sigma_\ast} z^{2m} \\
&- 2m \int \bar{u}_{Byy} q U_x \langle x \rangle^{\frac 1 2 - 2 \sigma_\ast}  z^{2m-1}.
\end{align}
The first term above cancels the contribution from the third term from \eqref{energetic:2}. The second term above is estimated in an identical manner to \eqref{identical:1}. For the third term above, we estimate via 
\begin{align} \n
|\int \int 2 \bar{u}_{By} U_y U_x \langle x \rangle^{1-2\sigma_\ast} z^{2m}| \lesssim &\| \bar{u}_{Byy} z^{2m-1} y \langle x \rangle^{\frac 1 2} \|_\infty \| \frac{q}{y} \langle x \rangle^{- \frac 1 2 - \sigma_\ast} \| \| U_x \langle x \rangle^{\frac 1 2 - \sigma_\ast} \|  \\
\lesssim & \| U \langle x \rangle^{- \frac 1 2 - \sigma_\ast} \| \| U_x \langle x \rangle^{\frac 1 2 - \sigma_\ast} \|. 
\end{align}

We finally have the forcing term,
\begin{align}
|\int \int F U_x \langle x \rangle^{1-2\sigma_\ast} z^{2m}| \lesssim \| F \langle x \rangle^{\frac 1 2 - \sigma_\ast} z^m \| \| U_x \langle x \rangle^{\frac 1 2-\sigma_\ast} z^m \|.
\end{align}
To conclude the proof of the lemma, we integrate \eqref{energetic:2} over $x \in [0, X_0]$ and then take the supremum over $X_0$. Upon doing so, the first term from \eqref{treL1} and the first term from \eqref{energy2222} give the positive quantities we need to control, according to \eqref{def:X12:m}. Upon using the elementary inequality $\sup_{X_0} |\int_0^{X_0} \int g \ud y \ud x| \le \int \int |g| \ud y \ud x$, we appeal to the estimation of the error terms above.   
\end{proof}

\begin{lemma} \label{lem:LP:3}   Assume the inductive hypotheses \eqref{indu:1} -- \eqref{indu:2}. Let $U$ be a solution to \eqref{Pr:af}. Then the following estimates are valid for $m = 1,...,m_{max}$
\begin{align}
\| U \|_{Y_{\frac 12, 0}}^2 \lesssim & C_\delta \| U \|_{X_{0,0}}^2 + \delta \|U \|_{X_{\frac 1 2, 0}}^2 + \| \p_y F \langle x \rangle^{1 - \sigma_\ast} \|^2 + \| \bar{u}^0_p \p_y U_0 \|_{L^2_y}^2,  \\ \n
\| U \|_{Y_{\frac 1 2, m}}^2 \lesssim &\| U \|_{X_{0,m}}^2 + \delta \|U \|_{X_{\frac 1 2, m}}^2 + C_\delta \|U \|_{X_{0,0}}^2 + \| U \|_{X_{0, m-1}}^2+ \| \p_y F \langle x \rangle^{1 - \sigma_\ast} z^m \|^2 \\
& + \| \bar{u}_B|_{x= 0} \p_y U_0 y^m \|_{L^2_y}^2.
\end{align}
\end{lemma}
\begin{proof} We apply the multiplier $U_{y} \langle x \rangle^{1-2\sigma_\ast} z^{2m}$, for $m = 0,...,m_{max}$, to the vorticity formulation of the equation, which produces 
\begin{align} \n
&\int \p_y \mathcal{T}[U] U_{y} \langle x \rangle^{1-2\sigma_\ast} z^{2m} - \int \p_y^3 u U_y \langle x \rangle^{1-2\sigma_\ast}  z^{2m} \\ \label{stinson:1}
&+ \int \p_y (\bar{u}_{Byyy} q) U_y \langle x \rangle^{1-2\sigma_\ast}  z^{2m}  = \int \p_y F U_y \langle x \rangle^{1-2\sigma_\ast}  z^{2m}. 
\end{align}

We first address the terms arising from the $\p_y \mathcal{T}[U]$ contribution from \eqref{stinson:1}. Using the definition of $\mathcal{T}[U]$ from \eqref{def:Pr:T}, this generates
\begin{align} \n
\int \p_y \mathcal{T}[U] U_y \langle x \rangle^{1-2\sigma_\ast} z^{2m} = & \int |\bar{u}_B|^2 U_{xy} U_y \langle x \rangle^{1-2\sigma_\ast} z^{2m} +2 \int \bar{u}_B \bar{u}_{By} U_{x} U_y \langle x \rangle^{1-2\sigma_\ast} z^{2m}  \\ \n
& + \int (\bar{u}_{By} \bar{v}_B + \bar{v}_{By} \bar{u}_B) U_y^2 \langle x \rangle^{1-2\sigma_\ast} z^{2m}  + \int \bar{u}_B \bar{v}_B U_{yy} U_y  \langle x \rangle^{1-2\sigma_\ast} z^{2m} \\ \label{oklo:1}
& + \int \bar{u}_{Byyy} U U_y \langle x \rangle^{1-2\sigma_\ast} z^{2m}  + \int \bar{u}_{Byy} U_y^2 \langle x \rangle^{1-2\sigma_\ast} z^{2m} . 
\end{align}
The first term from \eqref{oklo:1} is energetic, and we rewrite it as 
\begin{align} \n
\int |\bar{u}_B|^2 U_{xy} U_y \langle x \rangle^{1-2\sigma_\ast} z^{2m}  = &\frac{\p_x}{2} \int |\bar{u}_B|^2 U_y^2 \langle x \rangle^{1-2\sigma_\ast} z^{2m} - \int \bar{u}_B \bar{u}_{Bx} U_y^2 \langle x \rangle^{1-2\sigma_\ast} z^{2m} \\ \label{lewis:1}
&- \frac{1-2\sigma_\ast + m}{2} \int |\bar{u}_B|^2 U_y^2 \langle x \rangle^{-2\sigma_\ast} z^{2m}.
\end{align}
We estimate the latter two quantities above upon integrating in $x$ by $\| U \|_{X_{0,m}}^2$, upon invoking the decay $|\bar{u}_{Bx}| \lesssim \langle x \rangle^{-1}$, according to \eqref{water:1} and \eqref{indu:1}. The same can be said for the third and sixth terms from \eqref{oklo:1}. For the second term from \eqref{oklo:1}, we estimate via 
\begin{align} \n
|\int \int \bar{u}_B \bar{u}_{By} U_{x} U_y \langle x \rangle^{1-2\sigma_\ast} z^{2m} | \lesssim & \| \bar{u}_{By} \langle x \rangle^{\frac 1 2}\|_\infty \| \bar{u}_B U_x \langle x \rangle^{\frac 1 2 - \sigma_\ast} z^m \| \| U_y \langle x \rangle^{-\sigma_\ast} z^m \| \\
\le & \delta \| U \|_{X_{\frac 1 2,m}}^2 + C_\delta \| U \|_{X_{0,m}}^2.
\end{align}
For the fourth term from \eqref{oklo:1}, we estimate via 
\begin{align} \n
|\int \int \bar{u}_B \bar{v}_B U_{yy} U_yu \langle x \rangle^{1-2\sigma_\ast} z^{2m}| \lesssim &\| \bar{v}_B \langle x \rangle^{\frac 1 2} \|_\infty \| \sqrt{\bar{u}_B} U_{yy} \langle x \rangle^{\frac 1 2 - \sigma_\ast} z^m \|_\infty \| U_y \langle x \rangle^{-\sigma_\ast} z^m \| \\
\le &C_\delta \| U \|_{X_{0,m}}^2 + \delta \| U \|_{Y_{\frac 12, m}}^2.
\end{align}
Finally, for the fifth term from \eqref{oklo:1}, we estimate via 
\begin{align} \n
|\int \int \bar{u}_{Byyy} U U_y \langle x \rangle^{1-2\sigma_\ast} z^{2m}| \lesssim & \| \bar{u}_{Byyy} z^{2m} \langle x \rangle^{\frac 3 2} \| \| U \langle x \rangle^{- \frac 1 2 - \sigma_\ast} \| \| U_y \langle x \rangle^{-\sigma_\ast} \| \\
\le & C_\delta \| U \|_{X_{0,0}}^2 + \delta \| U \|_{Y_{\frac 12, 0}}^2.
\end{align}

For the second term from \eqref{stinson:1}, we integrate by parts once in $y$ to obtain 
\begin{align} \n
- \int \p_y^3 u U_y \langle x \rangle^{1-2\sigma_\ast} z^{2m} = & \int u_{yy} U_{yy} \langle x \rangle^{1-2\sigma_\ast} z^{2m} + 2m \int u_{yy} U_y \langle x \rangle^{\frac 1 2-2\sigma_\ast} z^{2m-1} \\ \label{from:from}
&+ \delta_{m =0} u_{yy} U_y \langle x \rangle^{1-2\sigma_\ast}(x, 0). 
\end{align}
We first handle the boundary contribution from above, after integration in $x$
\begin{align} \label{rih:1}
\int u_{yy} U_y \langle x \rangle^{1-2\sigma_\ast}(x, 0) \ud x = \int 3 \bar{u}_{By} U_y^2(x, 0) \langle x \rangle^{1-2\sigma_\ast} \ud x. 
\end{align}

We next address the main term from \eqref{from:from}, for which we get 
\begin{align} \n
\int u_{yy} U_{yy} \langle x \rangle^{1-2\sigma_\ast} z^{2m} = &\int \bar{u}_B U_{yy}^2 \langle x \rangle^{1-2\sigma_\ast} z^{2m} + \int 3 \bar{u}_{Byy} U U_{yy} \langle x \rangle^{1-2\sigma_\ast} z^{2m} \\ \n
&+ \int 3 \bar{u}_{By} U_y U_{yy} \langle x \rangle^{1-2\sigma_\ast} z^{2m} + \int \bar{u}_{Byyy} q U_{yy} \langle x \rangle^{1-2\sigma_\ast} z^{2m} \\ \n
= & \int \bar{u}_B U_{yy}^2 \langle x \rangle^{1-2\sigma_\ast} z^{2m} + \int 3 \bar{u}_{Byy} U U_{yy} \langle x \rangle^{1-2\sigma_\ast} z^{2m} \\ \n
&- \int \frac 3 2 \bar{u}_{Byy} U_y^2 \langle x \rangle^{1-2\sigma_\ast} z^{2m} - \int 3m \bar{u}_{By} U_y^2 \langle x \rangle^{\frac 1 2-2\sigma_\ast} z^{2m-1} \\ \label{sc:1}
& - \delta_{m = 0} \frac 3 2 \bar{u}_{By} U_y(x, 0)^2 \langle x \rangle^{1-2\sigma_\ast} +  \int \bar{u}_{Byyy} q U_{yy} \langle x \rangle^{1-2\sigma_\ast} z^{2m}.
\end{align}

We now handle the middle term from \eqref{from:from}, which produces 
\begin{align} \n
2m \int u_{yy} U_y \langle x \rangle^{\frac 1 2-2\sigma_\ast} z^{2m-1} = &2m \int \bar{u}_B U_{yy} U_y \langle x \rangle^{\frac 1 2-2\sigma_\ast} z^{2m-1} + 6m \int \bar{u}_{B yy} U U_y\langle x \rangle^{\frac 1 2-2\sigma_\ast} z^{2m-1} \\ \label{sc:2}
& + 6m \int \bar{u}_{B y} U_y^2 \langle x \rangle^{\frac 1 2 - 2\sigma_\ast} z^{2m-1} + 2m \int \bar{u}_{B yyy} q U_y \langle x \rangle^{\frac 1 2 - 2\sigma_\ast} z^{2m-1}.
\end{align}

Finally, we handle the third term from \eqref{stinson:1}, which produces 
\begin{align}
\int \p_y (\bar{u}_{Byyy}q) U_y \langle x \rangle^{1-2\sigma_\ast} z^{2m} = - \int \bar{u}_{Byyy} q U_{yy} \langle x \rangle^{1-2\sigma_\ast} z^{2m} - 2m \int \bar{u}_{Byyy} q U_y \langle x \rangle^{\frac 1 2-2\sigma_\ast} z^{2m-1}.
\end{align}
These two contributions cancel out the last terms from \eqref{sc:1} and \eqref{sc:2}. The remaining error terms from \eqref{sc:1} - \eqref{sc:2} are all easily seen to be controlled by $\| U \|_{X_{0,m}}^2 + \| U \|_{X_{0,m-1}}^2$.

We now treat the terms from the forcing, $F^{(i)}$. In the case when $m = 0$, we have upon integration over $x$, 
\begin{align} \n
|\int \int \p_y F U_y \langle x \rangle^{1 - 2 \sigma_\ast} z^m| \lesssim \| \p_y F \langle x \rangle^{1-\sigma_\ast} z^m \|  \| U_y \langle x \rangle^{-\sigma_\ast} z^m \|.
\end{align}
To conclude the proof of the lemma, we integrate the identity \eqref{stinson:1} over $x \in [0, X_0]$ and then take the supremum over $X_0$. Upon doing so, the first term from \eqref{lewis:1}, the term \eqref{rih:1} and the fifth term from \eqref{sc:1}, and the first term from \eqref{sc:1} give the positive quantities we need to control, according to \eqref{Yh0:norm}. Upon using the elementary inequality $\sup_{X_0} |\int_0^{X_0} \int g \ud y \ud x| \le \int \int |g| \ud y \ud x$, we appeal to the estimation of the error terms above.  
\end{proof}

We can successively differentiate the system in $\p_x$ and re-apply Lemmas \ref{lem:LP:1}, \ref{lem:LP:2}, \ref{lem:LP:3} with minor modifications (estimating lower order commutators) in order to give 
\begin{lemma} \label{lemma:LP4} Assume the inductive hypotheses \eqref{indu:1} -- \eqref{indu:2}. Let $U$ be a solution to \eqref{Pr:af}. Fix $k \le k_{max}+1, 0 \le m \le m_{max}$. For any $0 < \delta << 1$, the following estimates are valid:
\begin{align} \label{k:again:1}
\| U  \|_{X_{k, m}}^2 \lesssim & \| U \|_{X_{k, m-1}}^2 +  \| \bar{u}_B \p_x^k U_0 y^m \|_{L^2_y}^2 +   \| \p_x^k F \langle x \rangle^{k + \frac 1 2- \sigma_\ast} z^m \|^2  + \| U \|_{\mathcal{X}_{\le k - \frac{1}{2}, m}}, \\ \n 
\|  U  \|_{X_{k + \frac 1 2, m }}^2 \lesssim & \| U \|_{X_{k,m}}^2 + \|  U  \|_{X_{k + \frac 1 2, m -1 }}^2 +  \delta \| U  \|_{X_{k+1,0}}^2 + \delta \| U \|_{Y_{k + \frac 1 2,0}}^2 \\  \label{k:again:2}
&  + \| \p_x^k F \langle x \rangle^{k + \frac 1 2- \sigma_\ast} z^m \|^2 + \| \sqrt{\bar{u}_B} \p_y \p_x^k U_0  y^m\|_{L^2_y}^2 + C_\delta \|U \|_{\mathcal{X}_{\le k, m}},\\ \n
\|  U \|_{Y_{k + \frac 1 2, m}}^2 \lesssim & \|  U \|_{X_{k,m}}^2 + \delta \|  U  \|_{X_{k + \frac 1 2, m}}^2 + \|  U \|_{X_{k, m-1}}^2 + \| \p_x^k \p_y F \langle x \rangle^{ k + 1 - \sigma_\ast} z^m \|^2 \\  \label{k:again:3}
&+ \| \bar{u}_B \p_y \p_x^k U_0 y^m \|_{L^2_y}^2 + C_\delta \|U \|_{\mathcal{X}_{\le k, m}}.
\end{align}
\end{lemma}

\section{Nonlinear Prandtl Layer, $i = 0$} \label{NL:Section}

We will establish the following proposition on the nonlinear Prandtl equations.

\begin{proposition} \label{Prop:nonlinear}Assume the data $U^0_p(y)$ satisfies \eqref{near:blasius}. Let $\ell_0$ be as in \eqref{near:blasius}. Then there exists a unique global solution to the Prandtl system which converges asymptotically to Blasius:
\begin{align} \label{pnld1}
\| \p_x^k \p_y^j ( \bar{u}^0_p - \bar{u}_\ast) \langle z \rangle^{m_0} \|_{L^\infty_y} &\lesssim \delta_\ast \langle x \rangle^{- \frac 1 4 - k - \frac j 2 + \sigma_\ast}, \\  \label{pnld2}
\| \p_x^k \p_y^j ( \bar{v}^0_p - \bar{v}_\ast) \langle z \rangle^{m_0} \|_{L^\infty_y} &\lesssim \delta_\ast \langle x \rangle^{- \frac 3 4 - k - \frac j 2 + \sigma_\ast},
\end{align}
for $2k + j \le \ell_0$.
\end{proposition}

In order to establish this proposition, we proceed in a few steps.  We define the rescaled perturbations 
\begin{align*}
\bar{u}^0_p = \bar{u}_\ast + \delta_{\ast} \tilde{u}^0_p, \qquad \bar{v}^0_p = \bar{v}_\ast + \delta_{\ast} \tilde{v}^0_p,
\end{align*}
which upon immediately relabeling $\tilde{u}^0_p \mapsto u$ and $\tilde{v}^0_p \mapsto \bar{v}$, satisfy the following system:
\begin{align} \label{size:small:eq}
&\bar{u}_\ast \p_x u + u \p_x \bar{u}_\ast + \bar{v}_\ast \p_y u + \bar{v} \p_y \bar{u}_\ast - u_{yy} = \delta_\ast Q(u, \bar{v}),
\end{align}
with initial datum 
\begin{align}
u|_{x=0} = \frac{1}{\delta_\ast} ( \bar{U}^0_p(y) - \bar{u}_\ast(0, y) ) =: u_0(y),
\end{align}
and quadratic nonlinearity $- Q(u, \bar{v}) := u \p_x u + \bar{v} \p_y u$. We note that the initial datum $u_0(y)$ is size $1$, but the nonlinearity is size $\delta_\ast$ in \eqref{size:small:eq}. 

We now define the norms in which we control the solution to \eqref{size:small:eq}. Specifically, we control the $\mathcal{X}_{P_0}$ norm as defined below, where we take the background flow to be $\bar{u}_\ast$ in all the definitions \eqref{X00:norm} -- \eqref{Yhm:norm}:
\begin{align} \label{def:space:X:lin}
\|U \|_{\mathcal{X}_{P_0}} := & \sum_{k = 0}^{k_{0}} \sum_{m = 0}^{m_0}  \|  U \|_{X_{k, m}[\bar{u}_\ast]} +  \|  U \|_{X_{k + \frac 1 2, m}[\bar{u}_\ast]} +  \| U \|_{Y_{k + \frac 1 2, m}[\bar{u}_\ast]}  + \sum_{m = 0}^{m_0}  \| U \|_{X_{k_0 + 1, m}[\bar{u}_\ast]}, \\
\|U \|_{\mathcal{X}_{P_0, X_{\ast}}} := & \sum_{k = 0}^{k_{0}} \sum_{m = 0}^{m_0}  \|  U \|_{X_{k, m, X_{\ast}}[\bar{u}_\ast]} +  \|  U \|_{X_{k + \frac 1 2, m,  X_{\ast}}[\bar{u}_\ast]} +  \| U \|_{Y_{k + \frac 1 2, m,  X_{\ast}}[\bar{u}_\ast]}  + \sum_{m = 0}^{m_0}  \| U \|_{X_{k_0 + 1, m,  X_{\ast}}[\bar{u}_\ast]}.
\end{align}

\begin{lemma} Assume \eqref{near:blasius}, and the parabolic compatibility conditions in Definition \ref{para:compat:def}. Then, the following estimate is valid:
\begin{align} \label{hghghg:2}
 \sum_{k = 0}^{k_0} \sum_{m = 0}^{m_0}  \|  U \|_{X_{k, m}[\bar{u}_\ast]}^2 +  \|  U \|_{X_{k + \frac 1 2, m}[\bar{u}_\ast]}^2 +  \| U \|_{Y_{k + \frac 1 2, m}[\bar{u}_\ast]}^2 \lesssim C_{Data} + \delta_{\ast} P( \| U \|_{\mathcal{X}_{P_0}} ),
\end{align}
where $P(\cdot)$ is a quartic polynomial. 
\end{lemma}
\begin{proof} Motivated by \eqref{size:small:eq}, we define $[\bar{u}_B, \bar{v}_B] = [\bar{u}_\ast, \bar{v}_\ast]$ and $F := \delta_\ast Q(u, v)$. We now apply estimates \eqref{k:again:1} -- \eqref{k:again:3}, which can be applied since the background $[\bar{u}_\ast, \bar{v}_\ast]$ is trivially seen to satisfy the inductive hypotheses \eqref{indu:1} -- \eqref{indu:2} and the bootstrap hypothesis \eqref{boots:1}, as $\delta_{\ast \ast} = 0$ for this choice of background. Moreover, straightforward Sobolev embeddings establish the nonlinear estimates 
\begin{align}
\sum_{k = 0}^{k_0} \| \p_x^k Q(u, v) \langle x \rangle^{k + \frac 1 2 - \sigma_\ast} z^m \| + \| \p_x^k \p_y Q(u, v) \langle x \rangle^{k + 1 - \sigma_\ast} z^m \| \lesssim \| U \|_{\mathcal{X}_{P_0}}^2. 
\end{align}
\end{proof}

We now address the top order estimate, for which we apply Lemma \ref{lem:LP:1}, but with the choice of background flow $[\bar{u}_B, \bar{v}_B] := [\bar{u}_\ast + \delta_\ast \tilde{u}^0_p, \bar{v}_\ast + \delta_\ast \tilde{v}^0_p]$. Indeed, applying $\p_x^{k+1}$ to \eqref{size:small:eq}, we obtain the identity
\begin{align} \label{id:1:1:id}
\bar{u}_B \p_x u^{(k+1)} + v^{(k+1)} \p_y \bar{u}_B + u^{(k+1)} \p_x \bar{u}_B + \bar{v}_B \p_y u^{(k+1)} - \p_y^2 u^{(k+1)}= F_{k+1},   
\end{align}
where we define the forcing (that contains only lower order terms now)
\begin{align} \n
- F_{k+1} := & \sum_{j = 1}^{k} \binom{k+1}{j} ( \p_x^j \bar{u}_B \p_x^{k+1-j} u_{x} + \p_x^{k+1-j} \bar{u}_x \p_x^j u + \p_x^j \bar{v}_B \p_x^{k+1-j} u_y + \p_x^{k+1-j} \bar{u}_{\ast y} \p_x^j v ) \\
& + u \p_x^{k+2} \bar{u}_{\ast} + u_y \p_x^{k+1} \bar{v}_{\ast} + v \p_y \bar{u}_{\ast}^{(k+1)} + u_x \p_x^{k+1} \bar{u}_{\ast}.
\end{align}
Motivated by \eqref{id:1:1:id}, we define 
\begin{align} \label{nlgu}
\mathcal{U}_{k+1} := \p_y \{ \frac{\psi^{(k+1)}}{\bar{u}_B} \} = \frac{1}{\bar{u}_B} \Big( u^{(k+1)} - \frac{\bar{u}_{By}}{\bar{u}_B} \psi^{(k+1)} \Big).
\end{align}
Due to the norms themselves depending on the background flow, we define the modified norms:
\begin{align} \label{def:tilde:xP0}
\| U \|_{\widetilde{\mathcal{X}}_{P_0}}^2 := &  \|  \mathcal{U}_{k+1} \|_{X_{0, m}[\bar{u}_B]}^2 +  \sum_{k = 0}^{k_0} \sum_{m = 0}^{m_0}  \|  U \|_{X_{k, m}[\bar{u}_\ast]}^2 +  \|  U \|_{X_{k + \frac 1 2, m}[\bar{u}_\ast]}^2 +  \| U \|_{Y_{k + \frac 1 2, m}[\bar{u}_\ast]}^2, \\
\| U \|_{\widetilde{\mathcal{X}}_{P_0, X_{\ast}}}^2 := &  \|  \mathcal{U}_{k+1} \|_{X_{0, m, X_{\ast}}[\bar{u}_B]}^2 +  \sum_{k = 0}^{k_0} \sum_{m = 0}^{m_0}  \|  U \|_{X_{k, m, X_{\ast}}[\bar{u}_\ast]}^2 +  \|  U \|_{X_{k + \frac 1 2, m, X_{\ast}}[\bar{u}_\ast]}^2 +  \| U \|_{Y_{k + \frac 1 2, m, X_{\ast}}[\bar{u}_\ast]}^2.
\end{align}

We are now ready to prove Proposition \ref{Prop:nonlinear}.
\begin{proof}[Proof of Proposition \ref{Prop:nonlinear}]  This follows by repeating essentially identically the proof of Lemma \ref{lem:LP:1} on the system \eqref{id:1:1:id}, with $k = k_0$. More precisely, we proceed via a continuation argument. Fix any $0 < X_{\ast} < \infty$. Applying \eqref{case:case:m:fin}, we obtain the bound  
\begin{align} \n
\sum_{m = 0}^{m_0} \|  \mathcal{U}_{k_0+1} \|_{X_{0, m, X_{\ast}}[\bar{u}_B]}^2 \lesssim & C_{Data} + \sum_{m = 0}^{m_0} \| F_{k_0+1} \langle x \rangle^{k_0 + \frac 3 2 - \sigma_{\ast}} z^m \|_{L^2_x(0, X_{\ast}) L^2_y}^2 \\ \n
\lesssim & C_{Data} +  \sum_{k = 0}^{k_0} \sum_{m = 0}^{m_0}  \|  U \|_{X_{k, m}[\bar{u}_\ast]}^2 +  \|  U \|_{X_{k + \frac 1 2, m}[\bar{u}_\ast]}^2 +  \| U \|_{Y_{k + \frac 1 2, m}[\bar{u}_\ast]}^2 \\ \label{hghghg:1}
& + \delta_{\ast} P( \| U \|_{\mathcal{X}_{P_0, X_{\ast}}} ).
\end{align}
Adding together \eqref{hghghg:2} and \eqref{hghghg:1}, we obtain
\begin{align} \label{conmap:1}
\| U \|_{\widetilde{\mathcal{X}}_{P_0, X_{\ast}}}^2 \lesssim &C_{Data}  + \delta_{\ast} P( \| U \|_{\mathcal{X}_{P_0, X_{\ast}}} )
\end{align}
Consulting now the definitions  \eqref{def:space:X:lin}, \eqref{nlgu}, and \eqref{def:tilde:xP0}, it is easy to see that $\| U \|_{\mathcal{X}_{P_0, X_{\ast}}} \lesssim \| U \|_{\widetilde{\mathcal{X}}_{P_0, X_{\ast}}}$. Consulting the definition \eqref{norm:B:def} and applying a standard Sobolev interpolation gives that $\| u, v \|_{X_{B, X_{\ast}}} \lesssim \| U \|_{\mathcal{X}_{P_0, X_{\ast}}}$. Therefore, estimate \eqref{conmap:1} implies that the $\delta_{\ast \ast}^{- \frac 1 2}$ required in \eqref{boots:2} continues to hold with an improved factor of $\delta_{\ast \ast}^{- \frac 14}$. By a standard continuation argument, we can send $X_{\ast} \rightarrow \infty$ to obtain the bound 
\begin{align} \label{dogs:1}
\| U \|_{\mathcal{X}_{P_0}}^2 \lesssim &C_{Data} 
\end{align}
From here, the claimed estimates \eqref{pnld1} -- \eqref{pnld2} again follow. 
\end{proof}

\section{Intermediate Prandtl Layers, $i = 1, \dots, N_1 - 1$} \label{Section:LinPRA}

We are now ready to consider the intermediate Prandtl terms in the approximate solution. Specifically, we consider the following problem 
\begin{align} \label{eq:lin:Pr}
&\bar{u}^0_p \p_x u^i_p + \bar{u}^0_{px} u^i_p + \bar{v}^0_p \p_y u^i_p + \bar{v}^i_p \p_y \bar{u}^0_p - \p_y^2 u^i_p = F^{(i)}_p, \\
&v^i_p =  \int_y^\infty \p_x u^i_p, \qquad \bar{v}^i_p = v^i_p - v^i_p|_{y = 0},
\end{align}
where we will fix for now $i = 1,...,N_1 - 1$, as the $i = N_1$ case needs to be modified slightly, which is treated in the next section below. These equations are supplemented with the boundary conditions 
\begin{align} \label{LP:BC:1}
u^i_p|_{y = 0} = - u^i_E|_{y = 0}, \qquad u^i_p|_{y \rightarrow \infty} = 0, \qquad u^i_p|_{x = 0} = U^i_p(y). 
\end{align}
The forcing $F_p^{(i)}$ in \eqref{eq:lin:Pr} is defined in \eqref{Fpi}. We recall the hypothesis on the initial datum, $U^i_p(y)$, is given by \eqref{hyp:1}. In addition, we will assume inductively that the boundary layer profiles for $j = 0, ..., i-1$ and the Euler profiles for $j = 0, ..., i$ have already been constructed to satisfy the estimates in Theorem \ref{thm:approx}. Our main proposition here is that the estimates in Theorem \ref{thm:approx} continue to hold for $j = i$, thereby verifying the induction. 
\begin{proposition} \label{prop:PrLin} Given initial datum $U^i_p(y)$ which satisfies \eqref{hyp:1} and the parabolic compatibility conditions in Definition \ref{para:compat:def}, there exists a unique solution to \eqref{eq:lin:Pr} for all $x > 0$, and moreover which satisfies the following estimates, for $2k + j \le \ell_i$,  
\begin{align} \label{water:4}
&\| \langle z \rangle^{m_i} \p_x^k \p_y^j u_p^{i} \|_{L^\infty_y} \lesssim \langle x \rangle^{- \frac 1 4 - k - \frac j 2 + \sigma_\ast} \\ \label{water:5}
&\| \langle z \rangle^{m_i} \p_x^k \p_y^j v_p^{i} \|_{L^\infty_y} \lesssim \langle x \rangle^{- \frac 3 4 - k - \frac j 2 + \sigma_\ast}.
\end{align}
\end{proposition}

\subsection{Homogenizations}

In this subsection, we will manipulate \eqref{eq:lin:Pr} so that the results from Section \ref{section:PL} can be used. First, we  homogenize the boundary condition at $Y = 0$ in \eqref{LP:BC:1} by introducing the following quantities
\begin{align} \label{def:def:LP}
u^i := & u^i_p - \phi(z) u^i_p(x, 0) = u^i_p + \phi(z) u^i_E(x, 0), \\ \label{def:def:LP2}
 \bar{v}^i := & \bar{v}^i_p - u^i_{Ex}(x, 0) \langle x \rangle^{\frac 1 2} \int_0^z \phi(w) \ud w + \frac{1}{2\langle x \rangle^{\frac 1 2}} u^i_E(x, 0) \int_0^z w \phi'(w) \ud w, \\ \label{def:def:LP:v}
 v^i := &\bar{v}^i - \bar{v}^i(x, \infty). 
\end{align}
$\phi(\cdot)$ is a $C^\infty$ function satisfying the following properties: $\phi(0) = 1$, $\phi$ is supported on $[0, 1)$, and $\int_0^\infty \phi(z) \ud z = 0$. First, we note that the pair defined above is divergence-free. Second, we note that the integration by parts identity 
\begin{align}
\int_0^\infty w \phi'(w) \ud w = - \int_0^\infty \phi(w) \ud w + w \phi(w)|_{w = \infty} - w\phi(w)|_{w = 0} = 0, 
\end{align}
by definition of $\phi$ guarantees that $\bar{v}^i|_{y = \infty} = \bar{v}^i_p|_{y = \infty}$. 

We now take the notational convention of dropping the superscript $i$ for this section, as $i$ will be fixed, so: 
\begin{align*}
[u, v] := [u^i, v^i]. 
\end{align*}
By the specification in \eqref{def:def:LP}, we get the homogenized system for $[u, v]$, which we will analyze,
\begin{align} \label{eq:lin:Pr:hom}
&\bar{u}^0_p u_x + \bar{u}^0_{px} u + \bar{v}^0_p \p_y u + \bar{v} \p_y \bar{u}^0_p - u_{yy} = F^{(i)}, \\
&u_x + v_y = 0, \\
&u|_{y = 0} = \bar{v}|_{y = 0} = u|_{y = \infty}  = 0, \\ \label{eq:compat:initial:1}
&u|_{x = 0} = U^i_p(y) + \phi(y) u^i_E(0, 0) =: u_0(y).
\end{align}
Above, the new forcing is defined to be $F^{(i)} = F^{(i)}_p + H^{(i)}_p$, where $H^{(i)}_p$ contains those terms coming from the homogenization process, and is specifically given by
\begin{align}\n
H^{(i)}_p := &\phi(z) \bar{u}^0_p u^i_{Ex}(x, 0) + \phi(z) \bar{u}^0_{px} u^i_E(x, 0) + \langle x \rangle^{- \frac 1 2} \bar{v}^0_p \phi'(z) u^i_E(x, 0)  \\ \n
& - \bar{u}^0_{py} u^i_{Ex}(x, 0) \langle x \rangle^{\frac 1 2} \int_0^z \phi(w) \ud w + \bar{u}^0_{py}\frac{1}{2 \langle x \rangle^{\frac 1 2}} u^i_E(x, 0) \int_0^z w \phi'(w) \ud w \\ \label{def:Hip}
&- \frac{1}{2\langle x \rangle} \bar{u}^0_p \phi'(z) u_E(x, 0) - \frac{1}{\langle x \rangle} \phi''(z) u_E(x, 0).  
\end{align}

We now note that \eqref{eq:lin:Pr:hom} -- \eqref{eq:compat:initial:1} is of the form \eqref{eq:abs:1:a} -- \eqref{eq:abs:1:d}. In this case, we use as the background profile $[\bar{u}_B, \bar{v}_B] := [\bar{u}^0_p, \bar{v}^0_p]$.

\subsection{Estimates on the Forcing}

We will need estimates on the forcing, $F^{(i)}$.
\begin{lemma} The forcing $F^{(i)} = F_p^{(i)} + H_p^{(i)}$ satisfies the following estimates  
\begin{align} \label{forcing:est:PR:2} 
\| (y \p_y)^k (x \p_x)^j F^{(i)} \langle z \rangle^M \|_{L^\infty_y} &\lesssim \langle x \rangle^{- \frac{23}{16}}  \\ \label{forcing:est:Pr:lin}
\| (y \p_y)^k (x \p_x)^j F^{(i)} \langle z \rangle^M \|_{L^2_y} &\lesssim \langle x \rangle^{- \frac{19}{16}},
\end{align}
for $M \le m_i$, $k + 2j \le \ell_i$. 
\end{lemma}
\begin{proof} We first estimate the terms in $H_p^{(i)}$, which has been defined in \eqref{def:Hip}. First, due to the localization in $z$, we have immediately that $\| H^{(i)}_p \|_{L^2_y} \lesssim \langle x \rangle^{\frac 1 4} \| H^{(i)}_p \|_{L^\infty_y}$. 
\begin{align} \n
\|H_p^{(i)}\|_{L^\infty_y} \lesssim &|u^i_{Ex}(x, 0)| + \|\bar{u}^0_{px}\|_{L^\infty_y} |u^i_{E}(x, 0)| + \|\bar{v}^0_p\|_{L^\infty_y} |u^i_E(x, 0)| \langle x \rangle^{- \frac 1 2} \\ \n
&+ \| \bar{u}^0_{py} \|_\infty |u^i_{Ex}(x, 0)| \langle x \rangle^{\frac 1 2} + \langle x \rangle^{- \frac 1 2} \| \bar{u}^0_{py} \|_\infty |u^i_E(x, 0)| \\  \n
& + \frac{1}{\langle x \rangle} |u_E(x, 0)| + |u_{Ex}(x, 0)|\\
\lesssim & \langle x \rangle^{- \frac 3 2} + \langle x \rangle^{- \frac 3 2} + \langle x \rangle^{-\frac 3 2} + \langle x \rangle^{-\frac 3 2} + \langle x \rangle^{-\frac 3 2} +  \langle x \rangle^{-\frac 3 2} + \langle x \rangle^{-\frac 3 2}  \lesssim \langle x \rangle^{-\frac 3 2},
\end{align}
where we have invoked estimates \eqref{blas:conv:1} - \eqref{blas:conv:2} for the estimates on $\bar{u}^0_p, \bar{v}^0_p$, and \eqref{water:78} - \eqref{water:88} for estimates on $u^i_E$ (and derivatives thereof).

We now address the terms in $F_p^{(i)}$, which has been defined in \eqref{Fpi}. First, according to the definition \eqref{Fpi:cal}, we have 
\begin{align} \n
\| \eps^{- \frac i 2} \mathcal{F}_p^{(i)} \|_{L^2_y} \lesssim & \sum_{j = 1}^{i}   \| \p_x u_p^{i-1} \|_{L^2_y} \| u^j_E \|_{L^\infty_y} + \|u^i_E\|_{L^\infty_y} \sum_{k = 0}^{i-2} \| \p_x u^k_p \|_{L^2_y} + \| u^{i-1}_p \|_{L^2_y} \sum_{j = 1}^{i} \| \p_x u^j_E\|_{L^\infty_y} \\ \n
& + \| \p_x u^i_E \|_\infty \sum_{j = 0}^{i-2} \| u^j_p \|_{L^2_y} + \| u^{i-1}_p \|_{L^2_y} \sum_{k = 1}^{i-1}  \| \p_x u^k_p \|_{L^\infty_y} + \| \p_x u^{i-1}_p \|_{L^2_y} \sum_{k = 1}^{i-2} \| u^k_p \|_{L^\infty_y} \\
\lesssim & \langle x \rangle^{- \frac 5 4},
\end{align}
where we have invoked estimates \eqref{blas:conv:1} - \eqref{water:65}, which are assumed inductively to hold for the $0$ through $i-1$'th boundary layers, as well as estimates \eqref{water:88} which is also assumed inductively to hold up to the $i$'th Euler layer.

We now arrive at the terms 
\begin{align} \n
\| \eps^{- \frac i 2} \mathcal{G}_p^{(i)} \|_{L^2_y} \lesssim & \| v^{i-1}_p \|_{L^2_y} \sum_{k = 1}^{i}  \| u^k_{eY} \|_{L^\infty_y} +\|  u^i_{eY} \|_{L^\infty_y} \sum_{k = 0}^{i-2}  \| v^k_p \|_{L^2_y} +  \| \bar{v}^{i-1}_p \|_{L^\infty_y} \sum_{k = 1}^{i-1}  \| u^k_{py} \|_{L^2_y} \\ \n
&+ \|  u^{i-1}_{py} \|_{L^2_y} \sum_{k = 1}^{i - 2} \| \bar{v}^k_p \|_{L^\infty_y}+ \sum_{j = 1}^{i} \|  u^{i-1}_{py} \bar{v}^j_E \|_{L^2_y} + \eps^{- \frac 1 2}  \sum_{k = 0}^{i-2} \| \bar{v}^i_E u^k_{py} \|_{L^2_y} \\ \n
\lesssim & \langle x \rangle^{- \frac 7 4} +  \langle x \rangle^{- \frac 7 4} + \delta_{i \ge 2} \sum_{k = 1}^{i-1} \langle x \rangle^{- \frac 5 4 + 2\sigma_{\ast}} + \delta_{i \ge 3} \sum_{k = 1}^{i-2} \langle x \rangle^{- \frac 5 4 + 2\sigma_{\ast}}  + \langle x \rangle^{- \frac 5 4} \\
\lesssim & \langle x \rangle^{- \frac{19}{16}},
\end{align}
where we have used the inductive estimates \eqref{blas:conv:1} - \eqref{water:88}, as well as the estimate 
\begin{align}
\| u^0_{py} \bar{v}^i_E \|_{L^2_y} =  \| u^0_{py} (\int_0^Y \bar{v}^i_{EY} \ud Y') \|_{L^2_y} \lesssim \sqrt{\eps} \| y u^0_{py} \|_{L^2_y} \| v^i_{EY} \|_{L^\infty_Y} \lesssim \sqrt{\eps} \langle x \rangle^{- \frac 5 4}, 
\end{align}
where we have invoked estimate \eqref{blas:conv:1} on $u^0_{py}$ and the inductively assumed estimate \eqref{water:78} on $v^i_{EY}$.

We now, in consultation with \eqref{Fpi} and \eqref{def:PPi}, need to estimate the term 
\begin{align} \n
\| \eps^{- \frac i 2} \p_x P^{(i)}_p \|_{L^2_y} &\le  \eps^{- \frac i 2 + 1} \| \int_y^\infty \p_x \mathcal{H}^{(i)}_p \|_{L^2_y} + \eps^{- \frac i 2 + 1} \| \int_y^\infty \p_x \mathcal{J}^{(i)}_p \|_{L^2_y} + \eps^{\frac 1 2} \| \int_y^\infty \p_x \Delta_\eps v^{i-1}_p \|_{L^2_y} \\ \n
&\lesssim_\kappa \eps^{- \frac i 2 + 1} \| \langle y \rangle^{\frac 3 2 + \kappa} \p_x \mathcal{H}^{(i)}_p \|_{L^\infty_y} + \eps^{- \frac i 2 + 1} \| \langle y \rangle^{\frac 3 2 + \kappa} \p_x \mathcal{J}^{(i)}_p \|_{L^\infty_y} + \eps^{\frac 1 2} \| \langle y \rangle^{\frac 3 2 + \kappa} \p_x \Delta_\eps v^{i-1}_p \|_{L^\infty_y} \\ \n
& \lesssim_\kappa  \eps^{- \frac i 2 + 1} \langle x \rangle^{\frac 3 4 + \frac \kappa 2} \| \langle z \rangle^{\frac 3 2 + \kappa} \p_x \mathcal{H}^{(i)}_p \|_{L^\infty_y} +  \eps^{- \frac i 2 + 1}  \langle x \rangle^{\frac 3 4 + \frac \kappa 2} \| \langle z \rangle^{\frac 3 2 + \kappa} \p_x \mathcal{J}^{(i)}_p \|_{L^\infty_y} \\ \n
&\qquad + \eps^{\frac 1 2}  \langle x \rangle^{\frac 3 4 + \frac \kappa 2} \| \langle z \rangle^{\frac 3 2 + \kappa} \p_x \Delta_\eps v^{i-1}_p \|_{L^\infty_y} \\
&\lesssim_\kappa \sqrt{\eps} \langle x \rangle^{- \frac 7 4 + \frac \kappa 2},
\end{align}
for a small but fixed $\kappa > 0$. Above, we have, upon consulting the definitions \eqref{Hpi} and \eqref{def:JPi}, used the estimate 
\begin{align}
\|\langle z \rangle^{\frac 3 2 + \kappa} \p_x \mathcal{H}_p^{(i)} \|_{L^\infty_y} + \| \langle z \rangle^{\frac 3 2 + \kappa} \p_x \mathcal{J}_p^{(i)} \|_{L^\infty_y} \lesssim \sqrt{\eps} \langle x \rangle^{- \frac 5 2}.
\end{align}

Consulting \eqref{Fpi}, the final term we need to estimate is $\eps^{\frac 1 2}\| u^{i-1}_{pxx} \|_{L^2_y} \lesssim \eps^{\frac 1 2} \langle x \rangle^{- \frac 7 4}$, where we have again used the inductive estimates \eqref{blas:conv:1} - \eqref{water:54} as well as \eqref{water:1}. 

The $L^\infty_y$ estimates work in a nearly identical manner, with an extra factor of $\langle x \rangle^{- \frac 1 4}$ relative to the $L^2_y$ estimates appearing. 
\end{proof}

\subsection{Embeddings}

We now state embedding estimates that are valid for the above norms.  
\begin{lemma} \label{abc:1} The following estimate is valid, for $0 \le j \le \frac{\ell_i}{2}$,
\begin{align} \label{emb:start:1}
\| U^{(j)} \langle z \rangle^{m_i} \|_{L^\infty_y} \lesssim \| U \|_{\mathcal{X}_P} \langle x \rangle^{-j - \frac 1 4 - \sigma_\ast}.
\end{align}
\end{lemma}
\begin{proof} We integrate
\begin{align}\label{ffe:1}
|U^{(j)}|^2 \langle x \rangle^{2j + \frac 1 2 - 2 \sigma_\ast} \langle z \rangle^{2m_i} \le & | \int_y^\infty 2 U^{(j)} U^{(j)}_y \langle x \rangle^{2j+\frac 1 2 - 2 \sigma_\ast} \langle z \rangle^{2m_i}| + 2m_i |\int_y^\infty |U^{(j)}|^2 \langle x \rangle^{2j - 2 \sigma_\ast} \langle z \rangle^{2m_i-1}|.
\end{align}
We integrate the first quantity from \eqref{ffe:1} now in $x$ to get 
\begin{align} \n
 | \int_y^\infty 2 U^{(j)} U^{(j)}_y \langle x \rangle^{2j+\frac 1 2 - 2 \sigma_\ast} \langle z \rangle^{2m_i}|  \lesssim &  \| U^{(j)}_x \langle x \rangle^{j + \frac 1 2.- \sigma_\ast} \langle z \rangle^{m_i} \| \| U^{(j)}_y \langle x \rangle^{j - \sigma_\ast} \langle z \rangle^{m_i} \|  \\ \n
 &+ \| U^{(j)} \langle x \rangle^{j - \frac 1 2 - \sigma_\ast} \langle z \rangle^{m_i} \| \| U^{(j+1)}_y \langle x \rangle^{j+1 - \sigma_\ast} \langle z \rangle^{m_i} \|  \\ 
 &+  \| U^{(j)} \langle x \rangle^{j - \frac 1 2 - \sigma_\ast} \langle z \rangle^{m_i} \| \| U^{(j)}_y \langle x \rangle^{j - \sigma_\ast} \langle z \rangle^{m_i} \| \lesssim \| U \|_{\mathcal{X}}^2,
\end{align}
We integrate the second quantity from \eqref{ffe:1} to get 
\begin{align} \n
|\int_y^\infty |U^{(j)}|^2\langle x \rangle^{2j - 2 \sigma_\ast} \langle z \rangle^{2m_i-1}| \lesssim & \| U^{(j)} \langle x \rangle^{j - \frac 1 2 - \sigma_\ast} \langle z \rangle^{m_i} \| \| U^{(j+1)} \langle x \rangle^{j + \frac 1 2 - \sigma_\ast} \langle z \rangle^{m_i-1} \| \\
+& \| U^{(j)} \langle x \rangle^{j - \frac 1 2 - \sigma_\ast} \langle z \rangle^{m_i} \|  \| U^{(j)} \langle x \rangle^{j - \frac 1 2 - \sigma_\ast} \langle z \rangle^{m_i-1} \| \lesssim \|U \|_{\mathcal{X}_P}^2,
\end{align}
where we have invoked estimate \ref{Hardy:1:here}.
\end{proof}

We now use the estimates in Lemma \ref{abc:1} to recover estimates on $u$ and $v$ (and derivatives thereof) that are required from Proposition \ref{prop:PrLin}.
\begin{lemma} \label{lemma:gE:1} Let $[u^i_p, v^i_p]$ satisfy equation \eqref{eq:lin:Pr}. Then the following estimates are valid 
\begin{align}  \label{L2:1:a}
& \| \langle z \rangle^{m_i} \p_x^k \p_y^j u^{i}_p \|_{L^2_y} \lesssim \langle x \rangle^{-k - \frac j2 + \sigma_\ast} (\|U \|_{\mathcal{X}_P} + 1)\\ \label{L2:1:b}
& \| \langle z \rangle^{m_i} \p_x^k \p_y^j v^{i}_p \|_{L^2_y} \lesssim \langle x \rangle^{- \frac 1 2-k - \frac j2 + \sigma_\ast} (\|U \|_{\mathcal{X}_P} + 1)\\ \label{hgj:1}
&\| \langle z \rangle^{m_i}\p_x^k \p_y^j u_p^{i} \|_{L^\infty_y} \lesssim \langle x \rangle^{- \frac 1 4 - k - \frac j 2 + \sigma_\ast} ( \|U \|_{\mathcal{X}_P} + 1) \\ \label{hgj:2}
&\| \langle z \rangle^{m_i} \p_x^k \p_y^j v_p^{i} \|_{L^\infty_y} \lesssim \langle x \rangle^{- \frac 3 4 - k - \frac j 2 + \sigma_\ast} ( \| U \|_{\mathcal{X}_P} + 1),
\end{align}
for $2k + j \le \ell_i$. 
\end{lemma}
\begin{proof} We will establish \eqref{L2:1:a} - \eqref{hgj:2} for the quantities $[u, v]$. To translate back to $[u^i_p, v^i_p]$, we simply appeal to \eqref{def:def:LP} - \eqref{def:def:LP2}, and the corresponding estimates for $u^i_E(x, 0)$ from \eqref{water:88}. 

We first set $j = k = 0$ and address \eqref{L2:1:a} and \eqref{hgj:1}. In this case, we use the formula 
\begin{align} \n
\| u \langle x \rangle^{\frac 1 4 - \sigma_\ast} \langle z \rangle^{m_i}\|_{L^\infty_y} \le &  \| \bar{u}^0_p U \langle x \rangle^{\frac 1 4 - \sigma_\ast} \langle z \rangle^{m_i}\|_{L^\infty_y} + \| \bar{u}^0_{py} y \langle z \rangle^{m_i}\|_{L^\infty} \| \frac{q}{y} \langle x \rangle^{\frac 1 4 - \sigma_\ast} \|_{L^\infty_y} \\ \label{jko1}
\lesssim & \| U \langle x \rangle^{\frac 14 - \sigma_\ast }\|_{L^\infty_y} \lesssim \| U \|_{\mathcal{X}_P},
\end{align} 
where we have invoked estimates \eqref{water:1} and \eqref{blas:conv:1}, for the final step we use \eqref{emb:start:1}. The proof of the $L^2$ bound, \eqref{L2:1:a}, follows by replacing $L^\infty_y$ by $L^2_y$ in the estimate \eqref{jko1}. 

For $v$, we only need to treat the case when $j = 0$, as when $j \ge 1$ we use the divergence-free condition to appeal to \eqref{hgj:1}. The $k > 0$ case works in an analogous manner to $k = 0$, which we now demonstrate. For \eqref{L2:1:b}, we use the standard Hardy inequality 
\begin{align}
\| v \|_{L^2_y} \lesssim \| y u_x \|_{L^2_y} \lesssim \langle x \rangle^{\frac 1 2} \| \langle z \rangle u_x \|_{L^2_y} \lesssim \langle x \rangle^{\frac 1 2} \langle x \rangle^{-  1 + \sigma_\ast} \| U \|_{\mathcal{X}_P}.
\end{align}
To address \eqref{hgj:2}, we may perform a standard Sobolev embedding 
\begin{align}
\| v \|_{L^\infty_y} \lesssim \| v \|_{L^2_y}^{\frac 1 2} \| v_y \|_{L^2_y}^{\frac 1 2} \lesssim  (\langle x \rangle^{- \frac 1 2 + \sigma_\ast} \| U \|_{\mathcal{X}_P} )^{\frac 1 2} (\langle x \rangle^{- 1 + \sigma_\ast} \| U \|_{\mathcal{X}_P} )^{\frac 1 2} \lesssim \langle x \rangle^{- \frac 3 4 + \sigma_\ast} \|U \|_{\mathcal{X}_P}.
\end{align}

Next, we move to the $k = 0, j = 1$ case from \eqref{L2:1:a}. For this, we simply expand via 
\begin{align} \n
\| u_y  \langle z \rangle^{m_i} \langle x \rangle^{\frac 1 2 - \sigma_\ast} \|_{L^2_y} \lesssim &\| \bar{u}^0_p U_y \langle z \rangle^{m_i} \langle x \rangle^{\frac 1 2 - \sigma_\ast} \|_{L^2_y} + \| \bar{u}^0_{py} U \langle z \rangle^{m_i} \langle x \rangle^{\frac 1 2 - \sigma_\ast} \|_{L^2_y} \\
\lesssim & \| \bar{u}^0_p U_y \langle z \rangle^{m_i} \langle x \rangle^{\frac 1 2 - \sigma_\ast} \|_{L^2_y} +\|  U \langle z \rangle^{m_i} \langle x \rangle^{ - \sigma_\ast} \|_{L^2_y}  \lesssim \|U \|_{\mathcal{X}}.
\end{align}

Next, we move to the $k= 0, j = 2$ case from \eqref{L2:1:a}. For this, we simply use the equation \eqref{eq:lin:Pr:hom} to obtain 
\begin{align} \n
\| u_{yy} \langle z \rangle^{m_i} \langle x \rangle^{1-\sigma_\ast}\|_{L^2_y} \lesssim & \| F^{(i)} \langle z \rangle^{m_i} \langle x \rangle^{1-\sigma_\ast} \|_{L^2_y} + \| \bar{u}^0_p u_x \langle z \rangle^{m_i} \langle x \rangle^{1-\sigma_\ast}\|_{L^2_y} + \| \bar{u}^0_{px} u \langle z \rangle^{m_i} \langle x \rangle^{1-\sigma_\ast}\|_{L^2_y} \\ \n
&+ \| \bar{v}^0_p u_y \langle z \rangle^{m_i} \langle x \rangle^{1-\sigma_\ast}\|_{L^2_y} + \| \bar{v} \bar{u}^0_{py}\langle z \rangle^{m_i} \langle x \rangle^{1-\sigma_\ast}\|_{L^2_y} \lesssim  1 + \| U \|_{\mathcal{X}_P},
\end{align}
where we have invoked the estimate on the forcing, \eqref{forcing:est:Pr:lin}. 

Next, we move to the $k = 0, j = 1$ case from \eqref{hgj:1}, which follows from a straightforward Sobolev embedding,
\begin{align} \n
\| u_y \langle z \rangle^{m_i} \langle x \rangle^{\frac 3 4 - \sigma_\ast} \|_{L^\infty_y} \lesssim  &\| u_y \langle z \rangle^{m_i} \langle x \rangle^{\frac 12 - \sigma_\ast} \|_{L^2_y}^{\frac 1 2} \| u_{yy} \langle z \rangle^{m_i} \langle x \rangle^{1 - \sigma_\ast} \|_{L^2_y}^{\frac 1 2} \\ \n
&+ \| u_y \langle z \rangle^{m_i} \langle x \rangle^{\frac 12 - \sigma_\ast} \|_{L^2_y}^{\frac 1 2} \| u_y \langle z \rangle^{m_i-1} \langle x \rangle^{\frac 12 - \sigma_\ast} \|_{L^2_y}^{\frac 1 2} \lesssim  1 + \| U \|_{\mathcal{X}_P}.
\end{align}
Next, we jump to the the $k = 0, j = 2$ case from \eqref{hgj:1}. For this, we use the equation \eqref{eq:lin:Pr:hom} to rewrite 
\begin{align} \n
\| u_{yy} \langle z \rangle^{m_i} \|_{L^\infty_y} \lesssim &\| F^{(i)} \langle z \rangle^{m_i} \|_{L^\infty_y} + \| \bar{u}^0_p u_x \langle z \rangle^{m_i} \|_{L^\infty_y} + \| \bar{u}^0_{px} u \langle z \rangle^{m_i} \|_{L^\infty_y} \\ \n
&+ \| \bar{v}^0_p \p_y u \langle z \rangle^{m_i} \|_{L^\infty_y} + \| \bar{v} \bar{u}^0_{py} \langle z \rangle^{m_i} \|_{L^\infty_y} 
\\  \n
\lesssim & \langle x \rangle^{- \frac{23}{16}} + \langle x \rangle^{- \frac 5 4 + \sigma_\ast} \| u_x \langle x \rangle^{\frac 5 4 - \sigma_\ast} \langle z \rangle^{m_i} \|_{L^\infty_y} + \| \bar{u}^0_{px} \langle x \rangle \langle z \rangle^{m_i} \|_\infty  \| u \langle x \rangle^{\frac 1 4 - \sigma_\ast} \langle z \rangle^{m_i} \|_{L^\infty_y} \langle x \rangle^{- \frac 5 4 + \sigma_\ast} \\ \n
& + \| \bar{v}^0_p \langle x \rangle^{\frac 1 2} \|_\infty \| u_y \langle z \rangle^{m_i} \langle x \rangle^{\frac 3 4 - \sigma_\ast} \|_{L^\infty_y}  \langle x \rangle^{- \frac 5 4 + \sigma_\ast}+ \| \bar{v} \langle x \rangle^{\frac 3 4 - \sigma_\ast} \|_{L^\infty_y} \| \bar{u}^0_{py} \langle x \rangle^{\frac 1 2} \|_{L^\infty_y}  \langle x \rangle^{- \frac 5 4  + \sigma_\ast} \\ \label{ajaj}
\lesssim & \langle x \rangle^{- \frac{23}{16}} + \langle x \rangle^{- \frac 5 4 + \sigma_\ast} \|U \|_{\mathcal{X}_P},
\end{align}
where we have invoked the estimate on the forcing, \eqref{forcing:est:PR:2}. Higher order $x$ derivatives work by commuting with $x \p_x$, whereas higher order $y$ derivatives work by invoking the equation as in \eqref{ajaj}.
\end{proof}

\subsection{Closing the Induction}

We will now need initial data estimates
\begin{lemma} Let $(\p_x^k U)_0:= \p_x^k U|_{x = 0}$. Assume the compatibility condition on $U^i_p$ from Definition \ref{para:compat:def}, and estimate \eqref{hyp:1}. Then the following estimates are valid, 
\begin{align} \label{est:IDLK}
\| \bar{u}^0_p (\p_x^k U)_0 \langle y \rangle^M \|_{L^2_y} +  \| \sqrt{\bar{u}^0_p} \p_y (\p_x^k U)_0 \langle y \rangle^M \|_{L^2_y}^2 \lesssim  1.
\end{align}
\end{lemma}
\begin{proof} According to the definition \eqref{eq:compat:initial:1}, and due to the compatibility condition placed on the datum, $U^i_p(\cdot)$, \eqref{compat:0}, we have the compatibility $u|_{x = 0}(0) = U^i_p(0) + \phi(0) u^i_E(0, 0) = 0 = u|_{y = 0}(0)$. Therefore, $U_0(y) = \frac{u|_{x = 0}(y)}{\bar{u}^0_p} - \bar{u}^0_{py} \frac{\psi|_{x = 0}(y)}{|\bar{u}^0_p|^2}$, is a sufficiently smooth, rapidly decaying function, satisfying the estimates \eqref{est:IDLK} due to the corresponding hypothesis \eqref{hyp:1}. For higher order $x$-derivatives, we derive the initial datum  by evaluating equation \eqref{Pr:af} at $\{x = 0\}$. For the first derivative, we obtain $U_x|_{x = 0} = \frac{F^{(i)} + \p_y^2 u(0, \cdot)}{|\bar{u}^0_p|^2} - \frac{\bar{u}^0_p \bar{v}^0_p}{|\bar{u}^0_p|^2} U_0' - \frac{2 \bar{u}^0_{pyy}}{|\bar{u}^0_p|^2} U_0$. We now note that the quotient, $\frac{F^{(i)} + \p_y^2 u(0, \cdot)}{|\bar{u}^0_p|^2}$, is bounded due to the compatibility condition from Definition \eqref{para:compat:def}. Higher order derivatives follow in an analogous manner, so we omit the details. 
\end{proof}

\begin{proof}[Proof of Proposition \ref{prop:PrLin}] The result follows immediately from our \textit{a-priori} estimates, Lemmas \ref{lem:LP:1}, \ref{lem:LP:2}, \ref{lem:LP:3}, and \ref{lemma:LP4}, coupled with the initial datum estimates, \eqref{est:IDLK} and the forcing estimates, \eqref{forcing:est:Pr:lin}, which provides the \textit{a-priori} estimate $\|U \|_{\mathcal{X}_P} \lesssim 1$. When applying Lemmas \ref{lem:LP:1}, \ref{lem:LP:2}, \ref{lem:LP:3}, and \ref{lemma:LP4}, we have made the choice $[k_{max}, \ell_{max}, m_{max}] = [k_i, \ell_i, m_i]$. We have also invoked that, due to \eqref{pnld1} - \eqref{pnld2}, we take $[\bar{u}_B, \bar{v}_B] := [\bar{u}_\ast, \bar{v}_{\ast}] + \delta_{\ast}[\tilde{u}^0_p, \tilde{v}^0_p] = [\bar{u}^0_p, \bar{v}^0_p]$, and the strong inductive hypotheses, \eqref{indu:1} -- \eqref{indu:2} hold according to Proposition \ref{Prop:nonlinear}.  We subsequently use Lemma \ref{lemma:gE:1} to obtain the estimates required by Proposition \ref{prop:PrLin}. It is then a standard matter to repeat all of the above estimates in the framework of a Galerkin method to obtain existence and uniqueness.
\end{proof} 

\section{Final Prandtl Layer, $i = N_1$} \label{cutoff:SEC}

In this subsection we perform the cut-off argument that is required for the final Prandtl layer. A cut-off is required for the boundary layer, $[u^{N_1}_p, v^{N_1}_p]$, as we need to ensure the boundary conditions $v^{N_1}_p(x, 0) = v^{N_1}_p(x, \infty) = 0$ simultaneously. We first define the auxiliary profiles, $[u_p, v_p]$ through the following system,
\begin{align} \label{eq:lin:Pr:N1}
&\bar{u}^0_p \p_x u_p + \bar{u}^0_{px} u_p + \bar{v}^0_p \p_y u_p + \bar{v}_p \p_y \bar{u}^0_p - \p_y^2 u_p = F^{(N_1)}_p, \\
&\bar{v}_p =  \int_0^y \p_x u_p, 
\end{align}
with initial and boundary conditions 
\begin{align}
u_p|_{x = 0} = U_{p}^{N_1}(y), \qquad u_p|_{y = 0} = - u^{N_1}_E(x, 0), \qquad u_p|_{y = \infty} = 0, 
\end{align}

We subsequently define $[u_p^{N_1}, v_p^{N_1}]$ by cutting off $[u_p, v_p]$ in a divergence-free manner via 
\begin{align}
u_p^{N_1} := & \chi(\sqrt{\eps}\frac{y}{\langle x \rangle}) u_p - \int_x^\infty \frac{\sqrt{\eps}}{\langle x' \rangle} \frac{y}{\langle x' \rangle} \chi'(\frac{\sqrt{\eps} y}{x'}) u_p \ud x' - \int_x^\infty \frac{\sqrt{\eps}}{\langle x' \rangle} \frac{1}{\langle x' \rangle} \chi'(\frac{\sqrt{\eps}y}{\langle x' \rangle}) \bar{v}_p \ud x' , \\
\bar{v}_p^{N_1} := & \chi(\sqrt{\eps} \frac{y}{\langle x \rangle}) \bar{v}_p.
\end{align}
As we only need to verify the boundary conditions $v^{N_1}_p(x, 0) = v^{N_1}_p(x, \infty) =0$, we have the freedom to choose the scale of the cut-off function. In this particular case, we have selected to cut-off at scale $\sqrt{\eps}\frac{y}{\langle x \rangle}$ for convenience, but other choices are certainly possible. For convenience of notation, we define 
\begin{align}
u_{\chi} := - \int_x^\infty \frac{\sqrt{\eps}}{\langle x' \rangle} \frac{y}{\langle x' \rangle} \chi'(\frac{\sqrt{\eps} y}{x'}) u_p \ud x' - \int_x^\infty \frac{\sqrt{\eps}}{\langle x' \rangle} \frac{1}{\langle x' \rangle} \chi'(\frac{\sqrt{\eps}y}{\langle x' \rangle}) \bar{v}_p \ud x',
\end{align}
so that $u_p^{N_1} =\chi(\sqrt{\eps}\frac{y}{\langle x \rangle}) u_p + u_{\chi}$. 

Due to the presence of cutting off $[u_p, v_p]$, this creates an additional output error which we denote by $F_{\chi}$, and define as  
\begin{align} \n
F_{\chi} :=&- \eps^{\frac{N_1}{2}} \Big(\bar{u}^0_p \frac{1}{\langle x \rangle} \frac{\sqrt{\eps}y}{\langle x \rangle} \chi'(\frac{\sqrt{\eps}y}{\langle x \rangle}) \bar{v}_p +  \bar{u}^0_{px} u_{\chi}  + \bar{v}^0_p \frac{\sqrt{\eps}}{\langle x \rangle} \chi'(\frac{\sqrt{\eps}y}{\langle x \rangle}) u_p + \bar{v}^0_p \p_y u_\chi \\ \label{def:F:chi}
&- \frac{\eps}{\langle x \rangle^2} \chi''(\frac{\sqrt{\eps}y}{\langle x \rangle}) u_p - 2 \frac{\sqrt{\eps}}{\langle x \rangle} \chi'(\frac{\sqrt{\eps}y}{\langle x \rangle}) u_{py} - \p_y^2 u_\chi + (1-\chi(\frac{\sqrt{\eps}y}{\langle x \rangle})) F^{(N_1)}_p \Big) 
\end{align}

We now estimate the output forcing, $F_R, G_R$, which is defined through \eqref{forcing:remainder}, 
\begin{lemma} \label{lemma:ref:2} Let $0 \le j \le 11$, and $k = 0$ or $1$. Then, 
\begin{align} \label{est:forcings:part1}
\| \p_x^j \p_y^k F_R \langle x \rangle^{\frac{11}{20}+ j + \frac k 2} \| + \sqrt{\eps} \|  \p_x^j \p_y^k G_R \langle x \rangle^{\frac{11}{20}+ j + \frac k 2} \| \le \eps^{5}. 
\end{align}
\end{lemma}
\begin{proof} Consulting definition \eqref{forcing:remainder}, we have that $G_R = 0$, and so we just need to estimate $F_R$. Moreover, we have $F_R = F^{N_1 + 1} + \eps^{- \frac{N_2}{2}} F_\chi$, where $F^{N_1 + 1}$ is defined in \eqref{def:star}. First of all, by our choice of parameters $N_1, N_2$, we have in a nearly identical fashion to \eqref{forcing:est:PR:2} - \eqref{forcing:est:Pr:lin} the estimate
\begin{align}
\| (y \p_y)^k (x \p_x)^j F^{(N_1 + 1)}  \langle z \rangle^M \|_{L^2_y} \lesssim \eps^{\frac{N_1 - N_2}{2}} \langle x \rangle^{- \frac{19}{16}}. 
\end{align}
We now need to estimate the terms arising from \eqref{def:F:chi}. Indeed, 
\begin{align} \n
\| \eps^{- \frac{N_2}{2}} F_{\chi} \|_{L^2_y} \lesssim & \eps^{\frac{N_1 - N_2}{2}} \Big( \frac{1}{\langle x \rangle^{\frac 1 2}} \| \bar{v}_p \|_{L^\infty_y} + \frac{1}{\langle x \rangle^{\frac 3 4}} \| u_\chi \|_{L^\infty_y} + \langle x \rangle^{- 1} \| u_p \|_{L^\infty_y} + \langle x \rangle^{- \frac 1 2} \| \p_y u_\chi \|_{L^2_y}  \\  \n
& + \frac{\eps^{\frac 3 2}}{\langle x \rangle^{\frac 3 2}} \| u_p \|_{L^\infty_y} + \frac{1}{\langle x \rangle^{\frac 1 2}} \| u_{py} \|_{L^\infty_y} + \| \p_y^2 u_\chi \|_{L^2_y} + \| F_p^{(N_1)} \|_{L^2_y} \Big) \\
\lesssim &  \eps^{\frac{N_1 - N_2}{2}}  \langle x \rangle^{- \frac 5 4 + \sigma_\ast}.
\end{align}
Higher order $\p_x$ and $\p_y$ derivatives work in the same manner. This establishes \eqref{est:forcings:part1}.
\end{proof}

We now establish our main result: 
\begin{proof}[Proof of Theorem \ref{thm:approx}] This follows from Propositions \ref{Prop:nonlinear}, \ref{prop:PrLin}, \ref{prop:Euler}, and Lemma \ref{lemma:ref:2}.
\end{proof}

\section{Background profiles $\bar{u}, \bar{v}$} \label{subsection:background}

We recall the definition of $[\bar{u}, \bar{v}]$ from \eqref{exp:base}. In addition, for a few of the estimates in our analysis, we will require slightly more detailed information on these background profiles, in the form of decomposing into an Euler and Prandtl component. Indeed, define 
\begin{align} \label{split:split:1}
\bar{u}_P &:= \bar{u}^0_P + \sum_{i = 1}^{N_1} \eps^{\frac i 2} u^i_P  , &  \bar{u}_E &:=  \sum_{i = 1}^{N_{1}} \eps^{\frac{i}{2}}  u^i_E.
\end{align}
We will now summarize the quantitative estimates on $[\bar{u}, \bar{v}]$. 
\begin{lemma} \label{lemma:cons:1} For $0 \le j,m,k, M, l \le 20$, the following estimates are valid
\begin{align} \label{prof:u:est}
 \| \p_x^j (y\p_y)^m \p_y^k \bar{u} x^{j + \frac k 2} \|_\infty + \| \frac{1}{\bar{u}} \p_x^j \bar{u} x^j  \|_{\infty}  \le & C_{k,j} , \\ \label{prof:v:est}
 \| \p_x^j (y\p_y)^m \p_y^k \bar{v} x^{j + \frac k 2 + \frac 1 2}  \|_\infty   + \| \frac{1}{\bar{u}} \p_x^j \bar{v} x^{j + \frac 1 2}  \|_\infty  \le & C_{k,j}, \\ \label{est:Eul:piece}
\eps^{- \frac 1 2}\| \p_x^j (Y \p_Y)^l \p_Y^k \bar{u}_E \langle x \rangle^{j+k + \frac 1 2} \|_\infty + \| \p_x^j (Y \p_Y)^l \p_Y^k \bar{v}_E \langle x \rangle^{j+k + \frac 1 2} \|_\infty \le & C_{k,j} , \\ \label{est:Pr:piece}
\| \p_x^j (y\p_y)^k \p_y^l \bar{u}_P \langle z \rangle^M  \langle x \rangle^{j + \frac l 2}\|_\infty   \le & C_{k,j,M} \\ \label{est:PR:bar:v}
\| \p_x^j (y\p_y)^k \p_y^l \bar{v}_P \langle z \rangle^M  \langle x \rangle^{j + \frac l 2+ \frac 1 2}\|_\infty   \le & C_{k,j,M}.
\end{align}
\end{lemma}
\begin{proof} This is proven by combining estimates \eqref{water:1}, \eqref{blas:conv:1}, \eqref{blas:conv:2}, \eqref{water:4}, \eqref{water:5}, \eqref{water:7}, \eqref{water:8}.
\end{proof}

We will need estimates which amount to showing that $\bar{u}$ remains a small perturbation of $\bar{u}^0_p$ for all $x$. 
\begin{lemma} \label{lemma:bofz} Define a monotonic function $b(z) := \begin{cases} z \text{ for } 0 \le z \le \frac{3}{4} \\ 1 \text{ for } 1 \le z  \end{cases}$, where $b \in C^\infty$. Then 
\begin{align} \label{est:ring:1}
&\| \p_y^j \p_x^k (\bar{u} - \bar{u}^0_p) \langle x \rangle^{\frac j 2 + k + \frac{1}{50}} \|_{\infty} \le C_{k,j} \sqrt{\eps}, \\ \label{samezies:1}
&1 \lesssim \frac{\bar{u}^0_p}{b(z)}  \lesssim 1 \text{ and } 1 \lesssim \Big| \frac{\bar{u}}{\bar{u}^0_p} \Big| \lesssim 1, \\ \label{prime:pos}
&|\bar{u}_y|_{y = 0}(x)| \gtrsim \langle x \rangle^{- \frac 1 2}. 
\end{align}
\end{lemma}
\begin{proof} For estimate \eqref{est:ring:1}, we simply appeal to the definition \eqref{exp:base} and to the corresponding estimates \eqref{water:65} - \eqref{water:88}, and subsequently the fact that $- \frac 14 + \sigma_\ast < - \frac{1}{50}$ by our choice of $\sigma_\ast$ in Theorem \ref{thm:approx}.  

We now move to the first assertion in \eqref{samezies:1}. For the upper bound $|\bar{u}^0_p| \lesssim z$,we have 
\begin{align}
|\frac{\bar{u}^0_p}{z}| \le & |\frac{\bar{u}_\ast}{z}|  +  |\frac{\bar{u}^0_p - \bar{u}_\ast}{z}| \lesssim \sqrt{x} \| \p_y \bar{u}_\ast \|_{L^\infty_y} + \sqrt{x} \| \p_y (\bar{u}^0_p - \bar{u}_\ast) \|_{L^\infty_y} \lesssim 1 + \delta_\ast \langle x \rangle^{- \frac 1 4 + \sigma_\ast}, 
\end{align}
where we have invoked \eqref{water:1} and \eqref{blas:conv:1}. For the lower bound $\bar{u}^0_p \gtrsim z$ for $0 \le z \le 1$, we appeal to the elementary algebraic identity 
\begin{align}
\frac{z}{\bar{u}^0_p} = \Big( \frac{1}{1- \frac{\bar{u}^0_p - \bar{u}_\ast}{\bar{u}_\ast}} \Big) \frac{z}{\bar{u}_\ast}, 
\end{align}
after which we invoke that $\bar{u}_\ast \gtrsim z$ for $0 \le z \le 1$, and the estimate 
\begin{align}
|\frac{\bar{u}^0_p - \bar{u}_\ast}{\bar{u}_\ast}| \lesssim \delta_\ast \langle x \rangle^{- \frac 1 4 + \sigma_\ast} \text{ for } 0 \le z \le 1, 
\end{align}

For the second assertion in \eqref{samezies:1}, we note that it suffices to establish $1 \lesssim |\frac{\bar{u}}{b(z)}| \lesssim 1$. In addition, this means that we can localize to $0 \le z \le 1$. Using \eqref{exp:base}
\begin{align} \n
\Big|\frac{\bar{u}}{\bar{u}^0_p} - 1\Big| \lesssim &\sum_{i = 1}^{N_1} \eps^{\frac i 2} \| \frac{u^i_E + u^i_p}{b(z)} \|_\infty \lesssim \sum_{i = 1}^{N_1} \eps^{\frac i 2} ( \| u^i_E + u^i_p \|_\infty + \sqrt{x} \| \frac{u^i_E + u^i_p}{y} \|_\infty )\\ \lesssim & \sqrt{\eps} ( \langle x \rangle^{- \frac 1 2} + \langle x \rangle^{- \frac 1 4 + \sigma_\ast} ).
\end{align}

We now arrive at estimate \eqref{prime:pos}. For this, we have 
\begin{align} \n
\bar{u}_y(x, 0) = & \p_y \bar{u}_{\ast}(x, 0) + \p_y (\bar{u}^0_p - \bar{u}_{\ast})(x, 0) + \sum_{i = 1}^{N_1} \eps^{\frac i 2} (\sqrt{\eps} u^i_{EY}(x, 0) + u^i_{Py}(x, 0)) \\
 \gtrsim &  \langle x \rangle^{- \frac 1 2} - \delta_\ast \langle x \rangle^{- \frac 3 4 + \sigma_\ast} + \sum_{i =1}^{N_1} \eps^{\frac i 2} (\sqrt{\eps} \langle x \rangle^{-\frac 32} - \langle x \rangle^{- \frac 3 4 + \sigma_\ast} ) \gtrsim \langle x \rangle^{- \frac 1 2}, 
\end{align}
where we have used the first term above satisfies $\p_y \bar{u}_\ast(x, 0) \gtrsim \langle x \rangle^{- \frac 1 2}$ by properties of the Blasius profile, \eqref{Blas:prop:1}, as well as the estimates \eqref{blas:conv:1}, \eqref{water:88}, and \eqref{water:65}. 
\end{proof}

We will need to remember the equations satisfied by the approximate solutions, $[\bar{u}, \bar{v}]$, which we state in the form of a lemma. 
\begin{lemma}  \label{lemma:cons:2} The following identity holds
\begin{align} \label{expression:1}
\bar{u} \bar{u}_x + \bar{v} \bar{u}_y = \bar{u}^0_{pyy} + \zeta,
\end{align}
where $\zeta$ is defined by 
\begin{align} \label{def:zeta}
\zeta :=  \sum_{i = 1}^{N_1} (\eps^{\frac i 2} u^{i}_{Ex} + \mathcal{F}_E^{(i)} + \mathcal{G}_E^{(i)}) + \sum_{i = 1}^{N_1} \eps^{\frac i 2} u^i_{pyy} + \sum_{i = 1}^{N_1} \eps^{\frac i 2} F^{(i)}_p + \mathcal{E}_{N_1}^{(1)}  + \mathcal{E}_{N_1}^{(2)}  + \mathcal{F}_{E}^{N_1 + 1} + \mathcal{G}_E^{(N_1 + 1)}.
\end{align}
\end{lemma}
\begin{proof} We appeal to the identities \eqref{UUx} and \eqref{VVy} developed in the Appendix, which verify the identity \eqref{expression:1}.
\end{proof}

We now define the auxiliary quantity 
\begin{align} \label{def:alpha}
&\alpha := \bar{u} \bar{v}_x + \bar{v} \bar{v}_y, 
\end{align}
We record relevant estimates on these quantities now. 
\begin{lemma}  \label{lemma:cons:3} For any $j, k, m \ge 0$,
\begin{align} \label{S:0}
|(x\p_x)^k (y\p_y)^m \zeta|  \lesssim &\sqrt{\eps} \langle x \rangle^{- (1 + \frac{1}{50})}  \\ \label{est:zeta:2}
|(x \p_x)^k (x^{\frac 1 2} \p_y)^j \zeta| \lesssim & \sqrt{\eps} \langle x \rangle^{- (1 + \frac{1}{50})}  \\ \label{S:1}
|(x \p_x)^k (y \p_y)^m \alpha| \lesssim & \bar{u} \langle x \rangle^{-\frac 3 2}, 
\end{align}
\end{lemma}
\begin{proof} This proof follows upon combining \eqref{def:zeta} and the estimates \eqref{prof:u:est} - \eqref{est:PR:bar:v}. 
\end{proof}

\appendix 

\section{Derivation of Approximate Solution Equations} \label{app:A}

In this section, we perform a detailed derivation of all equations satisfied by the various terms in expansion \eqref{exp:base}. 

\subsection{Asymptotic Expansions}

We first calculate the quadratic expression
\begin{align} \n
U^\eps \p_x U^\eps = & \Big( \bar{u}^0_p + \sum_{i = 1}^{N_1} \eps^{\frac i 2} (u^i_E + u^i_p) + \eps^{\frac{N_2}{2}} u \Big) \Big( \bar{u}^0_{px} + \sum_{j = 1}^{N_1} \eps^{\frac j 2} (\p_x u^j_E + \p_x u^j_p) + \eps^{\frac{N_2}{2}} u_x \Big) \\ \n
= &\bar{u}^0_p \bar{u}^0_{px} + \sum_{i = 1}^{N_1} \Big( \eps^{\frac i 2} \Big( \bar{u}^0_p \p_x u^i_p + \bar{u}^0_{px} u^i_p\Big) +  \mathcal{F}^{(i)}_p \Big)   + \sum_{i = 1}^{N_1} \Big( \eps^{\frac i 2} \p_x u^i_E + \mathcal{F}_E^{(i)} \Big)   \\ \label{UUx}
& + \eps^{\frac{N_2}{2}} \Big( \bar{u} \p_x u  + u \p_x \bar{u} \Big) + \eps^{N_2} uu_x + \mathcal{E}_{N_1}^{(1)} + \mathcal{F}_{E}^{N_1 + 1},  
\end{align}
where we have denoted, 

\begin{align} \n
\mathcal{F}^{(i)}_p := & \eps^{\frac{i-1}{2}} \p_x u_p^{i-1} \sum_{j = 1}^{i} \eps^{\frac j 2} u^j_E  + \eps^{\frac i 2} u^i_E \sum_{k = 0}^{i-2} \eps^{\frac k 2} \p_x u^k_p + \eps^{\frac{i-1}{2}} u^{i-1}_p \sum_{j = 1}^{i} \eps^{\frac j 2} \p_x u^j_E \\ \label{Fpi:cal}
& + \eps^{\frac i 2} \p_x u^i_E \sum_{j = 0}^{i-2} \eps^{\frac j 2} u^j_p + \eps^{\frac{i-1}{2}} u^{i-1}_p \sum_{k = 1}^{i-1} \eps^{\frac k 2} \p_x u^k_p + \eps^{\frac{i-1}{2}} \p_x u^{i-1}_p \sum_{k = 1}^{i-2} \eps^{\frac k 2} u^k_p
\end{align}
\begin{align} \label{FEi}
\mathcal{F}^{(i)}_E := & \eps^{\frac{i-1}{2}} u^{i-1}_E \sum_{j = 1}^{i-1} \eps^{\frac j 2} \p_x u^j_E + \eps^{\frac{i-1}{2}} \p_x u^{i-1}_E \sum_{j = 1}^{i-2} \eps^{\frac j 2} u^j_E, \\ \n
\mathcal{E}_{N_1}^{(1)} := & \eps^{\frac{N_1}{2}} \p_x u^{N_1}_p \sum_{j = 1}^{N_1} \eps^{\frac j 2} u^j_E + \eps^{\frac{N_1}{2}} u^{N_1}_p \sum_{j = 1}^{N_1} \eps^{\frac j 2} \p_x u^j_E + \eps^{\frac{N_1}{2}} u^{N_1}_p \sum_{k = 1}^{N_1} \eps^{\frac k 2} \p_x u^k_p \\ \label{EN1}
& + \eps^{\frac{N_1}{2}} \p_x u^{N_1}_p \sum_{k = 1}^{N_1 - 1} \eps^{\frac k 2 } u^k_p. 
\end{align}

We next calculate 
\begin{align} \n
V^\eps \p_y U^\eps = & \Big( (\bar{v}^0_p + \bar{v}^1_e) + \sum_{i = 1}^{N_1 - 1} \eps^{\frac i 2} (v^i_p + v^{i+1}_E) + \eps^{\frac{N_1}{2}} v^{N_1}_p + \eps^{\frac{N_2}{2}}v \Big) \\ \n
& \times \Big( \bar{u}_y + \sum_{i = 1}^{N_1} \eps^{\frac{i}{2}} (\sqrt{\eps} u^i_{EY} + u^i_{py}) + \eps^{\frac{N_2}{2}} u_y  \Big) \\ \n
= & \bar{v}^0_p \bar{u}^0_{py} + \sum_{i = 1}^{N_1} \Big( \eps^{\frac i 2} \Big( \bar{v}^0_p u^i_{py} + \bar{u}^0_{py} \bar{v}^i_p \Big) +  \mathcal{G}_p^{(i)} \Big)   +  \sum_{i = 1}^{N_1} \mathcal{G}_E^{(i)}  \\ \label{VVy}
& + \mathcal{E}_{N_1}^{(2)} + \mathcal{G}_{E}^{(N_1+1)}+  \eps^{\frac{N_2}{2}} \Big( \bar{v} \p_y u  + \p_y \bar{u} v \Big) + \eps^{N_2} v u_y
\end{align}
where we have defined 
\begin{align} \n
\mathcal{G}_p^{(i)} := & \eps^{\frac{i-1}{2}} v^{i-1}_p \sum_{k = 1}^{i} \eps^{\frac{k+1}{2}} u^k_{eY} + \eps^{\frac{i+1}{2}} u^i_{eY} \sum_{k = 0}^{i-2} \eps^{\frac{k}{2}} v^k_p + \eps^{\frac{i-1}{2}} \bar{v}^{i-1}_p \sum_{k = 1}^{i-1} \eps^{\frac k 2} u^k_{py} \\ \label{Gpi}
&+ \eps^{\frac{i-1}{2}} u^{i-1}_{py} \sum_{k = 1}^{i - 2} \eps^{\frac k 2} \bar{v}^k_p + \eps^{\frac{i-1}{2}} u^{i-1}_{py} \sum_{j = 1}^{i} \eps^{\frac{j-1}{2}} \bar{v}^j_E + \eps^{\frac{i-1}{2}} \bar{v}^i_E \sum_{k = 0}^{i-2} \eps^{\frac k 2} u^k_{py}, \\ \label{Gei}
\mathcal{G}_E^{(i)} := &\eps^{\frac{i-2}{2}} v^{i-1}_E \sum_{k = 1}^{i-1} \eps^{\frac{k+1}{2}} u^k_{eY} + \eps^{\frac i 2} u^{i-1}_{eY} \sum_{j = 1}^{i-2} \eps^{\frac{j-1}{2}} v^j_E, \\ \n
\mathcal{E}_{N_1}^{(2)} := & \eps^{\frac{N_1}{2}} v^{N_1}_p \sum_{k = 1}^{N_1} \eps^{\frac{k+1}{2}} u^k_{eY} + \eps^{\frac{N_1}{2}} \bar{v}^{N_1}_p \sum_{k = 1}^{N_1} \eps^{\frac k 2} u^k_{py} + \eps^{\frac{N_1}{2}} u^{N_1}_{py} \sum_{k = 1}^{N_1 - 1}\eps^{\frac k 2} \bar{v}^k_p \\ \label{EN2}
& + \eps^{\frac{N_1}{2}} u^{N_1}_{py} \sum_{k = 1}^{N_1} \eps^{\frac{k-1}{2}} \bar{v}^k_E. 
\end{align}

We now calculate 
\begin{align} \label{UVx}
U^\eps V^\eps_x = & \sum_{i = 1}^{N_1} \eps^{\frac{i-1}{2}} v^i_{Ex} + \sum_{i = 1}^{N_1 + 1} \mathcal{H}_E^{(i)} + \sum_{i = 1}^{N_1 + 1} \mathcal{H}_p^{(i)},  
\end{align}
where 
\begin{align} \label{Hei}
\mathcal{H}_E^{(i)} := &\eps^{\frac{i-2}{2}} v^{i-1}_{Ex} \sum_{j = 1}^{i-1} \eps^{\frac j 2} u^j_E + \eps^{\frac{i-1}{2}} u^{i-1}_E \sum_{k = 1}^{i-2} \eps^{\frac{k-1}{2}} v^k_{Ex}, \\ \n
\mathcal{H}_p^{(i)} := & \eps^{\frac{i-1}{2}} u^{i-1}_p \sum_{k = 0}^{i-1} \eps^{\frac k 2} v^k_{px} + \eps^{\frac{i-1}{2}} v^{i-1}_{px} \sum_{j = 0}^{i-2} \eps^{\frac j 2} u^j_p + \eps^{\frac{i-1}{2}} u_p^{i-1} \sum_{k = 1}^{i} \eps^{\frac{k-1}{2}} v^k_{Ex} \\ \label{Hpi}
& + \eps^{\frac{i-1}{2}} v^{i}_{Ex} \sum_{j = 0}^{i-2} \eps^{\frac j 2} u^j_p + \eps^{\frac{i}{2}} u_E^{i} \sum_{k = 0}^{i-1} \eps^{\frac k 2} v^k_{px} + \eps^{\frac{i-1}{2}} v^{i-1}_{px} \sum_{k = 0}^{i-1} \eps^{\frac k 2} u^k_E, 
\end{align}
and by abuse of notation, for $N_1 + 1$ we set
\begin{align} \n
\mathcal{H}_p^{(N_1 + 1)} := & \eps^{\frac{N_1}{2}} u^{N_1}_p \sum_{k = 0}^{N_1} \eps^{\frac k 2} v^k_{px} + \eps^{\frac{N_1}{2}} v^{N_1}_{px} \sum_{k = 0}^{N_1 - 1} \eps^{\frac k 2} u^k_p + \eps^{\frac{N_1}{2}} u^{N_1}_p \sum_{k = 1}^{N_1} \eps^{\frac{k-1}{2}} v^k_{Ex} \\ \label{HpN}
& + \eps^{\frac{N_1}{2}} v^{N_1}_{px} \sum_{k = 0}^{N_1} \eps^{\frac k 2} u^k_E. 
\end{align}

We now calculate 
\begin{align} \label{VVyreal}
V^\eps V^\eps_y = \sum_{i = 1}^{N_1 + 1} \mathcal{J}_E^{(i)} + \sum_{i = 1}^{N_1 + 1} \mathcal{J}_{P}^{(i)}
\end{align}
where
\begin{align}
\mathcal{J}_{E}^{i} := & \eps^{\frac{i-2}{2}} v^{i-1}_E \sum_{k = 1}^{i -1} \eps^{\frac k 2} v^k_{EY} + \eps^{\frac{i-1}{2}} v^{i-1}_{EY} \sum_{k = 1}^{i-2} \eps^{\frac{k-1}{2}} v^k_E ,\\ \n
\mathcal{J}_{P}^{i} := & \eps^{\frac{i-1}{2}} v^i_E \sum_{k = 0}^{i-1} \eps^{\frac k 2} v^k_{py} + \eps^{\frac{i-1}{2}} v^{i-1}_{py} \sum_{k = 1}^{i-1} \eps^{\frac{k-1}{2}} v^k_E + \eps^{\frac{i-1}{2}} v^{i-1}_p \sum_{k = 1}^{i} \eps^{\frac k 2} v^k_{EY} \\ \label{def:JPi}
& + \eps^{\frac i 2} v^i_{EY} \sum_{k = 0}^{i-2} \eps^{\frac k 2} v^k_p + \eps^{\frac{i-1}{2}} v^{i-1}_p \sum_{k = 0}^{i-1} \eps^{\frac k 2} v^k_{py} + \eps^{\frac{i-1}{2}} v^{i-1}_{py} \sum_{k = 0}^{i-2} \eps^{\frac k 2} v^k_p. 
\end{align}
and for the case of $N_1 + 1$, we set 
\begin{align} \n
\mathcal{J}_{P}^{N_1 + 1} := & \eps^{\frac{N_1}{2}} v^{N_1}_{py} \sum_{k = 1}^{N_1} \eps^{\frac{k-1}{2}} v_E^k + \eps^{\frac{N_1}{2}} v^{N_1}_p \sum_{k = 1}^{N_1} \eps^{\frac k 2} v^k_{EY} + \eps^{\frac{N_1}{2}} v^{N_1}_p \sum_{k = 0}^{N_1} \eps^{\frac k 2} v^k_{py} \\
& + \eps^{\frac{N_1}{2}} v^{N_1}_{py} \sum_{k = 0}^{N_1 - 1} \eps^{\frac k 2} v^k_p. 
\end{align}

We also calculate 
\begin{align} \label{visc:1}
&\Delta_\eps U^\eps := \sum_{k = 0}^{N_1} \eps^{\frac k 2} \Delta_\eps u^k_p + \sum_{k = 0}^{N_1} \eps^{\frac k 2} \Delta u^k_E = \sum_{k = 0}^{N_1} \eps^{\frac k 2} \Delta_\eps u^k_p, \\ \label{visc:2}
&\Delta_\eps V^\eps := \sum_{k = 0}^{N_1} \eps^{\frac k 2} \Delta_\eps v^k_p + \sum_{k = 0}^{N_1} \eps^{\frac k 2} \Delta v^k_E = \sum_{k = 0}^{N_1} \eps^{\frac k 2} \Delta_\eps v^k_p
\end{align}
where above we have used that $\Delta u^k_E = 0$ and $\Delta v^k_E = 0$ according to \eqref{Eul:harm} below to simplify these expressions. 

We now consolidate these identities to obtain the expansion of the Navier-Stokes equations. Starting with equation \eqref{eq:NS:1}, we obtain 
\begin{align} \n
&U^\eps U^\eps_x + V^\eps U^\eps_y + P^\eps_x - \Delta_\eps U^\eps\\ \n
 = & (\bar{u}^0_p \bar{u}^0_{px} +  \bar{v}^0_p \bar{u}^0_{py} - \bar{u}^0_{pyy} + P^0_{px}) + \sum_{i = 1}^{N_1} \Big( \eps^{\frac i 2} \Big( \bar{u}^0_p \p_x u^i_p + \bar{u}^0_{px} u^i_p +  \bar{v}^0_p u^i_{py} + \bar{u}^0_{py} \bar{v}^i_p + P^{i}_{px} - u^i_{pyy} \Big) +  \mathcal{F}^{(i)}_p  + \mathcal{G}_p^{(i)} \Big)   \\ \n
&+ \sum_{i = 1}^{N_1} \Big( \eps^{\frac i 2} (\p_x u^i_E + \p_x P^i_E) + \mathcal{F}_E^{(i)} + \mathcal{G}_E^{(i)} \Big)    + \eps^{\frac{N_2}{2}} \Big( \bar{u} \p_x u  + u \p_x \bar{u} + \bar{v} \p_y u  + \p_y \bar{u} v + P_x - \Delta_\eps u\Big) \\  \label{exp:fully:done:1}
&+ \eps^{N_2}( uu_x + vu_y) + \mathcal{E}_{N_1}^{(1)} +  \mathcal{E}_{N_1}^{(2)} + \mathcal{F}_{E}^{N_1 + 1} +  \mathcal{G}_{E}^{(N_1+1)} - \sum_{i = 0}^{N_1} \eps^{1 + \frac i 2} u^{i}_{pxx} + \eps^{\frac{N_1 + 1}{2}} P^{N_1+1}_{px},
\end{align}
where we have inserted the pressure expansions from \eqref{exp:base}, as well as used identities \eqref{UUx}, \eqref{VVy}, and \eqref{visc:1}.

Next, we use identities \eqref{UVx}, \eqref{VVyreal}, and \eqref{visc:2} to obtain the following expansion of the equation \eqref{eq:NS:2}, 
\begin{align} \n
&U^\eps V^\eps_x + V^\eps V^\eps_y + \frac{P^\eps_y}{\eps} - \Delta_\eps V^\eps \\ \n
= & \sum_{i = 1}^{N_1} \eps^{\frac{i-1}{2}} (v^i_{Ex} + P^i_{EY} ) + \sum_{i = 1}^{N_1 + 1} (\mathcal{H}_E^{(i)} + \mathcal{J}_E^{(i)} ) + \sum_{i = 1}^{N_1 + 1}( \mathcal{H}_P^{(i)}+ \mathcal{J}_{P}^{(i)} )-  \sum_{i = 0}^{N_1} \eps^{\frac i 2} \Delta_\eps v^i_p \\ \n
&+   \sum_{i = 0}^{N_1+1} \eps^{\frac{i}{2} - 1} P^i_{py} + \eps^{\frac{N_2}{2}-1} P_y + \eps^{\frac{N_2}{2}} (\bar{u} v_x + \bar{v}_x u + \bar{v} v_y + \bar{v}_y v - \Delta_\eps v) \\ \label{exp:fully:done:2}
& + \eps^{\frac{N_2}{2}} (uv_x + vv_y). 
\end{align}

\subsection{Euler Layers}

Our outer Euler flow selection is 
\begin{align} \label{Eul:1:0}
[u^0_E, v^0_E, P^0_E] = [1, 0, 0]. 
\end{align}
We now consider the equations obtained by setting the first summand on the second line of \eqref{exp:fully:done:1} equal to zero and the first two summands on the first line of \eqref{exp:fully:done:2} equal to zero: 
\begin{align} \label{CR:pre}
&\p_x u^i_E + \p_x P^i_E = \eps^{- \frac i 2} ( \mathcal{F}_E^{(i)} + \mathcal{G}_E^{(i)} ), \\
& \p_x v^i_E + \p_Y P^i_E = \eps^{- \frac{i-1}{2}} (  \mathcal{H}_E^{(i)} + \mathcal{J}_E^{(i)}), \\
& \p_x u^i_E + \p_Y v^i_E = 0.  
\end{align}

We will make the following inductive hypothesis on the Euler layers, which will be verified to hold: we take each Euler layer to satisfy the Cauchy-Riemann equations, 
\begin{align} \label{induct:eul}
\p_x v^{j}_E = \p_Y u^j_E, \qquad \p_x u^j_E = - \p_Y v^j_E, \text{ for } j = 0,...,i-1.
\end{align}
It is clear that the leading order, \eqref{Eul:1:0}, verifies this hypothesis. 

We will decompose 
\begin{align}
P^i_E = \hat{P}^i_E + \mathring{P}^i_E, 
\end{align}
where $\hat{P}^i_E$ will be used to cancel the quadratic terms on the right-hand sides above. More specifically, we have the following lemma.
\begin{lemma} Define
\begin{align} \label{P:hat}
\hat{P}^i_E := \frac 1 2 \eps^{\frac{i-2}{2}} ( |u^{i-1}_E|^2 + |v^{i-1}_E|^2 )+ \sum_{j = 1}^{i-2} \eps^{\frac{j-1}{2}} (u_E^{i-1} u_E^j + v_E^{i-1} v^j_E)
\end{align}
Then, assuming \eqref{induct:eul}, the following identity holds
\begin{align}
\begin{pmatrix} \p_x \\ \p_Y \end{pmatrix} \hat{P}^i_E + \begin{pmatrix} \eps^{- \frac i 2} ( \mathcal{F}_E^{(i)} + \mathcal{G}_E^{(i)} ) \\ \eps^{- \frac{i-1}{2}}( \mathcal{H}_E^{(i)} + \mathcal{J}_E^{(i)}  ) \end{pmatrix} = 0
\end{align}
\end{lemma}
\begin{proof} This follows from a direct computation. 
\end{proof}

We then get to use $\mathring{P}^i_E$ to obtain the Cauchy-Riemann equations from \eqref{CR:pre}, via 
\begin{align}
\p_x u^i_E + \p_x \mathring{P}^i_E =0, \qquad \p_x v^i_E + \p_Y \mathring{P}^i_E = 0, \qquad \p_x u^i_E + \p_Y v^i_E = 0,
\end{align}
and therefore, setting $\mathring{P}^i_E = - u^i_E$, we get 
\begin{align} \label{Eul:CR}
\p_x v^i_E = \p_Y u^i_E, \qquad \p_x u^i_E = - \p_Y v^i_E \text{ on } \mathcal{Q} = (0, \infty) \times (0, \infty), 
\end{align}
which are the Cauchy-Riemann equations, and in particular which imply 
\begin{align} \label{Eul:harm}
\Delta u^i_E = 0, \qquad \Delta v^i_E =0.
\end{align}
These equations are supplemented with boundary conditions as well as compatibility conditions, which we now describe. By enforcing the Dirichlet condition $(v^0_p + v^1_E)|_{Y = 0} = 0$ at the leading order of \eqref{exp:base}, and similarly for subsequent orders, we get the boundary condition 
\begin{align}
v^i_E|_{Y = 0} = - v^{i-1}_p|_{x = 0}, 
\end{align}
where the quantity on the right-hand side above is inductively known. We also get to enforce a boundary condition at $Y = \infty$, which we again take to be the Dirichlet condition, $v^i_E|_{Y = \infty} = 0$. 

We now describe the class of datum $V^i_E(Y) := v^i_E|_{x = 0}$ that we will prescribe. Instead of prescribing $V^i_E(\cdot)$ directly, we choose to make the following definition: 
\begin{definition} \label{defn:compatible:harmonic} Define the function $V^i_E(\cdot)$ to satisfy the ``elliptic compatibility condition" if it can be realized in the following manner. Consider the function $v^{i-1}_p(x): \mathbb{R}_+ \rightarrow \mathbb{R}$. Consider any smooth extension of $v^{i-1}_p$ to $\mathbb{R}$, which we call $\tilde{v}^{i-1}_p(x)$, which we take to be compactly supported in the negative $x$ direction. Let $\tilde{v}^i_E$ be the harmonic extension of $\tilde{v}^{i-1}_p$ to $\mathbb{H}$. $V^i_E(\cdot)$ satisfies the ``elliptic compatibility condition" if it can be realized under this procedure, that is $V^i_E(Y) = \tilde{v}^i_E(0, Y)$, for some extension $\tilde{v}^i_E$. 
\end{definition}

That is, the freedom in prescribing the datum $V^i_E$ is not as explicit as selecting the function itself, but rather the freedom is in picking the extension $\tilde{v}^i_{p-1}$ (which simply needs to be smooth and rapidly decaying). 

We now discuss the construction of the Euler profiles, $(u^i_E, v^i_E)$ for $i = 1,..,N_1$. Recall from \eqref{Eul:CR} that our equations are
\begin{align} \label{Eul:CR:2}
\p_x v^i_E = \p_Y u^i_E, \qquad \p_x u^i_E = - \p_Y v^i_E \text{ on } \mathcal{Q} = (0, \infty) \times (0, \infty), 
\end{align}
supplemented with boundary conditions 
\begin{align} \label{BC}
v^i_E(x, 0) = - v^{i-1}_p(x, 0), \qquad v^i_E(x, \infty) = 0.
\end{align}
For these profiles, we have the following estimates
\begin{proposition} \label{prop:Euler} Solutions to \eqref{Eul:CR:2} - \eqref{BC} satisfy the following estimates 
\begin{align} \label{water:7}
&\| (Y \p_Y)^l \p_x^k \p_Y^j v^{i}_E \|_{L^\infty_Y} \lesssim \langle x \rangle^{- \frac 1 2 - k - j}, \\ \label{water:8}
&\|  (Y \p_Y)^l \p_x^k \p_Y^j u^{i}_E \|_{L^\infty_Y} \lesssim \langle x \rangle^{- \frac 1 2 - k - j}.
\end{align}
\end{proposition}
\begin{proof} This is a minor adaptation of Lemma 3.4 and Proposition 3.6 in \cite{Iyer2b}.
\end{proof}

\subsection{Prandtl Layers}

\subsubsection{Leading Order Prandtl}

To find the leading order boundary layer, we set the term in the first parenthesis of \eqref{exp:fully:done:1} equal to zero, as well as the pressure term $P^0_{py}$ from \eqref{exp:fully:done:2}, which gives 
\begin{align} \label{BL:0}
&\bar{u}^0_p \p_x \bar{u}^0_p + \bar{v} \p_y \bar{u}^0_p - \p_y^{2} \bar{u}^0_p + P^{i}_{px} = 0, \qquad P^i_{py} = 0, \qquad \p_x \bar{u}^0_p + \p_y \bar{v}^0_p = 0,  
\end{align}
which are supplemented with the boundary conditions
\begin{align} \label{BL:1}
&\bar{u}^0_p|_{x = 0} = \bar{U}^0_p(y), \qquad \bar{u}^0_p|_{y = 0} = 0, \qquad \bar{u}^0_p|_{y = \infty} = 1, \qquad \bar{v}^0_p = - \int_0^y \p_x \bar{u}^0_p. 
\end{align}

Recall now the Blasius profiles, defined in \eqref{Blasius:1} - \eqref{Blasius:3}. Recall also that, without loss of generality, we set $x_0 = 1$. We now record the following quantitative estimates on the Blasius solution, which are well-known and follow immediately from the self-similarity of the profile:
\begin{lemma} For any $k, j, M \ge 0$, 
\begin{align} \label{water:1}
&\| \langle z \rangle^M \p_x^k \p_y^j (\bar{u}_\ast - 1) \|_{L^\infty_y} \le  C_{M, k, j} \langle x \rangle^{- k - \frac j 2}, \\ \label{v:blasius}
&\|  \langle z \rangle^M \p_x^k \p_y^j (\bar{v}_\ast - \bar{v}_\ast(x, \infty)) \|_{L^\infty_y} \le C_{M, k, j} \langle x \rangle^{ - \frac 1 2- k - \frac j 2}, \\ \label{v:blasius:2}
&\| \p_x^k \p_y^j \bar{v}_\ast \|_{L^\infty_y} \le C_{k,j} \langle x \rangle^{- \frac 1 2 -  k - \frac j 2}.
\end{align}
\end{lemma}

We also have the following properties of the Blasius profile, which are well known and which will be used in our analysis.
\begin{lemma} For $[\bar{u}_\ast, \bar{v}_\ast]$ defined in \eqref{Blasius:1}, the following estimates are valid
\begin{align} \label{Blas:prop:1}
&|\p_y \bar{u}_\ast(x, 0)| \gtrsim \langle x \rangle^{- \frac 1 2}, \\  \label{Blas:prop:2}
&\p_{yy}\bar{u}_{\ast} \le 0. 
\end{align}
\end{lemma}

\subsubsection{$i = 1,...,N_1$ Prandtl}
We collect the $\eps^{\frac i 2}$ order boundary layer contribution in the expansions \eqref{UUx} and \eqref{VVy}, which gives the linearized Prandtl system. We define for $i = 1, ..., N_1$, 
\begin{align} \label{def:PPi}
P^{(i)}_{p} := - \eps \int_y^\infty \Big( \mathcal{H}_p^{(i)} + \mathcal{J}_p^{(i)} - \eps^{\frac{i-1}{2}} \Delta_\eps v^{i-1}_p \Big) \ud y'
\end{align}
We subsequently set for $i = 1,..., N_1$, 
\begin{align} \label{Fpi}
F^{(i)}_p := - \eps^{- \frac i 2} \Big( \mathcal{F}_p^{(i)} + \mathcal{G}_p^{(i)}  +  \p_x P^{(i)}_P - \eps \eps^{\frac{i-1}{2}} u^{i-1}_{pxx}  \Big),
\end{align}
and correspondingly the equations
\begin{align} \label{rustic:1}
&\bar{u}^0_p \p_x u^i_p + \bar{u}^0_{px} u^i_p + \bar{v}^0_p \p_y u^i_p + \bar{v}^i_p \p_y \bar{u}^0_p - \p_y^2 u^i_p = F^{(i)}_p, \\ \label{rustic:2}
&v^i_p =  \int_y^\infty \p_x u^i_p, \qquad \bar{v}^i_p = v^i_p - v^i_p|_{y = 0}.
\end{align}
These equations are supplemented with the boundary conditions 
\begin{align} \label{rustic:3}
u^i_p|_{y = 0} = - u^i_E|_{y = 0}, \qquad u^i_p|_{y \rightarrow \infty} = 0, \qquad u^i_p|_{x = 0} = U^i_p(y). 
\end{align}

We now define our notion of parabolic compatibility condition. The ``zeroeth order parabolic compatibility condition" that we demand is simply that the initial data matches the boundary data at $(0, 0)$, that is 
\begin{align} \label{compat:0}
U^i_p(0) = - u^i_E(0, 0). 
\end{align}

To now motivate the definition of higher order compatibility, suppose we seek the initial datum for $\p_x u^i_p|_{x = 0}$. We evaluate the equation \eqref{rustic:1} at $x = 0$ to obtain 
\begin{align}
\bar{u}^0_p \p_x u^i_p|_{x= 0} + \bar{u}^0_{py} \bar{v}^i_p|_{x = 0} = F^{(i)}_p - (\bar{u}^0_{px} U^i_p + \bar{v}^0_p \p_y U^i_p- \p_y^2 U^i_p).
\end{align}
We now rewrite the left-hand side by using the divergence-free condition as 
\begin{align}
-|\bar{u}^0_p|^2 \p_y (\frac{\bar{v}^i_p(0, y)}{\bar{u}^0_p}) =  F^{(i)}_p - (\bar{u}^0_{px} U^i_p + \bar{v}^0_p \p_y U^i_p- \p_y^2 U^i_p),
\end{align}
which upon solving for $\bar{v}^i_p(0, y)$ yields 
\begin{align}
\bar{v}^i_p(0, y) = - \bar{u}^0_p \frac{u^i_{px}(0, 0)}{\bar{u}^0_{py}(0, 0)} + \bar{u}^0_p \int_0^y \frac{F^{(i)}_p - (\bar{u}^0_{px} U^i_p + \bar{v}^0_p \p_y U^i_p- \p_y^2 U^i_p)}{|\bar{u}^0_p|^2},
\end{align}
and taking one $\p_y$ then yields 
\begin{align} \n
u^i_{px}(0, y) =& \frac{ \bar{u}^0_{py}(0,y)}{\bar{u}^0_{py}(0, 0)} u^i_{px}(0, 0) - \frac{F^{(i)}_p(0, y) - (\bar{u}^0_{px} U^i_p + \bar{v}^0_p \p_y U^i_p- \p_y^2 U^i_p)}{\bar{u}^0_p} \\ \n
&- \bar{u}^0_{py}(0, y) \int_0^y \frac{F^{(i)}_p(0, y) - (\bar{u}^0_{px} U^i_p + \bar{v}^0_p \p_y U^i_p- \p_y^2 U^i_p)}{|\bar{u}^0_p|^2}
\end{align}
\begin{align}
\n
\phantom{u^i_{px}(0, y)}
= &- \frac{ \bar{u}^0_{py}(0,y)}{\bar{u}^0_{py}(0, 0)} u^{i}_{Ex}(0, 0)- \frac{F^{(i)}_p(0, y) - (\bar{u}^0_{px} U^i_p + \bar{v}^0_p \p_y U^i_p- \p_y^2 U^i_p)}{\bar{u}^0_p} \\ \label{gconvoL}
&- \bar{u}^0_{py}(0, y) \int_0^y \frac{F^{(i)}_p(0, y) - (\bar{u}^0_{px} U^i_p + \bar{v}^0_p \p_y U^i_p- \p_y^2 U^i_p)}{|\bar{u}^0_p|^2}.
\end{align}
From the above, the compatibility condition we need to enforce is then 
\begin{align} \label{comp:2}
\Big(\frac{F^{(i)}_p(0, \cdot) - (\bar{u}^0_{px} U^i_p + \bar{v}^0_p \p_y U^i_p- \p_y^2 U^i_p)}{\bar{u}^0_p} - \bar{u}^0_{py}(0, \cdot) \int_0^y \frac{F^{(i)}_p(0, \cdot) - (\bar{u}^0_{px} U^i_p + \bar{v}^0_p \p_y U^i_p- \p_y^2 U^i_p)}{|\bar{u}^0_p|^2} \Big)_{y = 0} = 0, 
\end{align}
which is a condition on up to the second derivatives of $U^i_p(y)$ at $y = 0$. 

\begin{definition} \label{para:compat:def} $U^i_p(\cdot)$ satisfies the ``zeroeth order parabolic compatibility condition" if \eqref{compat:0} is valid, and the ``first order parabolic compatibility condition" if \eqref{comp:2} is valid. Higher order compatibility conditions are derived in the same manner, although we omit writing the precise formula for these higher order conditions.  
\end{definition}

\noindent \textbf{Acknowledgements:} S.I is grateful for the hospitality and inspiring work atmosphere at NYU Abu Dhabi, where this work was initiated. The work of S.I is partially supported by NSF grant DMS-1802940. The work of N.M. is supported by NSF grant DMS-1716466 and by Tamkeen under the NYU Abu Dhabi Research Institute grant
of the center SITE.  


\def\bibindent{3.5em}

\end{document}